\documentclass[a4paper]{article} 
\usepackage[left=1.0cm,right=1.0cm,top=2.9cm]{geometry}
\usepackage{authblk}
\usepackage{CJKutf8}
\usepackage{yfonts}
\usepackage{lipsum}
\usepackage[utf8]{inputenc}
\usepackage{dsfont}

\usepackage[T2A]{fontenc}  
\newcommand{\CYRbb}{\text{\char193}}
\newcommand{\CYRe}{\text{\char196}}

\newcommand{\CYRk}{\text{\char234}}

\newcommand{\CYRz}{\text{\char200}}

\newcommand{\CYRu}{\text{\char244}}

\newcommand{\CYRb}{\text{\char203}}
\newcommand{\CYRPhi}{\text{\char212}}

\usepackage[monochrome]{xcolor} 

\makeatletter
\def\color#1{\csname color@black\endcsname}
\makeatother

\title{Velocity surface disorder of large deviation rates for high-dimensional random walk in strongly mixing environment}

\author{Jiaming Chen\thanks{chen.jiaming@cims.nyu.edu}}
\affil{Courant Institute of Mathematical Sciences, New York University}
\date{\today}

\usepackage{graphicx}
\usepackage{amssymb}
\usepackage{amsmath}
\usepackage{amsthm}
\usepackage{tikz}
\usepackage{todonotes}
\usepackage{enumitem}
\usepackage{verbatim}

\usepackage{bbm}

\usepackage{centernot}
\usepackage{mathtools}
\usepackage{ stmaryrd }

\usepackage{hyperref}
\hypersetup{allcolors=black,
pdftitle={LDP at origin of RWRE},
pdfpagemode=FullScreen}
\numberwithin{equation}{section}
\usepackage{bbm}

\usepackage{braket}
\usepackage{amsmath,amsthm,amssymb}
\usepackage{physics}
\usepackage{mathtools}
\usepackage{bm}
\usepackage{esvect}
\usepackage[english]{babel}
\usepackage{extarrows} 
\usepackage{mathrsfs}
\usepackage{esint}

\usepackage{setspace}
\usepackage{titlesec}
\titleformat{\subsection}[runin]
  {\normalfont\large\bfseries}{\thesubsection}{1em}{}

\usepackage{amsmath}

\numberwithin{equation}{section}
\usepackage{bbm}

\usepackage{braket}
\usepackage{amsmath,amsthm,amssymb}
\usepackage{physics}
\usepackage{mathtools}
\usepackage{bm}
\usepackage{esvect}
\usepackage[english]{babel}

\newtheorem{theorem}{Theorem}[section]
\newtheorem{lemma}[theorem]{Lemma}

\newtheorem{proposition}[theorem]{Proposition}

\theoremstyle{definition}
\newtheorem{definition}[theorem]{Definition}

\theoremstyle{remark}
\newtheorem{remark}[theorem]{Remark}

\usepackage[english]{babel}

\newcommand\normx[1]{\lVert#1\rVert}
\newcommand\normy[1]{\big\lVert#1\big\rVert}

\makeatletter
\renewenvironment{proof}[1][\proofname]{%
  \par\pushQED{\qed}\normalfont%
  \topsep6\p@\@plus6\p@\relax
  \trivlist\item[\hskip\labelsep\bfseries#1\@addpunct{.}]%
  \ignorespaces
}{%
  \popQED\endtrivlist\@endpefalse
}
\makeatother
\begin{document}
\maketitle

\begin{abstract}
    In this work, we establish the large deviation principles of random walk in strongly mixing environments. {\color{violet}We show that} the quenched and annealed rate functions {\color{violet}share} the same zero-set whose shape is either a singleton point or a line segment, with an illustrative example communicated and given by F. Rassoul-Agha. {\color{violet}Furthermore,} this paper reveals that {\color{violet}under sufficiently low} disorder, the two rate functions match on compact sets {\color{violet}of both} the interior and the boundary of the domain under finitely strongly mixing condition.
\end{abstract}



\section{Introduction}
    Consider a discrete-time random walk in mixing and uniformly elliptic random environments (RWRE). The RWRE {\color{violet}provides} a natural {\color{violet}setting} for the study of statistical mechanics in random medium. Emerging from a wide range of applications such as DNA chain replication \cite{Chernov} and crystal growth \cite{Temkin1,Temkin2} as well as {\color{violet}models of turbulent behavior} in fluids through a Lorentz gas description \cite{Chennn,Sinai2}, {\color{violet}RWRE} has attracted profound attention in recent decades from mathematicians and physicists.\par
    
    The one-dimensional RWRE was proposed by Solomon \cite{Solomon} and Sinai \cite{Sinai}, {\color{violet}and is by now well understood}. {\color{violet}Fundamental results on} transience and recurrence {\color{violet}were established by} Solomon \cite{Solomon} in i.i.d.\ environment and {\color{violet}later extended} to ergodic environment by Alili \cite{Alili}. The equivalence of invariant densities was shown by Kozlov \cite{Kozlov} and a subsequent proof can be found in Bolthausen/Sznitman \cite{Bolthausen/Sznitman}. {\color{violet}Further developments include} the quenched moderate deviations by Hong/Wang \cite{Hong/Wang}. {\color{violet}For a comprehensive account of other limit theorems, we refer to the survey by Zeitouni} \cite{Zeitouni} in $d=1$.\par
    
    Multi-dimensional RWRE {\color{violet}remains significantly more} challenging. In this work, assuming fairly general mixing conditions on the environment, we {\color{violet}derive exponential estimates for} rare events from both quenched and annealed laws, and discern {\color{violet}conditions under which} they are equal. Over the years there has been {\color{violet}advance} on the study of law of large numbers (LLN) by Berger \cite{Berger}, Kozlov\cite{Kozlov}, Zerner \cite{Zerner}, Rassoul-Agha \cite{Rassoul-Agha0} from the viewpoint of the particle and especially by Comets/Zeitouni \cite{Comets/Zeitouni} under cone-mixing conditions with extra ballistic and integrability requirements. The central limit theorem (CLT) has been successfully investigated by Berger/Zeitouni \cite{Berger/Zeitouni}, Kipnis/Varadhan \cite{Kipnis/Varadhan}, Lawler \cite{Lawler}, Papanicolaou/Varadhan \cite{Papanicolaou/Varadhan}, Rassoul-Agha/Seppäläinen \cite{Rassoul-Agha/Seppalainen2}, and especially by Guerra Aguilar \cite{Guerra Aguilar} under strong mixing conditions with transience requirements. {\color{violet}While progress has been made on the classical limit theorems, many aspects of the large-scale behavior are still not fully understood under general mixing structures.}\par

    {\color{violet}Various} other RWRE models have also been {\color{violet}investigated regarding their} invariance principle. {\color{violet}This includes studies of} random walk in balanced environment {\color{violet}(e.g.,} Lawler \cite{Lawler}, Guo/Zeitouni \cite{Guo/Zeitouni}, Berger/Deuschel \cite{Berger/Deuschel}{\color{violet}) and} random walk in isotropic environment {\color{violet}(e.g.,} Bricmont/Kupiainen \cite{Bricmont/Kupiainen}, Bolthausen/Zeitouni \cite{Bolthausen/Zeitouni} {\color{violet}). Furthermore, Baur \cite{Baur} explored} random environments in the form of perturbations of simple random walks with invariance under reflections and balance in one coordinate direction. {\color{violet}More recent studies look at} RWRE models where the equivalent ballisticity conditions fail to hold, see Drewitz/Ramírez \cite{Drewitz/Ramirez} {\color{violet}and} Guerra Aguilar/Ramírez \cite{Guerra Aguilar/Ramirez2}.\par

    Although the quenched and the annealed laws exhibit {\color{violet}identical} limiting behavior {\color{violet}at} the level of LLN, {\color{violet}their} asymptotics {\color{violet}may diverge significantly regarding} CLT or large deviation principles (LDP). {\color{violet}For} the RWRE {\color{violet}$(X_n)_{n\geq0}$} on $\mathbb{Z}^d$ with $d\geq1$, the LDP {\color{violet}characterizes} the asymptotics $\lim n^{-1}\log P_{0,\omega}(n^{-1}X_n\approx x)\simeq I_q(x)$ and $\lim n^{-1}\log P_{0}(n^{-1}X_n\approx x)\simeq I_a(x)$ where {\color{violet}$I_q$ and $I_a$ denote the quenched and annealed rate functions, respectively. While the quenched law $P_{0,\omega}$ and annealed law $P_{0}$ exhibit similar behaviors across a wide class of environments}, their {\color{violet}LDP} rate functions may still have different {\color{violet}values}. The goal of this work is to {\color{violet}establish such} LDPs {\color{violet}under mixing environments} and {\color{violet}to demonstrate} that {\color{violet}their rate functions} $I_q(\vdot)$ and $I_a(\vdot)$ {\color{violet}are equal on compact sets at sufficiently low} level of {\color{violet}environmental} disorder.\par
    
    The existence of quenched and annealed LDP was displayed by Rassoul-Agha \cite{Rassoul-Agha} {\color{violet}for} a mild class of mixing environments {\color{violet}originally} introduced by Dobrushin/Shlosman \cite{Dobrushin/Shlosman2} in the context of {\color{violet}high temperature} spin-glass systems. To reconstruct or generalize Rassoul-Agha's work, our first contribution is to {\color{violet}extend} these LDP {\color{violet}existence results} to {\color{violet}two} slightly {\color{violet}broader} classes of mixing environments, respectively called strongly mixing, i.e.~$\textbf{(SM)}_{C,g}$, {\color{violet}and} Guo's strongly mixing, i.e.~$\textbf{(SMG)}_{C,g}$. Along with simplified arguments, this part of our work {\color{violet}resonates with} Richard Feynman's phrase: “There is pleasure in recognising old things from a new viewpoint” \cite{Feynman}.\par
    
    To this recognition, {\color{violet}we address the intricate question of the discrepancy between the quenched and annealed rate functions under the} \textit{finitely} strongly mixing condition, i.e.~$\textbf{(SMF)}_{C,g}$. {\color{violet}The primary contribution of this paper is to show that, away from certain singular regions} [corners, origin, etc.], the quenched and annealed LDP rate functions {\color{violet}$I_q$ and $I_a$ coincide} at almost every point {\color{violet}of} the unit disk {\color{violet}whenever} the level of disorder is sufficiently small. {\color{violet}By employing the methods introduced} in Section \ref{sec: chapter 4} which {\color{violet}systematically disentangle} the {\color{violet}dependence in the} environment, {\color{violet}we provide a framework that may be applicable to} establish similar results for RWRE in time-correlated space-time models as well as time-correlated polymer models, see Yilmaz \cite{Yilmaz2} for an LDP on such model with i.i.d.\ space-time. See also \cite{Bazaes/Mukherjee/Ramirez/Sagliett0,Bazaes/Mukherjee/Ramirez/Sagliett} for pioneering breakthroughs of similar type in i.i.d.\ environments.

\textbf{Acknowledgement.} I am grateful to my PhD advisor Prof. Dr. Alejandro Ramírez at NYU-ECNU Institute of New York University for pointing out the problem and for many useful discussions. I acknowledge Firas Rassoul-Agha for discussions on the zero set of rate functions and for giving me an illustrative example. I also thank Maximilian Nitzschner for reading the preliminary version of my manuscript, and Vlad Margarint for comments on structure of Brownian motion and pointing out Dynkin's formula. I am especially grateful to the anonymous reviewer from Electronic Journal of Probability for the constructive comments on the improvement of the writing and for the insightful suggestions on the discussion of the equality between LDP rate functions at the lower-dimensional relative boundary of the velocity surface.

\section{Random walk in strongly  mixing random environment}
    {\color{red}Let $(\omega_x)_{x\in\mathbb{Z}^d}$ denote a} random environment on $\mathbb{Z}^d${\color{red}, where each $\omega_x = (\omega(x, e))_{e\in\mathbb{V}}$ is a probability vector which satisfies} $\sum_{e\in\mathbb{V}}\omega(x,e)=1$, $\forall~x\in\mathbb{Z}^d$. Here we use $\mathbb{V}\coloneqq\{e\in\mathbb{Z}^d:\,\abs{e}_1=1\}$. {\color{red}Let} $\mathscr{F}_\Omega$ {\color{red}be} the canonical product $\sigma$-algebra on $\mathbb{Z}^d$, and {\color{red}and let} $\mathbb{P}$ {\color{red}(with expectation} $\mathbb{E}\coloneqq E_{\mathbb{P}}$) {\color{red}denote the environmental law on $(\Omega, \mathscr{F}_\Omega)$, which} will be specified later under explicit mixing conditions. Throughout this paper, the environment $(\omega_x)_{x\in\mathbb{Z}^d}$ is assumed {\color{violet}to be} \textbf{uniformly elliptic}, i.e.~{\color{red}there exists an absolute constant} $0<\kappa<1/(2d)$ {\color{red}such that} $\mathbb{P}(\omega(x,e)\geq\kappa)=1$ with any $x\in\mathbb{Z}^d$, $e\in\mathbb{V}$.\par 
    Given any environmental law $(\mathbb{P},\mathscr{F}_\Omega)$, we first define the level of disorder of the field.
    \begin{definition}\label{def: disorder}
        Given any ergodic environment $\mathbb{P}$, the disorder $\text{dis}(\mathbb{P})\coloneqq\inf\{\epsilon>0:\,\frac{\omega(x,e)}{\mathbb{E}[\omega(x,e)]}\in[1-\epsilon,1+\epsilon],\,\forall~e\in\mathbb{V}\;\text{and}\;x\in\mathbb{Z}^d,\;\mathbb{P}\text{-a.s.}\}$ is defined as the level of {\color{red} environmental fluctuations} deviated from {\color{red}the} mean of the environment.
    \end{definition}
    To {\color{red}explicitly} {\color{violet}describe} the mixing conditions on $(\mathbb{P},\mathscr{F}_\Omega)$, we first introduce some new notation. {\color{red}Let} $d_1(A,B)$ stands for the $\ell^1$-distance between sets $A,B\subseteq\mathbb{Z}^d$. For a {\color{red}fixed} constant $r>0$, the law $(\mathbb{P},\mathscr{F}_\Omega)$ {\color{red}is assumed $r$-\textbf{Markovian}} throughout this paper, i.e.
    \[
        \mathbb{P}\big((\omega_x)_{x\in V}\in\vdot\big|\mathscr{F}_{V^c}\big) = \mathbb{P}\big((\omega_x)_{x\in V}\in\vdot\big|\mathscr{F}_{\partial^r V}\big),\qquad\forall~\text{finite}~V~{\color{violet}\subseteq}~\mathbb{Z}^d,\qquad\mathbb{P}\text{-a.s.,}
    \]
    where $\partial^rV\coloneqq\{z\in\mathbb{Z}^d\backslash V:\,\exists~y\in V~\text{s.t.}~\abs{z-y}_1\leq r\}$ and $\mathscr{F}_\Lambda\coloneqq\sigma(\omega_x,\,x\in\Lambda)$ for any $\Lambda\subseteq\mathbb{Z}^d$.\par

    The simplest scenario {\color{red}arises} when there is no correlation among {\color{red}the vertices} $(\omega_x)_{x\in\mathbb{Z}^d}$: {\color{red}Specifically,} the condition $\textbf{(IID)}$ {\color{red}holds} if $(\omega_x)_{x\in\mathbb{Z}^d}$ are i.i.d.\ under $\mathbb{P}$. {\color{red}The situation becomes} significantly more involved when the {\color{red}lattice} correlation emerges. Given absolute positive constants $C$ and $g$, the environment $(\mathbb{P},\mathscr{F}_\Omega)$ satisfies the strong mixing condition $\textbf{(SM)}_{C,g}$ if 
    \begin{equation}        
    \label{eqn: sm}
        \frac{d\mathbb{P}((\omega_x)_{x\in\Delta}\in\vdot|\eta)}{d\mathbb{P}((\omega_x)_{x\in\Delta}\in\vdot|\eta^\prime)} \leq \exp\bigg(C\sum_{x\in\partial^r\Delta,y\in\partial^rA}e^{-g\abs{x-y}_1}\bigg),\qquad\forall~\text{finite}~\Delta\subseteq V\subseteq\mathbb{Z}^d\quad\text{with}~d_1(\Delta,V^c)\geq r
    \end{equation}
    and $A\subseteq V^c$, {\color{red}holds} simultaneously for all pairs of configurations $\eta,\eta^\prime\in\Omega$ which agree on $V^c\backslash A$, $\mathbb{P}$-a.s. An {\color{violet}alternative,} slightly different mixing condition {\color{red}introduced} by X. Guo \cite{Guo}, is {\color{red}defined} as follows. Given $C$ and $g$ as above, $(\mathbb{P},\mathscr{F}_\Omega)$ satisfies the Guo's strong mixing condition $\textbf{(SMG)}_{C,g}$ if
    \begin{equation}        
    \label{eqn: smg}    
        \frac{d\mathbb{P}((\omega_x)_{x\in\Delta}\in\vdot|\eta)}{d\mathbb{P}((\omega_x)_{x\in\Delta}\in\vdot|\eta^\prime)} \leq \exp\bigg(C\sum_{x\in\Delta,y\in A}e^{-g\abs{x-y}_1}\bigg),\qquad\forall~\text{finite}~\Delta\subseteq V\subseteq\mathbb{Z}^d\quad\text{with}~d_1(\Delta,V^c)\geq r
    \end{equation}
    and $A\subseteq V^c$, {\color{red}holds} simultaneously for all pairs of configurations $\eta,\eta^\prime\in\Omega$ which agree on $V^c\backslash A$, $\mathbb{P}$-a.s. A more {\color{red}localized} mixing condition {\color{red}captures the scenario where correlation decays over distance and vanish beyond a finite range. Specifically,} we say $(\mathbb{P},\mathscr{F}_\Omega)$ satisfies the {finitely} strong mixing condition $\textbf{(SMF)}_{C,g}$ if each pair $\omega_x$ and $\omega_y$ are independent \textit{only if} $\abs{x-y}\geq L_0$ for some absolute $L_0=L_0(C,g)>0$.\par

    It is clear that under $\textbf{(SMF)}_{C,g}$, the environment $\mathbb{P}$ {\color{red}exhibits} finite{\color{red}-range} dependence, and {\color{red}is} {\color{violet}thus} a more specific condition than either $\textbf{(SM)}_{C,g}$ or $\textbf{(SMG)}_{C,g}$. Indeed, under $\textbf{(SMF)}_{C,g}$, {\color{red}the variables} $\omega_x$ and $\omega_y$ {\color{red}can} be dependent only when $\abs{x-y}_1\geq L_0$ for some $L_0=L_0(C,g)$. Hence {\color{red}the Radon--Nikodým derivative satisfies} $\frac{d\mathbb{P}((\omega_x)_{x\in\Delta}\in\vdot|\eta)}{d\mathbb{P}((\omega_x)_{x\in\Delta}\in\vdot|\eta^\prime)} = 1$ whenever $d_1(\Delta,A)\geq L_0${\color{red}, holding} simultaneously for any pair of configurations $\eta,\eta^\prime\in\Omega$ which {\color{violet}coincide} except for on $A\subseteq\mathbb{Z}^d$. When $d_1(\Delta,A)<L_0$, {\color{red}the shift-invariance of} $\mathbb{P}$ {\color{red}together with the finiteness of} the box $\{y\in\mathbb{Z}^d:\,\abs{x-y}_1<L_0\}$ {\color{red}implies that both} (\ref{eqn: sm}) and (\ref{eqn: smg}) {\color{red}hold automatically.} Whence the condition $\textbf{(SMF)}_{C,g}$ {\color{red}implies both} $\textbf{(SM)}_{C,g}$ and $\textbf{(SMG)}_{C,g}$. On the other hand, {\color{red}while} $\textbf{(SMG)}_{C,g}$ {\color{red}appears to be more general} asymptotically, {\color{red}it} is not implied by $\textbf{(SM)}_{C,g}$. Nonetheless, {\color{red}$\textbf{(SMG)}_{C,g}$ is technically more demanding to work with in asymptotic analyses.}\par

    {\color{red}It should also be noted that condition} $\textbf{(SMF)}_{C,g}$ is {\color{red}elsewhere referred to as} \textit{"finite-range dependence"} in {\color{violet}many contexts}. Having {\color{red}specified the underlying environment, we now introduce} the random walk $(X_n)_{n\geq0}$ {\color{red}traveling} on $\mathbb{Z}^d$ with jump probabilities given by $\omega$, i.e.~ for each $x\in\mathbb{Z}^d$ the {\color{red}quenched} law $P_{x,\omega}$ {\color{violet}(with expectation} $E_{x,\omega}\coloneqq E_{P_{x,\omega}}$) of this random walk starting from $x$ is prescribed by $P_{x,\omega}(X_0=x)=1$ and $P_{x,\omega}(X_{n+1}=y+e|X_n=y)=\omega(y,e)$ for all $y\in\mathbb{Z}^d$, $e\in\mathbb{V}$ and $n\in\mathbb{N}$. Averaging $P_{x,\omega}$ over {\color{red}the environmental measure $\mathbb{P}$,} we {\color{red}define the annealed law of the random walk in random environment (RWRE) on $\mathbb{Z}^d$ by} the semi-direct product
    \[
        P_x( B)\coloneqq\int_{\Omega} P_{x,\omega}(B)\,d\mathbb{P(\omega)},\qquad\forall~B\in\mathscr{B}((\mathbb{Z}^d)^{\mathbb{N}}).
    \]
    Here $\mathscr{B}((\mathbb{Z}^d)^{\mathbb{N}})$ stands for the Borel $\sigma$-algebra on {\color{red}the path space} $(\mathbb{Z}^d)^{\mathbb{N}}$.\par

    {\color{red}Since the random walk moves between} nearest neighbors, {\color{red}the empirical} velocity {\color{red}satisfies} $|X_N/N|_1\leq1$ for all $N\geq1$. And thus we {\color{violet}refer to} $\mathbb{D}\coloneqq\{x\in\mathbb{R}^d:\,\abs{x}_1\leq1\}$ as the velocity surface and denote {\color{red}its $(d-2)$-dimensional boundary set by} $\partial\mathbb{D}_{d-2}\coloneqq\{x\in\partial\mathbb{D}:\,x_j=0\;\,\text{for some}~1\leq j\leq d\}$. In \cite{Varadhan1}, {\color{red}Varadhan established that under} $\textbf{(IID)}$, the RWRE {\color{red}satisfies} the following quenched and annealed large deviation principles on $\mathbb{Z}^d$ for any $d\geq1$:
    \begin{itemize}
        \item There exists {\color{red}a} rate function $I_a:\mathbb{D}\to[0,\infty)$ which is convex and continuous on $\text{int}(\mathbb{D})$ so that for any Borel $G\subseteq\mathbb{R}^d$,
        \[
            -\inf_{x\in G^{\circ}} I_a(x)\leq \varliminf_{N\to\infty} \frac{1}{N}\log P_0\bigg(\frac{X_N}{N}\in {\color{violet}G}\bigg)\leq\varlimsup_{N\to\infty} \frac{1}{N}\log P_0\bigg(\frac{X_N}{N}\in {\color{violet}G}\bigg) \leq -\inf_{x\in\overline{G}} I_a(x).
        \]
        \item There exists {\color{red}a} deterministic rate function $I_{\color{violet}q}:\mathbb{D}\to[0,\infty)$ which is convex and continuous on $\text{int}(\mathbb{D})$ such that $\mathbb{P}$-a.s. for any Borel $G\subseteq\mathbb{R}^d$,
        \[
            -\inf_{x\in G^{\circ}} I_q(x)\leq \varliminf_{N\to\infty} \frac{1}{N}\log P_{0,\omega}\bigg(\frac{X_N}{N}\in {\color{violet}G}\bigg)\leq\varlimsup_{N\to\infty} \frac{1}{N}\log P_{0,\omega}\bigg(\frac{X_N}{N}\in {\color{violet}G}\bigg) \leq -\inf_{x\in\overline{G}} I_q(x).
        \]
    \end{itemize}\par

    Recent advances \cite{Bazaes/Mukherjee/Ramirez/Sagliett,Bazaes/Mukherjee/Ramirez/Sagliett0} have revealed how the level of disorder {\color{red}dictates} the equality between $I_a$ and $I_q$ on $\mathbb{D}$. {\color{red}Specifically,} when $\mathbb{P}$ satisfies $\textbf{(IID)}$ {\color{red}there exist disorder} thresholds below which the two rate functions: 
    \begin{itemize}
        \item For any {\color{red}dimension} $d\geq4$ and compact $\mathcal{K}\subseteq\partial\mathbb{D}\backslash\partial\mathbb{D}_{d-2}$, there exists $\epsilon_1=\epsilon_1(\mathcal{K})>0$ such that $I_a(x)=I_q(x)$ for any $x\in\mathcal{K}$, whenever $\text{dis}(\mathbb{P})<\epsilon_1$.
        \item For any {\color{red}dimension} $d\geq4$ and compact $\mathcal{K}\subseteq\text{int}(\mathbb{D})\backslash\{0\}$, there exists $\epsilon_2=\epsilon_2(\mathcal{K})>0$ such that $I_a(x)=I_q(x)$ for any $x\in\mathcal{K}$, whenever $\text{dis}(\mathbb{P})<\epsilon_2$.
    \end{itemize}
    Our paper {\color{red}extends these insights} to {\color{red}dependent} random {\color{red}field in the same context of} large deviations of {\color{violet}the} limiting velocities. {\color{red}Our} first {\color{red}main} result establishes the general LDP {\color{red}under} either $\textbf{(SM)}_{C,g}$ or $\textbf{(SMG)}_{C,g}$, {\color{red}and thereby generalizing the results of} \cite{Rassoul-Agha} to {\color{red}a broader class of correlated} environments.

    \begin{theorem}\label{thm: existence of LDP}
        \normalfont
        For any $d\geq1$ and $\kappa>0$, when $\mathbb{P}$ {\color{red}satisfies} either $\textbf{(SM)}_{C,g}$ or $\textbf{(SMG)}_{C,g}$, \begin{itemize}
        \item There exists {\color{red}a} rate function $I_a:\mathbb{D}\to[0,\infty)$ which is convex and continuous on $\text{int}(\mathbb{D})$ {\color{red}such} that for any Borel $G\subseteq\mathbb{R}^d$,
        \[
            -\inf_{x\in G^{\circ}} I_a(x)\leq \varliminf_{N\to\infty} \frac{1}{N}\log P_0\bigg(\frac{X_N}{N}\in {\color{violet}G}\bigg)\leq\varlimsup_{N\to\infty} \frac{1}{N}\log P_0\bigg(\frac{X_N}{N}\in {\color{violet}G}\bigg) \leq -\inf_{x\in\overline{G}} I_a(x).
        \]
        \item There exists {\color{red}a} deterministic rate function $I_{\color{violet}q}:\mathbb{D}\to[0,\infty)$ which is convex and continuous on $\text{int}(\mathbb{D})$ such that $\mathbb{P}$-a.s. for any Borel $G\subseteq\mathbb{R}^d$,
        \[
            -\inf_{x\in G^{\circ}} I_q(x)\leq \varliminf_{N\to\infty} \frac{1}{N}\log P_{0,\omega}\bigg(\frac{X_N}{N}\in {\color{violet}G}\bigg)\leq\varlimsup_{N\to\infty} \frac{1}{N}\log P_{0,\omega}\bigg(\frac{X_N}{N}\in {\color{violet}G}\bigg) \leq -\inf_{x\in\overline{G}} I_q(x).
        \]
    \end{itemize}
    \end{theorem}
    {\color{red}Within the} LDP framework, the zero set{\color{red}s} of $I_a$ and $I_q$ {\color{red}characterize} the limiting velocities that {\color{red}exhibit sub-exponential decay} under the annealed and quenched {\color{red}laws, respectively. Consequently,} the geometry of the{\color{red}se} zero set{\color{red}s} of $I_a$ and $I_q$ is {\color{violet}of} {\color{violet}central} {\color{red}probabilistic} interest.  {\color{red}In this direction,} we establish the following result concerning {\color{red}their geometric properties.}

    \begin{theorem}\label{thm: shape of zero set}
        \normalfont
        For any $d\geq1$ and $\kappa>0$, when $\mathbb{P}$ {\color{red}satisfies} either $\textbf{(SM)}_{C,g}$ or $\textbf{(SMG)}_{C,g}$, the rate functions $I_a$ and $I_q$ {\color{red}share} the same zero set, i.e.~{\color{red}$\{x\in\mathbb{D}: I_q(x)=0\} = \{x\in\mathbb{D}: I_a(x)=0\}$.}  Furthermore, {\color{red}this common} zero set is either a single-point or a {\color{red}closed} line segment {\color{red}containing} the origin.
    \end{theorem}
    When {\color{violet}the} zero set reduces to a singleton in $\mathbb{D}$, {\color{red}this point} {\color{violet}is} not necessarily the origin. Consider an RWRE satisfying $\mathbb{P}$-a.s., $E_{0,\omega}[\langle X_1,e\rangle]\geq\delta$,  $\exists~e\in\mathbb{V}$ and $\delta>0$. This {\color{red}condition implies that} the {\color{red}directional limiting velocity satisfies} $\lim_{N\to\infty}\langle e, X_N/N\rangle\geq\delta$ {\color{red}forcing both} {\color{violet}rate functions} $I_a$ and $I_q$ {\color{red}to be strictly positive} at the origin. {\color{red}Instead, they vanish at a unique} nonzero limiting velocity. This illustrative {\color{red}construction} {\color{violet}was} {\color{red}suggested} by Firas Rassoul-Agha in private conversation.\par

    Beyond the geometry of the zero sets, {\color{red}a question} of fundamental {\color{red}importance is} the {\color{red}aggrement} between the annealed and quenched rate functions on the velocity surface $\mathbb{D}$, {\color{red}particularly} whether they {\color{violet}coincide under} sufficiently low disorder. The following result characterizes {\color{red}this} striking phenomen{\color{violet}on} at the boundary of the velocity surface $\mathbb{D}$ when $\mathbb{P}$ verifies the finite mixing condition. Here we emphasize that our Theorems \ref{thm: equality of LDP at boundary}--\ref{thm: equality of LDP in interior} hold for $d\geq4$ whereas Theorems \ref{thm: existence of LDP},\ref{thm: shape of zero set} and \ref{thm: non-equality of rate functions} {\color{violet}are valid} for all $d\geq1$. {\color{violet}It should be remarked} that the statement of Theorem \ref{thm: equality of LDP at boundary} {\color{violet}fails} when $d=1$, since $\partial\mathbb{D}$ coincide{\color{violet}s} with $\mathbb{V}$ {\color{red}in dimension} $d=1$ {\color{red}as shown by} Theorem \ref{thm: non-equality of rate functions}.

    \begin{theorem}\label{thm: equality of LDP at boundary}
        \normalfont
        For any $d\geq4$ and compact $\mathcal{K}\subseteq\partial\mathbb{D}\backslash\partial\mathbb{D}_{d-2}$, when $\mathbb{P}$ {\color{red}satisfies} $\textbf{(SMF)}_{C,g}$ there exists $\epsilon_1=\epsilon_1(\mathcal{K})>0$ such that $I_a(x)=I_q(x)$ for any $x\in\mathcal{K}$, whenever $\text{dis}(\mathbb{P})<\epsilon_1$.
    \end{theorem}
    It is worth noting that {\color{red}under $\textbf{(SMF)}_{C,g}$,} the conditioned RWRE which {\color{red}is forced to have a} limiting velocity {\color{red}in} $\partial\mathbb{D}\backslash\partial\mathbb{D}_{d-2}$ {\color{violet}behaves like} a deterministic space-time walk where the direction {\color{violet}perpendicular to the boundary} {\color{red}acts as the} time {\color{red}coordinate}. In fact, {\color{red}an inspection of} the proof of Lemma \ref{lem: square integrable of the martingale H} and the subsequent Remark \ref{rmk: remark subsequent to lemma 4.2} {\color{red}reveals that} the same {\color{red}conclusion extends to compact subsets of} lower-dimensional boundary {\color{red}strata} $\partial\mathbb{D}_{d-k}\backslash\partial\mathbb{D}_{d-k-1}$. Here {\color{red}f}or the $d$-dimensional velocity surface $\mathbb{D}$, we {\color{red}denote by} $\partial\mathbb{D}$ its $(d-1)$-dimensional boundary, and {\color{red}by} $\partial\mathbb{D}_{d-2}$ the relative boundary of $\partial\mathbb{D}$. Likewise, $\partial\mathbb{D}_{d-k}$ denote{\color{red}s} the $(d-k)$-dimensional relative boundary of $\partial\mathbb{D}_{d-(k-1)}$ for each ${\color{red}1\leq k}\leq d$, {\color{red}where} we identify $\partial\mathbb{D}_{0}$ with $\mathbb{V}$. {\color{red}With this notation,} the {\color{red}assertion} of Theorem \ref{thm: equality of LDP at boundary}  can be rephrased as {\color{red}follows}: $\exists~\epsilon_1(\mathcal{K})>0$ s.t. $I_a|_{\mathcal{K}}\equiv I_q|_{\mathcal{K}}$ whenever $\text{dis}(\mathbb{P})<\epsilon_1$ for any compact $\mathcal{K}\subseteq\partial\mathbb{D}\backslash\partial\mathbb{D}_{d-k}$ with $k=2$. The {\color{red}scenario} $k>2$ should be treated differently. {\color{red}As demonstrated in} the proof of our Theorem \ref{thm: non-equality of rate functions}, $I_a$ and $I_q$ already {\color{red}strictly} disagree {\color{red}on} $\mathbb{V}=\partial\mathbb{D}_{0}$. We {\color{red}conjecture} that the quenched and annealed LDP rates {\color{violet}remain distinct} on $\partial\mathbb{D}_{d-k}$ when $k=d-1,d-2$ {\color{red}whenever} $\text{dis}(\mathbb{P})>0$. This conjecture is {\color{red}motivated by} the {\color{red}classical} recurrence-transience criteria for simple random walks in lower ($d\leq2$) and higher ($d\geq3$) dimensions. {\color{red}A} rigorous {\color{red}analysis of these remaining boundary} $k=d-1,d-2$ cases {\color{red}lies} {\color{violet}beyond the scope of} {\color{red}the present} work and will be {\color{red}addressed} in future research.\par

    \begin{definition}
        Given {\color{red}an} ergodic environment $\mathbb{P}$, the {\color{red}directional} imbalance {\color{red}along an axial orientation $s \in \{-1,1\}^d$ is defined by} $\text{imb}_s(\mathbb{P})\coloneqq\inf\{\epsilon>0:\,\frac{\sum_{j=1}^d\omega(x,s_je_j)}{\sum_{j=1}^d\mathbb{E}[\omega(x,s_je_j)]}\in[1-\epsilon,1+\epsilon],\,\forall~e\in\mathbb{V}\;\text{and}\;x\in\mathbb{Z}^d,\,\mathbb{P}\text{-a.s.}\}$ {\color{red}which quantifies the level of environmental fluctuations restricted to the directional sector specified by $s$.} {\color{red}Note} that the numerator in th{\color{red}is} definition {\color{red}equals} $P_{0,\omega}(X_1\in\partial\mathbb{D}(s))$, and whence the denominator {\color{red}simplifies to} $P_{0}(X_1\in\partial\mathbb{D}(s))$.
    \end{definition}
    Our next result {\color{violet}establishes} the equality of $I_a$ and $I_q$ under {\color{red}considerably} weaker {\color{red}assumptions.} {\color{red}Specifically, these} rate functions {\color{violet}match} on nontrivial open subsets of $\partial\mathbb{D}$, provided {\color{red}the environment} $\mathbb{P}$ exhibits sufficiently {\color{violet}small} {\color{red}directional} imbalance. Intuitively, {\color{red}$\text{imb}_s(\mathbb{P})$ measures the local} {\color{violet}asymmetry} {\color{red}of environmental fluctuations aligned with the} axial direction $s$. We now {\color{red}formulate this} main result {\color{red}precisely} below.

    \begin{theorem}\label{thm: imbalance of LDP at boundary}
        \normalfont
        For any $d\geq4$ and $\kappa>0$, when $\mathbb{P}$ {\color{red}satisfies} $\textbf{(SMF)}_{C,g}$ there exist {\color{red}a} nontrivial open set $\mathcal{O}\subseteq\partial\mathbb{D}(s) \backslash \partial\mathbb{D}_{d-2}$ and {\color{red}a constant} $\epsilon^\dagger=\epsilon^\dagger(\kappa)>0$ such that $I_a(x)=I_q(x)$ for {\color{red}all} $x\in\mathcal{O}$, whenever $\text{imb}_s(\mathbb{P})<\epsilon^\dagger$, for any $s\in\{-1,1\}^d$.
    \end{theorem}
    {\color{red}For} each {\color{red}orientation} $s\in\{\pm1\}^d$, we {\color{red}denote by} $\partial\mathbb{D}(s)$ the unique facet of $\partial\mathbb{D}$ {\color{violet}to} {\color{red}whose outward normal vector is} $s$. Under the stronger condition {\color{red}controlling} $\text{dis}(\mathbb{P})$, {\color{red}one can establish the} equality $I_a$ and $I_q$ {\color{red}across} {\color{violet}arbitrary} {\color{red}compact subsets of} $\partial\mathbb{D}\backslash\partial\mathbb{D}_{d-2}$. In contrast, $\text{imb}_s(\mathbb{P})$ {\color{red}acts as a directional, averaged quantifier of disorder along} $s$. {\color{red}By replacing $\text{dis}(\mathbb{P})$ with this milder directional measure, \ref{thm: imbalance of LDP at boundary} demonstrates that controlling environmental fluctuations along a single orientation is already sufficient to guarantee the equality of} $I_a$ {\color{red}and} $I_q$ {\color{red}locally on open subregions of the boundary.}\par

    While the equality of $I_a$ and $I_q$ on the boundary $\partial\mathbb{D}$ {\color{violet}illuminates} the {\color{red}structural} role of disorder, it is equally important to {\color{red}determine} whether {\color{red}this equality extends} into the interior of the velocity surface. In particular, {\color{red}one} {\color{violet}naturally explores} conditions {\color{red}ensuring that} $I_a$ and $I_q$ coincide {\color{red}on subregions of} $\text{int}(\mathbb{D})$, {\color{violet}as} {\color{red}asserted} in the {\color{red}following result.}
    \begin{theorem}\label{thm: equality of LDP in interior}
        \normalfont
        For any $d\geq4$ and compact $\mathcal{K}\subseteq\text{int}(\mathbb{D})\backslash\{0\}$, when $\mathbb{P}$ {\color{red}satisfies} $\textbf{(SMF)}_{C,g}$ there exists $\epsilon_2=\epsilon_2(\mathcal{K})>0$ such that $I_a(x)=I_q(x)$ for {\color{red}all} $x\in\mathcal{K}$, whenever $\text{dis}(\mathbb{P})<\epsilon_2$.
    \end{theorem}
    {\color{violet}While} controlling the disorder enforces $I_a=I_q$ uniformly on compact subsets of $\text{int}(\mathbb{D})$ {\color{red}away from the origin}, {\color{red}these two rate functions} can never coincide simultaneously on the entire velocity surface {\color{violet}$\mathbb{D}$ as long as} the environment $\mathbb{P}$ {\color{violet}is} genuine{\color{violet}ly} random. This fundamental {\color{red}strict inequality} persists even under {\color{violet}the} strongly mixing condition $\textbf{(SMF)}_{C,g}$, {\color{red}extending} a classical result of Yilmaz \cite[Proposition 4]{Yilmaz} {\color{red}from the} i.i.d.\ {\color{red}setting to dependent} environments. The {\color{red}precise statement, established} in \cite{Chenn} is {\color{violet}presented} {\color{red}below} for {\color{violet}completeness}.

    \begin{theorem}\label{thm: non-equality of rate functions}
        \normalfont
        For any $d\geq1$ and $\kappa>0$, {\color{red}if} $\mathbb{P}$ {\color{red}satisfies} $\textbf{(SMF)}_{C,g}$ and {\color{red}its} support is not a singleton, i.e.~$\text{dis}(\mathbb{P})>0$, then $I_a(x)<I_q(x)$ at some interior point $x\in\text{int}(\mathbb{D})$ of the velocity surface.
    \end{theorem}
        The {\color{red}strict inequality between} $I_a$ and $I_q$ {\color{red}originates} {\color{violet}from} {\color{red}their failure to match on the lower-dimensional boundary strata} $\partial\mathbb{D}_{d-2}$. {\color{red}whether} this {\color{red}strict inequality holds under the broader} $\textbf{(SM)}_{C,g}$ or $\textbf{(SMG)}_{C,g}$ scenarios {\color{red}remains an open question.} Beyond {\color{red}these} qualitative {\color{red}characterizations} {\color{violet}of} how the level of disorder $\text{dis}(\mathbb{P})$ {\color{red}governs} the {\color{red}equality} between $I_a$ and $I_q$ on $\mathbb{D}$, {\color{red}a natural question is} whether {\color{red}these} results {\color{red}can be refined} in more quantitative {\color{violet}manners}.\par

        A natural viewpoint is to consider scenario{\color{red}s} where the {\color{red}environmental} randomness is parametrized {\color{red}explicitly}. Given a {\color{red}fixed} $\alpha:\mathbb{V}\to(0,1)$ with $\sum_{e\in\mathbb{V}}\alpha(e)=1$ {\color{red}and define the space of admissible fluctuation profiles} $\mathcal{E}_\alpha\coloneqq\{(r(e))_{e\in\mathbb{V}}\in[-1,1]^{\mathbb{V}}:\,\sum_{e\in\mathbb{V}}\alpha(e)r(e)=0\,\text{and}\,\sup_{e\in\mathbb{V}}\abs{r(e)}=1\}$. {\color{red}Let} $(\eta(x,e))_{x\in\mathbb{Z}^d,e\in\mathbb{V}}\in\mathcal{E}^{\mathbb{Z}^d}_\alpha$ {\color{red}be an ergodic and shift-invariant random field under} some probability $\mathbb{Q}$, {\color{red}satisfying} $E^{\mathbb{Q}}[\eta(x,e)]=0$, $\forall~(x,e)\in\mathbb{Z}^d\times\mathbb{V}$. The {\color{red}spatial correlations of} $(\eta(x,e))_{x\in\mathbb{Z}^d,e\in\mathbb{V}}$ under $\mathbb{Q}$ {\color{red}are chosen} so that the parametrized environment $\omega_\epsilon(x,e)\coloneqq\alpha(e)(1+\epsilon\eta(x,e))$ satisfies $(\textbf{SMF})_{C,g}$, or either $(\textbf{SM})_{C,g}$ or $(\textbf{SMG})_{C,g}$ under different mixing {\color{red}assumptions on $\mathbb{Q}$. In this framework, the law} of the parametrized environment $(\omega_\epsilon(x,e))_{x\in\mathbb{Z}^d,e\in\mathbb{V}}$ {\color{red}is governed by} $\mathbb{Q}$, {\color{red}while}{\color{violet}the magnitude} {\color{red}of disorder} is {\color{red}dialed} by the numerical parameter $\epsilon>0$, {\color{red}enabling sharper quantitative estimates for} the relevant {\color{red}rate functions.}\par
        Indeed, {\color{violet}this} {\color{red}parametrization framework} is quite general and {\color{red}subsumes} several widely studied models for RWRE such as {\color{red}those in} \cite{Campos/Ramirez,Sabot}. {\color{red}Moreover,} \cite[Theorem 2.9]{Bazaes/Mukherjee/Ramirez/Sagliett} {\color{violet}provides the first evidence} {\color{red}of} a phase transition {\color{red}for} $I_q-I_a$ on $\partial\mathbb{D}$ w.r.t. $\text{dis}(\mathbb{P})$ {\color{red}under i.i.d.\ parameterized} environments. {\color{violet}Inspired by these insights}, we {\color{red}leave} {\color{violet}the} {\color{red}generalization} of {\color{red}such phase transition phenomena} to {\color{red}dependent} mixing environments, such as $(\textbf{SMF})_{C,g}$,  $(\textbf{SM})_{C,g}$ or $(\textbf{SMG})_{C,g}$, to future {\color{red}works}.

\section{Existence of quenched and annealed large deviation principles}
    The LDP for RWRE on $\mathbb{R}$ was {\color{violet}addressed} by Greven/den Hollander \cite{Greven/den Hollander} for the quenched law, and {\color{red}subsequently} for both {\color{red}the} quenched and annealed laws by Comets/Gantert/Zeitouni \cite{Comets/Gantert/Zeitouni} via {\color{violet}a} variational formula. Kosygina/Rezakhanlou/Varadhan \cite{Kosygina/Rezakhanlou/Varadhan} {\color{red}introduced} novel homogenization {\color{red}techniques} to {\color{red}derive} quenched LDP for elliptic diffusions with random drift {\color{red}via} {\color{violet}a} convex variational approach, see also \cite{Kosygina/Varadhan} and Rosenbluth \cite{Rosenbluth}. The derivation of annealed LDP{\color{red}s} through the empirical process of the RWRE via variational {\color{red}methods} \cite{Rassoul-Agha} {\color{red}paved the way for} the more abstract process-level LDP. The quenched process-level LDP {\color{violet}for} ergodic environment{\color{violet}s} was {\color{red}established} in \cite{Rassoul-Agha/Seppalainen2} with {\color{violet}the} rate function
    {\color{violet}expressed in terms of} relative entropy. {\color{red}P}rocess-level LDP{\color{red}s of this type have been} further {\color{red}extended to} dynamic environments \cite{Rassoul-Agha/Seppalainen/Yilmaz} from the viewpoint of the particle, as well as in related settings such as random walks in random potentials \cite{Rassoul-Agha/Seppalainen/Yilmaz2}.\par

    In this section, we formulate both the quenched and annealed LDP{\color{red}s} of the RWRE {\color{red}under the strong mixing conditions} $\textbf{(SM)}_{C,g}$ {\color{red}and} $\textbf{(SMG)}_{C,g}$. The {\color{red}proof} strategy {\color{red}for} the quenched LDP {\color{red}builds upon the framework} of \cite[Theorem 2.3]{Varadhan1}, {\color{red}whereas the proof of} the annealed LDP {\color{red}relies on} {\color{violet}verifying} the {\color{red}criterion} {\color{violet}of} \cite[Theorem 2.1]{Rassoul-Agha} under our mixing {\color{red}assumptions}. The geometry of the zero sets of $I_q$ and $I_{\color{violet}a}$ follows from the {\color{red}arguments in} \cite[Theorem 8.1]{Varadhan1} and \cite[Remark 2.4]{Rassoul-Agha}. {\color{red}While our} exposition {\color{red}remains} streamlined {\color{red}and} closely {\color{red}aligns with the techniques} in \cite{Rassoul-Agha,Varadhan1}, {\color{red}a primary} distinction {\color{red}lies in our mixing setup which} is {\color{red}significantly} more general than the Dobrushin-type condition {\color{red}assumed} in \cite{Rassoul-Agha}.\par

    \subsection{Abstract path space and continuous maps.}\label{sec: abstract path space}
    Viewed from the {\color{violet}front of the} random walk, the environment at {\color{red}time $N$} becomes $\omega_N=\vartheta_{X_N}\omega$, {\color{red}where} $\vartheta$ {\color{red}denotes} the canonical shifts on $\mathbb{Z}^d$. From this {\color{red}point of view,} the {\color{red}shifted} trajectory {\color{red}at time} $N$ {\color{red}corresponds to} $\{X_{-N}=-X_N,\,X_{-N+1}=X_1-X_N,\ldots,\,X_{-1}=X_{N-1}-X_N,\,X_0=0 \}$, $\forall~N\in\mathbb{N}$. We whence consider the {\color{red}sequence} of increments {\color{red}defined by} $\mathbb{W}_N\coloneqq\{w_N\in\mathbb{V}^{N}:\,w_N=(z_{-N+1}=x_{-N+1}-x_{-N},\ldots,\,z_0=x_0-x_{-1})\}$, $\forall~N\in\mathbb{N}$, {\color{red}with} $\mathbb{W}_0=\{\phi\}$ {\color{red}being} a singleton. {\color{red}Furthermore,} $\mathbb{W}_{\infty}\coloneqq(\mathbb{V})^{-\mathbb{N}}$ {\color{red}denotes} the {\color{red}space} of {\color{red}backward trajectories} ending at the origin.\par

    We also {\color{red}define} $\mathbb{W}\coloneqq\cup_{N\geq0}\mathbb{W}_N\cup\mathbb{W}_{\infty}$ and {\color{red}set} $x_j(w)\coloneqq-(z_{-N+1}+\cdots+z_{-n-N})$,  $\forall~w\in\mathbb{W}_N$ and $-N\leq j\leq0$. Extending {\color{violet}this} notion to infinite paths, $\forall~w\in\mathbb{W}_{\infty}$ we {\color{red}denote by} $x_j(w)$ the {\color{red}position of $w$ at} $j$th {\color{red}step} for all $j\leq0$. {\color{red}T}o each vertex $x\in\mathbb{Z}^d$ we count its hitting times on oriented edges {\color{red}via} $n_x(w)\coloneqq\sum_{z\in\mathbb{V}}n_{x,z}(w)$ where $n_{x,z}(w)\coloneqq\sum_{j\leq-1}\mathbbm{1}_{\{x_j=x,\,x_{j+1}=x+z\}}(w)$, $\forall~w\in\mathbb{W}$ and $z\in\mathbb{V}$. Counting the hitting times on {\color{red}edges} encodes the conditional transition dynamics {\color{red}on} space $\cup_{N\geq0}\mathbb{W}_N$. To this end, we {\color{red}explicitly} specify the {\color{violet}admissible} transitions {\color{red}between} finite paths {\color{violet}of} {\color{red}varying} length{\color{red}s}. For each $w\in\mathbb{W}$ and $z\in\mathbb{V}$, {\color{red}we} define the shift {\color{red}operators} $\vartheta^*$ and $\vartheta_z$ {\color{red}by} 
    \[
        z_{j+1}(\vartheta^*(w))=z_j(w),\quad z_0(\vartheta_z(w))=z,\quad\text{and}\quad z_j(\vartheta_z(w))=z_{j+1}(w),\quad\forall~j\leq-1.
    \]

    {\color{violet}In accordance} with {\color{red}our} notation, for each {\color{red}finite path} $\gamma=(z_j)_{-N<j\leq0}\in\mathbb{W}_N$ we write $\vartheta_\gamma=\vartheta_{z_{-N+1}}\cdots\vartheta_{z_0}$. The annealed law of the RWRE corresponds naturally to transition probabilities $q:\cup_{N\geq0}\mathbb{W}_N\times\mathbb{V}\to[\kappa,1]$, defined {\color{red}by}
    \begin{equation}\label{eqn: transition probability of finite paths}
        q(w,z) = \frac{\mathbb{E}[ \prod_{x,y\in\mathbb{V}}\omega(x,x+y)^{n_{x,y}(w)}\omega(0,z) ]}{\mathbb{E}[ \prod_{x,y\in\mathbb{V}}\omega(x,x+y)^{n_{x,y}(w)} ]},\qquad\forall~w\in \cup_{N\geq0}\mathbb{W}_N\quad\text{and}\quad z\in\mathbb{V}.
    \end{equation}
    Starting from {\color{red}an initial path} $w\in\cup_{N\geq0}\mathbb{W}_N$, the transition{\color{red}s in} (\ref{eqn: transition probability of finite paths}) induce a Markov chain $Q_w$ on $\cup_{N\geq0}\mathbb{W}_N$ {\color{red}that} jumps to the state $\vartheta_z(w)$ with probability $q(w,z)$. {\color{violet}Alternatively,} $Q_w$ {\color{red}can be viewed} as {\color{red}the law of} a random walk $(\Tilde{X}_N)_{N\geq0}$ {\color{violet}conditioned on the} past {\color{violet}trajectory} $w$ {\color{red}with} future increments $Z_N\coloneqq\Tilde{X}_N-\Tilde{X}_{N-1}$, $\forall~N\geq1$.
\begin{figure}
    \includegraphics[width=\textwidth]{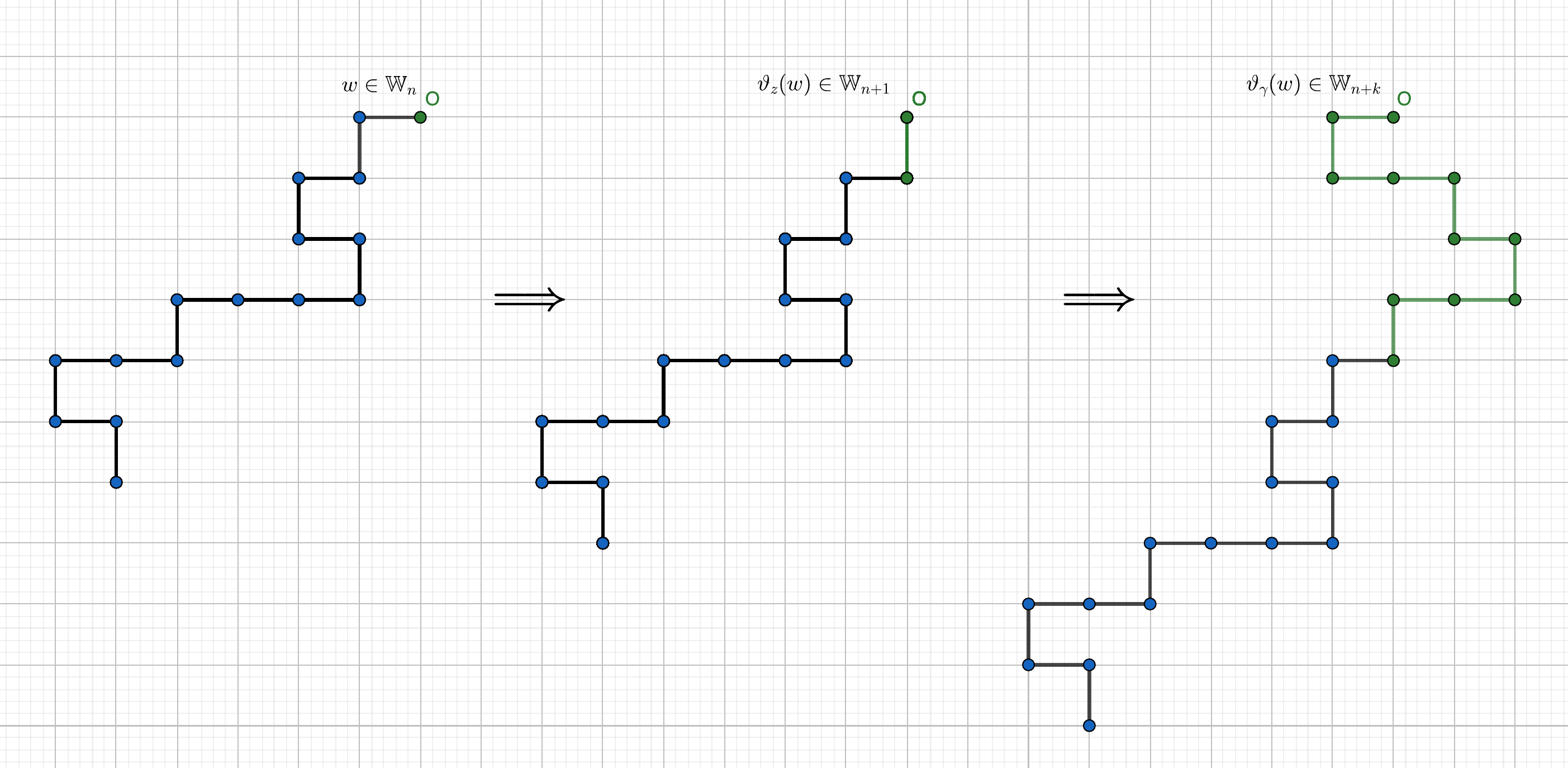}
    \caption{Transitions with respect to shifts $\vartheta_z$ and $\vartheta_\gamma$ for $\gamma\in\mathbb{W}_k$ in the abstract path space.}
    \label{RWRE}
\end{figure}
    In particular, the Markov chain $Q_{\phi}$ {\color{red}with initial state $\phi$} corresponds to the annealed law $P_0$. {\color{red}Furthermore}, {\color{red}the} transition {\color{red}probabilities} $q$ {\color{red}extend naturally to the space of transient} paths  $\mathbb{W}^{\text{tr}}\coloneqq\cup_{N\geq0}\mathbb{W}_N\cup\mathbb{W}^{\text{tr}}_{\infty}${, where} $\mathbb{W}^{\text{tr}}_{\infty}\coloneqq\{w\in\mathbb{W}_{\infty}:\,\abs{x_j(w)}\to\infty~\text{as}~-j\to\infty\}$. We {\color{red}whence obtain} the following result.

    \begin{lemma}
        \normalfont
        When $\mathbb{P}$ {\color{red}satisfies} either $\textbf{(SM)}_{C,g}$ or $\textbf{(SMG)}_{C,g}$, the transient kernels {\color{red}$q(w,z)$} extend {\color{red}continuously} to $\mathbb{W}^{\text{tr}}_{\infty}$ while preserving the {\color{red}uniform} ellipticity condition, {\color{red}yielding a map} $q:(w,z)\in\mathbb{W}^{\text{tr}}\times\mathbb{V}\mapsto[\kappa,1]$.
    \end{lemma}
    \begin{proof}
        For each $m\in\mathbb{N}$, let $V_m\coloneqq\mathbb{Z}^d\cap[-m,m]^d${\color{red}, and f}or any {\color{red}environment} $\omega\in\Omega$, {\color{red}let} $\omega_{V_m^c}\coloneqq(\omega_x)_{x\in V_m^c}$. For any $w\in\mathbb{W}^{\text{tr}}$ and $m\in\mathbb{N}$, define the {\color{red}local conditional} probability measure $\mathbb{P}^{\omega_{V_m^c}}_{V_m,w}$ via its Radon--Nikodým derivative{\color{red}:}
        \begin{equation}\label{eqn: derivative of Pw}        
            d\mathbb{P}_{V_m,w}^{\omega_{V_m^c}} = \frac{\prod_{y,x\in V_m} \omega(x,x+y)^{n_{x,y}(w)}}{ \mathbb{E}[\prod_{y,x\in V_m} \omega(x,x+y)^{n_{x,y}(w)}| \mathscr{F}_{V_m^c} ] }\,d\mathbb{P}( (\omega_x)_{x\in V_m}\in\vdot| \mathscr{F}_{V_m^c} ),\qquad\omega_{V_m^c}\text{-a.s.}
        \end{equation}
        According to Lemma \ref{lem: A0}, $\exists~\mathbb{P}_w$ a unique probability measure on $\Omega$ consistent\footnote{Here by saying consistent we mean $\mathbb{P}_w((\omega_x)_{x\in V_m}\in\vdot|\mathscr{F}_{V_m^c})=\mathbb{P}_{V_m,w}^{\omega_{V_m^c}}(\vdot)$, for all $m\in\mathbb{N}$. } with the specification $(\mathbb{P}_{V_m,w}^{\omega_{V_m^c}})_{V_m\subseteq\mathbb{Z}^d}$. {\color{red}We then set} $\mathfrak{q}(w,z)\coloneqq E_{\mathbb{P}_w} [\omega(0,z) ]$, $\forall~w\in\mathbb{W}^{\text{tr}}$, $z\in\mathbb{V}$. {\color{red}Note that} $\mathfrak{q}$ is not well-defined on $\mathbb{W}\backslash\mathbb{W}^{\text{tr}}${\color{red}, since $w \notin \mathbb{W}^{\mathrm{tr}}$ implies} $n_{x,z}(w)=\infty$ {\color{red}for some pair $(x,z)$,} {\color{violet}rendering} the marginal distributions (\ref{eqn: derivative of Pw}) {\color{red}ill-defined}. {\color{violet}Conversely, for} {\color{red}any} $w\in\cup_{N\geq0}\mathbb{W}_N$, {\color{red}taking $m$} sufficiently large {\color{red}so that all non-zero visits are contained in} $V_m$ {\color{red}yields} 
        \[       
            \mathfrak{q}(w,z) = E_{\mathbb{P}_w}[\omega(0,z)] = \lim_{m\to\infty} \frac{\mathbb{E}[\prod_{y,x\in V_m} \omega(x,x+y)^{n_{x,y}(w)}\omega(0,z)]}{ \mathbb{E}[\prod_{y,x\in V_m} \omega(x,x+y)^{n_{x,y}(w)}]}= \frac{\mathbb{E}\big[\prod_{y,x\in \mathbb{Z}^d} \omega(x,x+y)^{n_{x,y}(w)}\omega(0,z)\big]}{ \mathbb{E}\big[\prod_{y,x\in \mathbb{Z}^d} \omega(x,x+y)^{n_{x,y}(w)}\big]} = q(w,z).
        \]
        Hence, $\mathfrak{q}(\vdot,z):\mathbb{W}^{\text{tr}}\to[\kappa,1]$ indeed {\color{red}provides a well-defined extension of} $q(\vdot,z)$ {\color{red}from} $\cup_{N\geq0}\mathbb{W}_N\times\mathbb{V}$ {\color{red}to $\mathbb{W}^{\mathrm{tr}} \times \mathbb{V}$.} {\color{red}For brevity,} we {\color{red}retain the notation} $q(\vdot,z)$ {\color{red}for this extended kernel.}
    \end{proof}

    When $\mathbb{P}$ {\color{red}satisfies} $\textbf{(IID)}$, the transition kernel $q(\vdot,z):\mathbb{W}^{\text{tr}}\to[\kappa,1]$ is naturally continuous with respect to the topology $\mathcal{U}$ {\color{violet}induced} by {\color{red}the evaluations} $x_j:\mathbb{W}\to\mathbb{Z}^d$ and $n_{x,z}:\mathbb{W}\to\mathbb{N}\cup\{\infty\}$ for all $x,z,j$. {\color{red}Under dependent environments,} however, this continuity {\color{red}generally} {\color{violet}fails to} hold. Nevertheless, we {\color{red}establish} {\color{violet}that} {\color{red}the continuity of $q(\cdot, z)$ persists} under {\color{red}either mixing} condition $\textbf{(SM)}_{C,g}$ or $\textbf{(SMG)}_{C,g}$.
    \begin{lemma}\label{lem: continuity of q(w,z)}
        \normalfont
        When $\mathbb{P}$ {\color{red}satisfies} either $\textbf{(SM)}_{C,g}$ or $\textbf{(SMG)}_{C,g}$, the transition kernels $q:\mathbb{W}^{\text{tr}}\times\mathbb{V}\to[\kappa,1]$ are continuous with respect to the subspace topology of $\mathcal{U}$ on $\mathbb{W}$.
    \end{lemma}
    \begin{proof}
        The proof of this lemma {\color{red}proceeds by adapting the} argument {\color{red}of} \cite[Lemma 4.2]{Rassoul-Agha} {\color{red}to our} mixing conditions, {\color{red}and} {\color{violet}is deferred to Appendix \ref{lem: A00}.}
    \end{proof}
    {\color{red}Since} $(\mathbb{W},\mathcal{U})$ {\color{red}fails to be} compact{\color{red}, we invoke the} Stone--Čech Compactification Theorem \cite[Chapter 5]{Kelley} {\color{red}to construct} the compactified space $(\overline{\mathbb{W}},\mathcal{U})$ of $\mathbb{W}$. {\color{red}This construction}{\color{violet}guarantees} the {\color{red}continuous} extension {\color{red}of the mappings} $x_j,n_{x,z},\vartheta_z,\vartheta^*$ on $\overline{\mathbb{W}}$ and $q(\vdot,z):\overline{\mathbb{W}}^{\text{tr}}\to[\kappa,1]$, where $\overline{\mathbb{W}}^{\text{tr}}$ {\color{red}denotes} the compact closure of $\mathbb{W}^{\text{tr}}$ in $(\overline{\mathbb{W}},\mathcal{U})$. In fact, $\overline{\mathbb{W}}^{\text{tr}}=\overline{\mathbb{W}}$ {\color{red}since} $\mathbb{W}^{\text{tr}}$ is dense in $\mathbb{W}$. {\color{red}This} compactness {\color{red}framework plays a central role in establishing} the process-level LDP {\color{red}for} the empirical distributions on $\overline{\mathbb{W}}$. {\color{red}Related models and LDP formulations in dependent random field can be found in} Gantert/König/Shi \cite{Gantert/Konig/Shi}, Kubota \cite{Kubota} and Mathieu/Piatnitski \cite{Mathieu/Piatnitski}.

\subsection{The quenched large deviation principle.}
    The proof of the quenched LDP {\color{red}rests} on pathwise argument{\color{red}s} following \cite[Section 2]{Varadhan1} {\color{violet}which requires} {\color{red}only that} $\mathbb{P}$ {\color{violet}be} ergodic. {\color{red}Below, we present} a streamlined exposition {\color{red}highlighting the key steps of this} approach. Let $G_{u,\omega}(t,x,y) \coloneqq-\log\,\sup_{n\in\mathbb{N}}\{e^{-u\abs{n-t}} P_{x,\omega}(X_n=y) \}$, $\forall~x,y\in\mathbb{Z}^d$ for any $u>0$ and $t\geq0$. {\color{red}Under} uniform ellipticity of $\mathbb{P}$, {\color{red}$G_{u,\omega}$ exhibits} Lipschitz continuity on $[0,\infty)\times\mathbb{Z}^d\times\mathbb{Z}^d$ {\color{red}with a constant $\text{Lip}(G) > 0$}, i.e.
    \begin{equation}\label{eqn: G is Lipschitz}
        \abs{G_{u,\omega}(t+t^\prime,x+x^\prime,y+y^\prime) - G_{u,\omega}(t,x,y)} \leq \text{Lip}(G)\big(t^\prime+|x^\prime|_1+|y^\prime|_1\big),
    \end{equation}
    for all such $t,t^\prime,x,x^\prime,y,y^\prime$ in their {\color{violet}respective} domain{\color{violet}s}. {\color{red}By applying standard subadditive ergodic} convergence {\color{red}arguments for} stochastic processes with {\color{red}such} regularity, we {\color{red}obtain} the following lemma.

    \begin{lemma}\label{lem: quenched LDP lemma}
        \normalfont
        When $\mathbb{P}$ {\color{red}satisfies} either $\textbf{(SM)}_{C,g}$ or $\textbf{(SMG)}_{C,g}$, $\forall~u>0$,  $\exists~\Tilde{G}_u:\mathbb{R}^d\to\mathbb{R}$ convex deterministic on $\mathbb{R}^d$ such that
        \[
            \lim_{\substack{N\to\infty\\y_N/N\to\eta}} \tfrac{1}{N} G_{u,\omega}(tN,0,y_N) = t\Tilde{G}_u(t^{-1}\eta),\qquad\mathbb{P}\text{-a.s.,}\qquad\forall~\eta\in\mathbb{R}^d.
        \]
    \end{lemma}
    \begin{proof}
        First, {\color{red}using} the $\text{Lip}(G)$-Lipschitz continuity of $(G_{u,\omega})_{N\geq1}$, the {\color{red} rescaled} sequence $G^{(N)}_{u,\omega}(t,\xi,\eta)\coloneqq \frac{1}{N}G_{u,\omega}(tN,\xi N,\eta N)$ {\color{red}is relatively compact in the topology of uniform convergence on compact sets, yielding} sub-sequential converge{\color{red}nce} in distribution to a $\mathbb{P}$-a.s. continuous process $G^\infty_{u,\omega}:[0,\infty)\times\mathbb{R}^d\times\mathbb{R}^d$. {\color{red}We then} invoke the Subadditive Ergodic Theorem \cite[Theorem 1.10]{Liggett} to show {\color{red}that} the limit $\lim_{N\to\infty} G^{(N)}_{u,\omega}(t,0,\eta) \eqqcolon g(t,\eta)$ is deterministic and {\color{red}satisfies the scaling property} ${\color{red}g(t,\eta) = t\tilde{G}_u (t^{-1}\eta)}$ {\color{red}for a function} $\tilde{G}_u:\mathbb{R}^d\to\mathbb{R}$. The {\color{red}convexity and further} regularity of $\tilde{G}_u$ then follow {\color{red}by adapting} the remaining steps in \cite{Varadhan1}.
    \end{proof}
    For each $\eta \in \mathbb{R}^d$, the map{\color{red}ping} $u \mapsto \Tilde{G}_u(\eta)$ {\color{red}is non-decreasing} on $[0,\infty)$, {\color{red}ensuring} that {\color{red}the pointwise limit $I_q(\eta) \coloneqq \lim_{u\to\infty} \tilde{G}_u(\eta)$ exists in $\mathbb{R} \cup \{+\infty\}$.} The {\color{red}resulting} limit function $I_q : \mathbb{R}^d \to \mathbb{R} \cup \{+\infty\}$ is deterministic, convex and lower-semicontinuous. {\color{red}Finally,} its continuity on {\color{red}the velocity surface} $\mathbb{D}$ follows from \cite[Theorem 10.2]{Rockafellar}.

    \begin{proof}[Proof of the quenched LDP]
        We {\color{red}first establish} the quenched LDP upper-bound. For any closed $F\subseteq\mathbb{R}^d$, we observe that $\varlimsup_{N\to\infty}\frac{1}{N}\log P_{0,\omega}(X_N/N \in F ) = \varlimsup_{N\to\infty} \frac{1}{N}\log \sum_{z\in NF\cap\mathbb{Z}^d} P_{0,\omega}(X_N = z)$. Since {\color{red}cardinality of the lattice intersection grows polynomially, we have the standard estimate}
        \[
            0\leq \frac{1}{N}\log \sum_{z\in NF\cap\mathbb{Z}^d} P_{0,\omega}(X_N = z) - \max_{z\in NF\cap\mathbb{Z}^d} \frac{1}{N}\log P_{0,\omega}(X_N = z) \leq \frac{1}{N}\abs{N F\cap\mathbb{Z}^d}\to0\qquad\text{as}\quad N\to\infty,
        \]
        {\color{red}Applying} Lemma \ref{lem: quenched LDP lemma}, {\color{red}it follows that} for any $u>0$, 
        \begin{equation*}\begin{aligned}
            \varlimsup_{N\to\infty}\tfrac{1}{N}\log P_{0,\omega}\big(\tfrac{X_N}{N} \in F \big) = \varlimsup_{N\to\infty} \max_{z\in NF\cap\mathbb{Z}^d} \tfrac{1}{N}\log P_{0,\omega}(X_N = z) \leq \varlimsup_{N\to\infty} \max_{z\in N F\cap\mathbb{Z}^d} (-1/N) G_{u,\omega}(N,0,z) \underset{u\searrow0}{\leq} -\inf_{\xi\in F} I_q(\xi),
        \end{aligned}\end{equation*}
        verifying the quenched LDP upper-bound. {\color{red}We next turn to} the {\color{red}quenched} lower-bound. For any open $G\subseteq\mathbb{R}^d$, if $\inf_{\xi\in G}I_q(\xi)=\infty$, the lower-bound trivially holds. Otherwise, fix $\xi\in G$ {\color{red}with} $I_q(\xi)<\infty$, {\color{red}and select} $\delta>0$ {\color{red}sufficiently small such that the open ball} $\{\eta\in\mathbb{R}^d:\,\abs{\eta-\xi}_2<\delta\}\subseteq G$. {\color{red}Choose} $u(\delta)>0$ large enough with $u\delta> I_q(\xi)+1$. Clearly, $\Tilde{G}_u(\xi)\leq I_q(\xi)$. For any {\color{red}sequence} $(x_N)_{N\geq1}$ {\color{red}satisfying} $x_N/N\to\xi$ as $N\to\infty$,
        \[
            N^{-1}\log\,\sup_{k\in\mathbb{N}}\big\{e^{-u\abs{k-N}}P_{0,\omega}\big(X_k=x_N\big) \big\}\to-\Tilde{G}_u(\xi)\qquad\text{as}\quad N\to\infty,\qquad\mathbb{P}\text{-a.s.}
        \]
        Hence $\exists~N(\omega)\in\mathbb{N}$ {\color{red}such that} $\forall~N\geq N(\omega)$, the supremum in the definition of $G_{u,\omega}$ is achieved {\color{red}within the time interval} $\mathbb{K}_N\coloneqq\{k\in\mathbb{N}:\,\abs{k-N}\leq (I_q(\xi)+1)N/u\}$. {\color{red}This implies the lower bound} $\sup_{k\in\mathbb{K}_N} P_{0,\omega}(X_k=x_N) \geq  \sup_{k\in\mathbb{K}_N} \{e^{-u\abs{k-N}}P_{0,\omega}(X_k=x_N) \}\geq e^{-I_q(\xi)N+o(N)}$, $\mathbb{P}$-a.s. {\color{red}Because} $(X_N)_{N\geq0}$ takes unit jumps, if {\color{red}$X_k = x_N$} at {\color{red}some} time {\color{red}$k \in \mathbb{K}_N$}, then at time $N$ the {\color{red}position $X_N$ satisfies} $\abs{X_N-x_N}\leq  (I_q(\xi)+1)N/u<\delta N$. Hence, {\color{red}$X_k = x_N$ forces $X_N / N \in B(\xi, \delta) \subseteq G$. Combining these estimates yields}
        \[
            \varliminf_{N\to\infty}\tfrac{1}{N}\log P_{0,\omega}\big(|\tfrac{X_N}{N}-\xi|<\delta\big) \geq \varliminf_{N\to\infty}\tfrac{1}{N}\log\,\sup_{k\in\mathbb{K}_N} P_{0,\omega}(X_k=x_N) \geq -I_q(\xi),\qquad\mathbb{P}\text{-a.s.}
        \]
        Since $\delta>0$ {\color{red}and $\xi \in G$ were} taken arbitrary, {\color{red}taking the supremum over $\xi \in G$ confirms} the quenched LDP lower-bound.
    \end{proof}

\subsection{The annealed large deviation principle.}\label{sec: annealed LDP existence}
    We denote by $\overline{\mathbb{W}}$ the compactification of $\mathbb{W}$ with respect to the topology $\mathcal{U}$ {\color{red}and by $\overline{\mathbb{W}}^{\mathrm{tr}}$ the compact closure of the subspace} $\mathbb{W}^{\text{tr}}$ {\color{red}in $\overline{\mathbb{W}}$.} {\color{red}In related settings with} mixing {\color{red}environments} \cite{Rassoul-Agha}, {\color{red}the annealed LDP} was {\color{red}established} for the empirical processes $\mathcal{R}_N\coloneqq\frac{1}{N}\sum_{1\leq k\leq N}\delta_{\vartheta_{z_j}\cdots\vartheta_{z_1}\phi}$, $N\in\mathbb{N}$, via a variational approach \cite{Donsker/Varadhan,Varadhan2}. To {\color{red}establish} the annealed LDP {\color{red}for $(X_N)_{N\geq0}$} in our {\color{red}mixing} scenario{\color{red}s}, we {\color{violet}require} {\color{red}the following} technical lemma.

    \begin{lemma}\label{lem: kkkk}
        \normalfont
        When $\mathbb{P}$ {\color{red}satisfies} either $\textbf{(SM)}_{C,g}$ or $\textbf{(SMG)}_{C,g}$, $\forall~N\in\mathbb{N}$, {\color{red}finite} $S\subseteq\mathbb{Z}^d$ and $\mathbb{V}$-valued $\mathcal{Z}=(z_i)_{i\geq0}$, {\color{red}define the function} $H(\vdot)$ by $H(N,S,\mathcal{Z})\coloneqq \sum_{y\in\partial^r\{x_1,\ldots,x_N\}}\sum_{z\in S} e^{-g\abs{y-z}_1}$ under $(\textbf{SM})_{C,g}$.
        {\color{red}Alternatively define} $H(\vdot)$ by $H(N,S,\mathcal{Z})\coloneqq \sum_{1\leq i\leq N}\sum_{z\in S} e^{-g\abs{x_i-z}_1}$ {\color{red}under $(\textbf{SMG})_{C,g}$}. 
        Then $H(N,S,\mathcal{Z})$ depends only on $z_1,\ldots,z_n$. If $\exists~\ell\in\mathbb{S}^{d-1}$ {\color{red}such that} $\langle\eta,\ell\rangle>0$ {\color{red}for} $\eta=\lim_{N\to\infty}\frac{1}{N}\sum_{i=1}^N z_i\neq0$ {\color{red}and} $\sup_{x\in S}\langle x,\ell\rangle<\infty$, then $\sup_{N\in\mathbb{N}}H(N,S,\mathcal{Z})<\infty$. Moreover, there exists {\color{red}a} constant $C>0$ such that for any $w_1,w_2\in\mathbb{W}^{\text{tr}}$, $|\log\frac{dQ_{w_1}}{dQ_{w_2}}|_{\mathscr{F}_N}(\mathcal{Z})| \leq C H(N,S(w_1)\cup S(w_2),\mathcal{Z})$, where $\mathscr{F}_N$ is the $\sigma$-field generated by $Z_1,\ldots,Z_N$ and $S(w)\coloneqq\{x\in\mathbb{Z}^d:\,n_{x}(w)>0\}$ {\color{red}stands for the lattice support of} $w\in\mathbb{W}^{\text{tr}}$.
    \end{lemma}
    \begin{proof}
        Let $\Bar{w}=(z_1,\ldots,z_N)$, $S=\{x_1,\ldots,x_N\}$ and $\Bar{S}=S-S(w_1)\cup S({\color{red}w_2})$. The {\color{red}required} estimate {\color{red}for the Radon--Nikodým derivative reduces to establishing that}
        \begin{equation}\label{eqn: annealed (a)}
            \frac{dQ_{w_1}}{dQ_{w_2}}\bigg|_{\mathscr{F}_N}(\mathcal{Z}) \leq C(\kappa){\color{red}e}^{\sum_{j\leq N}\mathbbm{1}_{S(w_1)\cup S(w_2)}(x_j)} \vdot\sup_{\substack{V\supseteq S\\ V\cap\partial^rS=\emptyset }} \frac{E_{\mathbb{P}_{w_1}}[\prod_{y,x\in \Bar{S}}\omega(x,y)^{n_{x,y}(\Bar{w})}\big|\mathscr{F}_{V^c}]}{E_{\mathbb{P}_{w_2}}[\prod_{y,x\in \Bar{S} }\omega(x,y)^{n_{x,y}(\Bar{w})}  \big|\mathscr{F}_{V^c}]}.
        \end{equation}
        Here the constant $C(\kappa)>0$ originates from the {\color{red}uniform} ellipticity condition $\omega(x,y)\geq\kappa$. {\color{red}Under the mixing conditions $\mathbf{(SM)}_{C,g}$ and $\mathbf{(SMG)}_{C,g}$,} the Radon--Nikodým derivatives {\color{red}satisfy}
        \begin{equation}\label{eqn: annealed (c)}
            \frac{d\mathbb{P}\big((\omega_x)_{x\in V\cap S(w_i)}\in\vdot\big|\eta\big)}{d\mathbb{P}\big((\omega_x)_{x\in V\cap S(w_i)}\in\vdot\big|\eta^\prime\big)} \leq \exp\bigg(C\sum\limits_{\substack{y\in\partial^r\Bar{S} \\ x\in\partial^rV\cap S(w_i)}}e^{-g\abs{x-y}_1}\bigg),\quad \frac{d\mathbb{P}\big((\omega_x)_{x\in V\cap S(w_i)}\in\vdot\big|\eta\big)}{d\mathbb{P}\big((\omega_x)_{x\in V\cap S(w_i)}\in\vdot\big|\eta^\prime\big)} \leq \exp\bigg(C\sum\limits_{\substack{y\in \Bar{S}\\ x\in V\cap S(w_i)}}e^{-g\abs{x-y}_1}\bigg)   
        \end{equation}
        respectively under $(\textbf{SM})_{C,g}$ and $(\textbf{SMG})_{C,g}$ where $\eta$ and $\eta^\prime$ are environmental configurations on $V^c\cup\Bar{S}$ with $\eta=\eta^\prime$ on $V^c\backslash\Bar{S}$. {\color{red}Substituting} (\ref{eqn: annealed (c)}) back {\color{red}into the main ratio} (\ref{eqn: annealed (a)}), we {\color{red}deduce} that
        \[
            \log\frac{dQ_{w_1}}{dQ_{w_2}}\bigg|_{\mathscr{F}_N}(\mathcal{Z}) \leq C^\prime(\kappa) + {\sum_{1\leq j\leq N}\mathbbm{1}_{S(w_1)\cup S(w_2)}(x_j)}+{CH(N,S,\mathcal{Z})},\qquad\forall~N\in\mathbb{N},
        \]
        where $C^\prime(\kappa)\coloneqq\log C(\kappa)$. The {\color{red}remaining assertions} follow {\color{red}by arguments analogous to those} in \cite[Lemma 5.3]{Rassoul-Agha}, {\color{red}which completes} the {\color{red}proof}.
    \end{proof}

    {\color{violet}By an} argument {\color{violet}similar} to {\color{violet}that in} \cite[Lemma 7.2]{Varadhan2}, {\color{violet}the use of} a variational method yields the convexity of $I_a$ on the velocity surface $\mathbb{D}$. Since $I_a$ is convex and lower semicontinuous, its continuity on the convex domain $\mathbb{D}$ follows from \cite[Theorem 10.2]{Rockafellar}. We now {\color{violet}turn to the proof of} the annealed LDP for $(X_N)_{N\geq1}$.
    \begin{proof}[Proof of the annealed LDP]
        The {\color{red}overall} strategy {\color{violet}for proving} this theorem {\color{red}aligns with that of} \cite[Theorem 2.1]{Rassoul-Agha}, {\color{red}whose conditions (i-v) are verified in our framework as follows:} \cite[Theorem 2.1-conditions (i-ii)]{Rassoul-Agha} {\color{red}correspond to the setup} {\color{violet}presented} in Section  \ref{sec: abstract path space}, {\color{red}where} \cite[Theorem 2.1-condition (ii)]{Rassoul-Agha} {\color{red}in particular} is {\color{violet}established} in Lemma \ref{lem: continuity of q(w,z)}. \cite[Theorem 2.1-condition (iii)]{Rassoul-Agha} is {\color{red}verified} in Lemma \ref{lem: kkkk}. \cite[Theorem 2.1-condition (iv)]{Rassoul-Agha} follows from {\color{red}our} mixing conditions {\color{red}by an argument analogous} to \cite[Lemma 5.4]{Rassoul-Agha}. \cite[Theorem 2.1-condition (v)]{Rassoul-Agha} follows from \cite[Lemma 5.5]{Rassoul-Agha}, {\color{red}whose} proof {\color{violet}is independent of} the {\color{red}underlying} mixing assumptions. {\color{red}Henceforth} all prerequisites of \cite[Theorem 2.1]{Rassoul-Agha} are satisfied {\color{red}under} our scenario, {\color{red}and the} {\color{violet}desired} annealed LDP {\color{red}follows immediately.}
    \end{proof}

\subsection{Shape of the zero-sets of the quenched and annealed rate functions.}
    Another question worthy {\color{violet}of investigation} concerns the shape of the zero sets where the quenched and annealed rate functions $I_q$, $I_a$ vanish. Under a slightly different mixing scenario,  {\color{red}this} question {\color{red}was addressed} in \cite[Remark 2.4]{Rassoul-Agha}. The same {\color{red}line of reasoning applies in our setting,} and we {\color{red}detail it below} for completeness.
    \begin{proof}[\textbf{Proof of Theorem \ref{thm: shape of zero set}}.]
        Let $\mu,\lambda\in\mathcal{E}$ be two distinct $\vartheta^*$-ergodic measures for $q(\vdot)$ {\color{red}such that} $\langle\mu\rangle\neq0$ {\color{red}and} $\langle\lambda\rangle\neq0$. Invoking Lemma \ref{lem: kkkk} {\color{red}alongside} \cite[Theorem 4.2]{Varadhan2} (whose proof {\color{violet}does} not rely on i.i.d.\ structure of $\mathbb{P}$), we conclude that $\exists~C>0$ s.t. $\langle\mu\rangle=-C\langle\lambda\rangle$. Now, if $\exists~\mu,\lambda\in\mathcal{E}$ with $\langle\mu\rangle\neq0$ and $\langle\lambda\rangle\neq0$, then by the above arguments we can find $C>0$ with $\langle\mu\rangle=-C\langle\lambda\rangle$. If furthermore $\not\exists~\mu^\prime\in\mathcal{E}$ s.t. $\langle\mu^\prime\rangle=C^\prime\langle\mu\rangle$ with $C^\prime>1$ and $\not\exists~\lambda^\prime\in\mathcal{E}$ s.t. $\langle\lambda^\prime\rangle=C^{\prime\prime}\langle\lambda\rangle$ with $C^{\prime\prime}>1$, then the zero sets $\{I_q=0\}$ and $\{I_a=0\}$ {\color{red}coincide and form} the line segment {\color{red}passing} through the origin {\color{red}that connects} $\langle\mu\rangle$ and $\langle\lambda\rangle$. {\color{red}Otherwise,} if $\exists!~\nu\in\mathcal{E}$ with $\langle\nu\rangle\neq0$, then both $I_a,I_q$ vanish {\color{red}solely} at the single point $\nu\in\mathbb{D}$. These {\color{red}constitute} the only two possibilities {\color{red}(cf.\ }\cite[Remark 2.4]{Rassoul-Agha}).
    \end{proof}

\section{Threshold of disorder at the boundary of the velocity surface}\label{sec: chapter 4}
    When $\mathbb{P}$ {\color{red}satisfies condition} $\textbf{(SMF)}_{C,g}$, the existence of a nontrivial threshold window of disorder is established for any compact subset along the boundary $\partial\mathbb{D}\backslash\partial\mathbb{D}_{d-2}$ of the velocity surface. Below this threshold, the {\color{red}annealed and quenched} rate functions $I_a$ {\color{violet}and} $I_q$ {\color{red}necessarily} match. The key {\color{red}technical ingredient} is an auxiliary random walk that disentangles the intertwined information from the mixing environment.

\subsection{Constructing an auxiliary random walk on $(d-1)$-dimensional hyperplane.}\label{sec: section 4.2}
    For {\color{red}$\forall~s\in\{\pm1\}^d$}, let $\mathbb{V}_s\coloneqq\big\{s_ie_i:\,1\leq i\leq d\big\}\subseteq\mathbb{V}$ {\color{red}denote the set of} all $s$-allowed jumps. {\color{red}Note that} $\mathbb{V}_s$ {\color{red}affinely generates} the unique hyperplane $\{x:\in\mathbb{R}^d\,s_1x_1+\cdots s_dx_d= 1\}$. Let $\pi:\mathbb{R}^d\to\mathbb{R}^d$ be the affine transform {\color{red}mapping this hyperplane} onto $\{x\in\mathbb{R}^d,\,x_d=0\}\simeq\mathbb{R}^{d-1}$, {\color{red}defined uniquely} by $\pi(s)\coloneqq(d-1)e_d/d$, $\pi(s_de_d)\coloneqq-\sum_{j<d}e_j$ and $\pi(s_je_j)\coloneqq e_j$ for all $j<d$. {\color{red}Constructed} from the trace of $(X_n)_{n\geq0}$, the projected process $(S_n)_{n\geq0}$ is {\color{red}defined} by $S_0=0$ and $S_n\coloneqq\sum_{j=1}^n\pi(X_j-X_{j-1})$, $\forall~n\geq1$. {\color{red} For any $0 \le m < n < \infty$, we define} the event {\color{red}restricting} $(X_n)_{n\geq0}$ to {\color{red}take jumps exclusively} in $\mathbb{V}_s$ {\color{red}between times $m$ and $n$} as $\mathbb{B}_{m,n}\coloneqq\{X_j-X_{j-1}\in\mathbb{V}(s),\,\forall~m< j\leq n\}$. {\color{red}For brevity,} we write $\mathbb{B}_n\coloneqq \mathbb{B}_{0,n}$, {\color{red}for all} $n\geq1$.\par

    To disentangle the correlated {\color{red}lattice} coordinates of the mixing environment, we introduce an auxiliary random walk {\color{red}under} the {\color{red}product} probability measure $\overline{\mathbb{P}}^x \coloneqq \mathbb{P}\otimes Q\otimes\overline{P}^x_{\omega,\epsilon}$ on the {\color{red}state} space $\Omega\times(\mathcal{W})^{\mathbb{N}}\times(\mathbb{Z}^d)^{\mathbb{N}}$, where $\mathcal{W}\coloneqq\mathbb{V}\cup\{0\}$. The {\color{red}i.i.d.\ }product measure $Q$ on $(\mathcal{W})^{\mathbb{N}}$ {\color{red}governs the decoupling sequence} $\epsilon=(\epsilon_1,\epsilon_2,\ldots)\in(\mathcal{W})^{\mathbb{N}}$ {\color{red}with} $Q(\epsilon_1=\pm e_i)=\kappa$, {\color{red}$\forall~i=1,\ldots,d,$}  and $Q(\epsilon_1=0)=1-2d\kappa$. {\color{red}Conditioned on $(\omega, \epsilon)$,} the auxiliary random walk {\color{red}$(X_n)_{n \ge 0}$} starting {\color{red}at} $x\in\mathbb{Z}^d$ {\color{red}is driven by the quenched law $\overline{P}^x_{\omega,\epsilon}$, with} transition probabilities
    \[ 
        \overline{P}^x_{\omega,\epsilon}(X_{n+1}=X_n+e|X_n)= \Bar{\omega}_{n+1}(X_n,e) \coloneqq \mathbbm{1}_{\{\epsilon_{n+1}=e\}}+\frac{\mathbbm{1}_{\{\epsilon_{n+1}=0\}}}{1-2d\kappa}\big(\omega(X_n,e)-\kappa\big),\qquad\forall~e\in\mathbb{V}.
    \]
    {\color{red}Integrating out $Q$ shows that }the {\color{red}quenched} law of $(X_n)_{n\geq0}$ under {\color{red}$Q \otimes \overline{P}^x_{\omega,\epsilon}$ coincides with} $P_{x,\omega}$. {\color{red}Consequently,} the {\color{red}averaged} law under $\overline{\mathbb{P}}^x$ {\color{red}equals} $P_x$. Throughout this {\color{red}section}, we {\color{red}denote expectations by} $\overline{E}^x_{\omega,\epsilon}\coloneqq E_{\overline{P}^x_{\omega,\epsilon}}$ and $\overline{\mathbb{E}}^x\coloneqq E_{\overline{\mathbb{P}}^x}$, {\color{red}respectively}. For {\color{red}a fixed parameter} $L>r$ (where $\mathbb{P}$ is assumed {\color{red}to be} an $r$-Markovian {\color{red}random} field) we introduce the $Q$-stopping sequence $(\tau^{(L)}_n)_{n\geq0}$ {\color{red}iteratively} by $\tau^{(L)}_0=0$ and
    \[
        \tau^{(L)}_{n}\coloneqq\inf\{j\geq\tau^{(L)}_{n-1}+L:\,(\epsilon_{j-L},\ldots,\epsilon_{j-1}) = +s,\ldots,s\}, \qquad\forall~n\geq1.
    \]
    {\color{red}To track the environmental information generated by this sequence,} we introduce the $\sigma$-algebras $(\mathscr{G}_n)_{n\geq0}$ on $\Omega\times(\mathcal{W})^{\mathbb{N}}$ by $\mathscr{G}_0\coloneqq(\omega_y:\,\langle y,s\rangle\leq\abs{s}_1-L/d^{1/2})$ and\footnote{Here we use $\langle u,v\rangle\coloneqq\sum_{i=1}^d u_iv_i$ to denote the {\color{red}standard Euclidean inner} product of any vectors $u,v\in\mathbb{R}^d$. } for each $n\geq1$, $\mathscr{G}_n\coloneqq\sigma(\tau_1^{(L)},\ldots,\tau_n^{(L)},\,\omega_y:\,\langle y, s\rangle<\tau^{(L)}_n-\abs{s}_1-L/d^{1/2},\,\epsilon_i:\,{i\leq\tau_n^{(L)}})$. We also {\color{red}introduce} the localized $\sigma$-algebras by $\mathscr{F}^{(L)}_{k}\coloneqq\sigma(\omega_y:\,\langle y, s\rangle\leq k+\abs{s}_1-L/d^{1/2}+,\,\epsilon_i:\,i\leq k)$, $\forall~k\geq1$. The stopping sequence $(\tau_n^{(L)})_{n\geq1}$ {\color{red}exploits the auxiliary} $Q$-randomness to {\color{red}enforce} an $L$-spacing to separate the {\color{red}lattice} correlat{\color{red}ion of the} environment $\mathbb{P}$. {\color{red}This mechanism serves a purpose similar to that of} the renewal sequence in Section \ref{sec: 5.1} which {\color{red}likewise} explores new territory at each renewal time. The environmental $\sigma$-algebras $(\mathscr{G}_n)_{n\geq1}$ {\color{red}represent the accumulated environmental} information {\color{red}up to the $n$th stopping time} $\tau^{(L)}_n$, {\color{red}whereas} the localized $\sigma$-algebra $(\mathscr{F}^{(L)}_k)_{k\geq1}$ {\color{red}captures} the information {\color{red}revealed up to time $k$ on} the event $\{\tau^{(L)}_n=k\}$ for some $n\geq1$. The following lemma {\color{red}confirms} that the sequence $(\tau^{(L)}_n)_{n\geq0}$ {\color{red}effectively} disentangle the intertwined information {\color{red}inherent in} the mixing environment.

    \begin{lemma}\label{lem: exponential mixing inequality}
        \normalfont
        When $\mathbb{P}$ {\color{red}satisfies} $\textbf{(SMF)}_{C,g}$, for any bounded $\sigma({\color{red}(X_n)_{n\in\mathbb{N}}})$-measurable {\color{red}function} $f:(\mathbb{Z}^{d-1})^{\mathbb{N}}\to\mathbb{R}$ {\color{red}and for each $n \ge 1$}, $\mathbb{P}\otimes Q$-a.s. we have
        \[ 
        \overline{\mathbb{E}}^0\big[f(S^{\tau,1}_{\vdot})I_{\{\mathbb{B}(\tau_1^{(L)})\}} \big] = E_{\mathbb{P}\otimes Q}\big[\,\overline{E}^0_{\omega,\epsilon}\big(f(S^{\tau,n}_{\vdot})I_{\{\mathbb{B}(\tau_n^{(L)})\} }\big|\mathbb{B}_{\tau^{(L)}_{n-1}}\big)\big|\mathscr{G}_{n-1}\big],\qquad\forall~L>L_0,
        \]
        where $(S_i-S(\tau^{(L)}_{n-1}):\,\tau^{(L)}_{n-1}\leq i\leq\tau^{(L)}_n)$ by $S^{\tau,n}_{\vdot}$ {\color{red}denotes the path increment process during the $n$th interval.}
    \end{lemma}
    \begin{proof}
        We have $\mathbb{E}E_{Q}[\,\overline{E}^0_{\omega,\epsilon}(f(S^{\tau,n}_{\vdot})I_{\{\mathbb{B}(\tau^{(L)}_n) \} }|\mathbb{B}_{\tau^{(L)}_{n-1}})h] = \sum_{k\in\mathbb{N}} \mathbb{E} E_{ Q}[\,\overline{E}^0_{\omega,\epsilon}(f(S^{\tau,n}_{\vdot})I_{\{ \mathbb{B}(\tau^{(L)}_n) \} }|\mathbb{B}_{\tau^{(L)}_{n-1}})h_k,\,\tau^{(L)}_{n-1}=k]$ for any bounded $\mathscr{G}_{n-1}$-measurable {\color{red}function} $h:\Omega\times(\mathcal{W})^{\mathbb{N}}\to\mathbb{R}$, because on the {\color{red}localization} event $\{\nu^{(L)}_{n-1}=k\}$, we can find bounded $\mathscr{F}^{(L)}_{k}$-measurable $h_{k}$ which {\color{red}equals} $h$ on this event. Then, we observe that
        \begin{equation*}\begin{aligned}
            &\mathbb{E}E_{ Q}\big[\,\overline{E}^0_{\omega,\epsilon}(f(S^{\tau,n}_{\vdot})I_{ \{ \mathbb{B}(\tau^{(L)}_n) \} }|\mathbb{B}_{\tau^{(L)}_{n-1}})h\big] = \sum_{k\in\mathbb{N}} \mathbb{E}E_{ Q} \big[I_{\{\tau^{(L)}_{n-1}=k\}}h_{k}  \mathbb{E}E_{ Q}\big[\,\overline{E}^0_{\omega,\vartheta_k\epsilon}\big(f(S_{k+i}-S_k,\,i\leq \tau^{(L)}_1)I_{ \{ \mathbb{B}(k+\tau^{(L)}_1) \} }\big|\mathbb{B}_k\big)\big|\mathscr{F}^{(L)}_{k}\big]\big]\\
            &\qquad = \sum_{k\in\mathbb{N}} \mathbb{E}E_{ Q} \big[I_{\{\tau^{(L)}_{n-1}=k\}}h_{k} \mathbb{E}E_Q \overline{E}^0_{\omega,\epsilon}\big(f(S_{i},\,i\leq \tau^{(L)}_1)I_{ \{ \mathbb{B}(\tau^{(L)}_1) \} }\big)\big] = \overline{\mathbb{E}}^0[h] \overline{\mathbb{E}}^0\big[f(S^{\tau,1}_{\vdot})I_{\{\mathbb{B}(\tau_1^{(L)})\}} \big],
        \end{aligned}\end{equation*}
        {\color{red}by the strong Markov property with respect to} the shift $\vartheta_k$ on $(\mathcal{W})^{\mathbb{N}}$. The last equality is {\color{red}a consequence} of $\textbf{(SMF)}_{C,g}$ {\color{red}where any lattice correlations across the $L$-spacing} vanish {\color{violet}whenever} $L>L_0$. Since $h$ is chosen arbitrarily, the assertion {\color{red}follows}.  
    \end{proof}

\subsection{Equality between quenched and annealed logarithmic moment generating functions.}
    The comparison between the quenched and annealed log-MGFs relies on estimating the limiting {\color{red}behavior} of $\mathcal{L}^x_{n,\theta}(\omega,\epsilon)\coloneqq \overline{E}^x_{\omega,\epsilon}[e^{\langle\theta,S_n,\rangle},\,\mathbb{B}_n]/\overline{\mathbb{E}}^x[e^{\langle\theta,S_n,\rangle},\,\mathbb{B}_n]$, $\forall~\theta\in\mathbb{R}^{d-1}$, $x\in\mathbb{Z}^d$ and $n\geq1$. We {\color{red}introduce the sampled sequence}
    \[
        \mathcal{H}_{n,\theta}(\omega,\epsilon)\coloneqq  \mathcal{L}^x_{\tau^{(L)}_n,\theta}(\omega,\epsilon) = \frac{\overline{E}^x_{\omega,\epsilon}[e^{\langle\theta,S(\tau^{(L)}_n)\rangle},\,\mathbb{B}(\tau^{(L)}_n)]}{\overline{\mathbb{E}}^x[e^{\langle\theta,S(\tau^{(L)}_n)\rangle},\,\mathbb{B}(\tau^{(L)}_n)]},\qquad\forall~n\geq1.
    \]
    With $L\geq L_0$, Lemma \ref{lem: exponential mixing inequality} {\color{red}ensures} $\mathbb{E}E_Q[\mathcal{H}_{n+1,\theta}(\omega,\epsilon)|\mathscr{G}_{n}]=\mathbb{E}E_Q[\mathcal{L}^x_{\tau^{(L)}_{n+1},\theta}|\mathscr{G}_{n}] = \mathcal{H}_{n,\theta}$, $\forall~n\in\mathbb{N}$. We {\color{red}whence} call $(\mathcal{H}_{n,\theta})_{n\geq1}$ an \textit{approximate martingale}. The nonnegativity of this martingale then {\color{red}yields} $\mathcal{H}_{n,\theta}\to\mathcal{H}_{\infty,\theta}$ as $n\to\infty$, $\mathbb{P}\otimes Q$-a.s., together with $\mathbb{E}E_Q[\mathcal{H}_{n,\theta}] =\mathbb{E}E_Q[\mathcal{H}_{1,\theta}]=1$, $\forall~n\geq1$.

    \begin{lemma}\label{lem: square integrable of the martingale H}
        \normalfont
        When $\mathbb{P}$ {\color{red}satisfies} $\textbf{(SMF)}_{C,g}$, for any compact {\color{red}set} $\Theta\subseteq\mathbb{R}^{d-1}$, {\color{red}there exists some} $\epsilon_0 = \epsilon_0(\kappa,\Theta) > 0$ {\color{red}such that} $\sup_{n\in\mathbb{N},\theta\in\Theta}\normx{\mathcal{H}_{n,\theta}}_{L^2(\mathbb{P}\otimes Q)} < \infty$ whenever $\text{dis}(\mathbb{P})<\epsilon_0$.
    \end{lemma}
    \begin{proof}
        Choose the random sequence $(\Vec{X}^{(\theta)}_j)_{j\geq1}$ {\color{red}whose} increments {\color{red}are} i.i.d.\ {\color{red}copies of} $\Vec{X}^{(\theta)}_1=(X^{(\theta)}_{1,1},\ldots,X^{(\theta)}_{1,\tau^{(L)}_1})\in\cup_{k=1}^\infty(\mathbb{Z}^d)^k$ {\color{red}with probability distribution $\hat{P}^{(\theta)}_0$ defined by} 
        \[
            \hat{P}^{(\theta)}_0(\Vec{S}^X_1=s) = e^{\langle\theta,s_k\rangle}\frac{\overline{\mathbb{P}}^0(\mathbb{B}(\tau^{(L)}_1),\,(S_1,\ldots,S(\tau^{(L)}_1))=s)}{\mathbb{E}E_Q[e^{\langle\theta,S(\tau^{(L)}_1)\rangle},\,\mathbb{B}(\tau^{(L)}_1) ]},\qquad\forall~s\in\mathcal{S},
        \]
        where $\Vec{S}^X_1\in\cup_{k=1}^\infty(\mathbb{Z}^{d-1})^k$ {\color{red}stands for} the $\pi$-transform of $\Vec{X}^{(\theta)}_1$ and $\mathcal{S}$ {\color{red}denotes the set of} all finite paths in $\mathbb{Z}^{d-1}$. Let $(\Vec{Y}^{(\theta)}_j)_{j\geq1}$ be an independent copy of $(\Vec{X}^{(\theta)}_j)_{j\geq1}$. {\color{red}For each $j \ge 1$, we define}
        \[
            \mathscr{T}^X_j\coloneqq \prod_{k=\tau^{(L)}_j+1}^{\tau^{(L)}_{j+1}}\Bar{\omega}_{\tau^{(L)}_j+k}(X_{j,k-1},X_{j,k}-X_{j,k-1})\quad\text{and {\color{red}the interaction potential}}\;\;V(\Vec{X}^{(\theta)}_j,\Vec{Y}^{(\theta)}_j)\coloneqq \log \frac{ \mathbb{E}E_Q[\mathscr{T}^X_j\mathscr{T}^Y_j] }{E_Q[ \mathbb{E}\mathscr{T}^X_j \mathbb{E}\mathscr{T}^Y_j ]},
        \]
        {\color{red}where $\mathscr{T}^Y_j$ is defined analogously for $(\vec{Y}^{(\theta)}_j)_{j \ge 1}$ as $\mathscr{T}^X_j$ for $(\Vec{X}^{(\theta)}_j)_{j\geq1}$. By the above construction, the second moment admits the representation}  $\mathbb{E}E_Q[\mathcal{H}_{n,\theta}^2] = \hat{E}^{(\theta)}_0[\exp \sum_{j<n} V(\Vec{X}^{(\theta)}_j,\Vec{Y}^{(\theta)}_j)]$, $\forall~n\geq1$ and $\theta\in\mathbb{R}^{d-1}$. The nature of the $\textbf{(SMF)}_{C,g}$ condition {\color{red}guarantees} $\hat{E}^{(\theta)}_0[e^{\sum_{j<n} V(\Vec{X}^{(\theta)}_j,\Vec{Y}^{(\theta)}_j)}] = \hat{E}^{(\theta)}_0 [ e^{\sum_{j<n} V(\Vec{X}^{(\theta)}_j,\Vec{Y}^{(\theta)}_j)\mathbbm{1}\{ d_1(\Vec{X}^{(\theta)}_j,\Vec{Y}^{(\theta)}_j)<L\} }]$. Henceforth,
        \[
            \hat{E}^{(\theta)}_0\big[e^{\sum_{j<n} V(\Vec{X}^{(\theta)}_j,\Vec{Y}^{(\theta)}_j)}\big] \leq \hat{E}^{(\theta)}_0 \big[ e^{\sum_{j=1}^\infty \text{dis}(\mathbb{P}) N^{L}_j(\Vec{X}^{(\theta)},\Vec{Y}^{(\theta)}) \mathbbm{1}\{ d_1(\Vec{X}^{(\theta)}_j,\Vec{Y}^{(\theta)}_j)<L\} }\big],\qquad\forall~n\geq1,
        \]
        where $N^{L}_j(\Vec{X}^{(\theta)},\Vec{Y}^{(\theta)}) \coloneqq\sum_{k=\tau^{(L)}_{j}+1}^{\tau^{(L)}_{j+1}}\mathbbm{1}_{\{|X^{(\theta)}_{j,k}-X^{(\theta)}_{j,k}|_1<L\}}$, $\forall~j\geq1$, {\color{red}counts the number of close contacts during the $j$th block.} We claim {\color{red}that} for each $\ell\in\mathbb{N}$, {\color{red}there exists a} constant $\Bar{C}\coloneqq \Bar{C}(\Theta)>0$ {\color{red}such that} for any $x,y\in\mathbb{Z}^{d}$,
        \begin{equation}\label{eqn: left to prove in step III}
            \sup_{\theta\in\Theta}\sum_{j=1}^\infty \hat{P}^{(\theta)}_{x,y}\big(d_1(\Vec{X}^{(\theta)}_j,\Vec{Y}^{(\theta)}_j)\leq \ell\big)\leq \sup_{\theta\in\Theta} \sum_{j=1}^\infty \hat{E}^{(\theta)}_{x,y} \big[ N^{\ell}_j(\Vec{X}^{(\theta)},\Vec{Y}^{(\theta)})\mathbbm{1}_{\{d_1(\Vec{X}^{(\theta)}_j,\Vec{Y}^{(\theta)}_j)\leq \ell\}} \big]  \leq \Bar{C}\ell^2
        \end{equation}
        Assuming for now (\ref{eqn: left to prove in step III}), {\color{red}an expansion in powers of $\mathrm{dis}(\mathbb{P})$ yields}
        \begin{equation*}\begin{aligned}
            &\hat{E}^{(\theta)}_0\big[ e^{\sum_{j<n} V(\Vec{X}^{(\theta)}_j,\Vec{Y}^{(\theta)}_j)\mathbbm{1}\{ d_1(\Vec{X}^{(\theta)}_j,\Vec{Y}^{(\theta)}_j)<L\} }\big] \leq  \hat{E}^{(\theta)}_0\big[ e^{\sum_{j=1}^\infty \text{dis}(\mathbb{P}) N^{L}_j(\Vec{X}^{(\theta)},\Vec{Y}^{(\theta)}) \mathbbm{1}\{ d_1(\Vec{X}^{(\theta)}_j,\Vec{Y}^{(\theta)}_j)<L\} }\big]\\
            &\qquad \leq  \sum_{n=0}^\infty  \frac{2^n\text{dis}(\mathbb{P})^n}{n!} \hat{E}^{(\theta)}_0 \big[   \big(\sum_{j=1}^\infty N^{L}_j(\Vec{X}^{(\theta)},\Vec{Y}^{(\theta)})\mathbbm{1}_{\{ d_1(\Vec{X}^{(\theta)}_j,\Vec{Y}^{(\theta)}_j)<L \}}\big)^n \big]\leq \frac{1}{1-\Bar{C}L^2\text{dis}(\mathbb{P})}<\infty,\qquad\forall~\theta\in\Theta,
        \end{aligned}\end{equation*}
        {\color{red}provided} $\text{dis}(\mathbb{P})<1/\Bar{C}L^2$, {\color{red}} the {\color{red}final} inequality follows {\color{red}by standard combinatorics} \cite[Lemma A.2]{Chen} from (\ref{eqn: left to prove in step III}). It {\color{red}remains} to prove (\ref{eqn: left to prove in step III}). Let $(X^*_j,Y^*_j)\coloneqq\arg\min\{|X^{(\theta)}_i-Y^{(\theta)}_i|:\,\tau^{(L)}_{j-1}\leq i\leq\tau^{(L)}_{j}-1\}$ and {\color{red}set} $Z^*_j\coloneqq X^*_j-Y^*_j$. Then, we observe that $\sum_{j=1}^\infty\hat{P}^{(\theta)}_{x,y} (d_1(\Vec{X}^{(\theta)}_j,\Vec{Y}^{(\theta)}_j)\leq\ell) = \sum_{j=1}^\infty\hat{P}^{(\theta)}_{x,y} (Z_j^*\in \Bar{B}_\ell)$, where $\Bar{B}_\ell\coloneqq\{z\in\mathbb{R}^{d-1}:\,\abs{z}_1\leq\ell\}$. {\color{red}Let} $\mu$ {\color{red}denote} the law of $Z^*_1$. {\color{red}By inclusion of events,} $\sum_{j=1}^\infty\hat{P}^{(\theta)}_{x,y}(Z_j^*\in \Bar{B}_\ell) \leq \sum_{j=1}^\infty\mu^{*j}(\Bar{B}_\ell)$, {\color{red}where} $\mu^{*j}$ {\color{red}is} the $j$-fold convolution of $\mu$. Consider the function $\prod_{i=1}^d f(\ell^{-1}x_i)$, where $f(x_i)=\max\{1-\abs{x_i},0\}$. {\color{red}By} the Fourier inversion formula \cite[Section 5.1]{Stein}, 
        \[
            \widehat{\prod_{i=1}^{d-1}f(x_i)} = \prod_{i=1}^{d-1}\widehat{f}(\xi_i)\qquad\Longrightarrow\quad \int_{\mathbb{R}^{d-1}}\widehat{\prod_{i=1}^{d-1}f(\ell^{-1}x_i)}\,d\mu^{*j}(x) = \ell^{d-1}\int_{\mathbb{R}^{d-1}} \prod_{i=1}^{d-1} f(\ell\xi_i)\chi_\mu^{(\theta)}(\xi)^n\,d\xi,
        \]
        where $\widehat{f}(\xi_i)=2(1-\cos\xi_i)/\xi_i^2$ and $\chi_\mu^{(\theta)}(\xi)\coloneqq E_{\hat{P}^{(\theta)}_x\otimes\hat{P}^{(\theta)}_y}[\exp\langle i\xi,Z^*_1\rangle]$ {\color{red}is} the characteristic function of $Z^*_1$. {\color{red}Summing over $j \geq 1$ with a discount factor $0<a <1$ gives}
        \[
            \int_{\mathbb{R}^{d-1}}\widehat{\prod_{i=1}^{d-1}f(\ell^{-1}x_i)}\sum_{j\in\mathbb{N}}a^n\,d\mu^{*j}(x) = \ell^{d-1} \int_{\mathbb{R}^{d-1}} \frac{\prod_{i=1}^{d-1}f(\ell x_i)}{1-a\chi^{(\theta)}_\mu(\xi)}\,d\xi.
        \]
        Hence $\sum_{j=0}^\infty\mu^{*j}(\Bar{B}_\ell)\leq C_d \int_{\mathbb{R}^{d-1}}\widehat{\prod_{i=1}^{d-1}f(\ell^{-1}x_i)}\sum_{j=0}^\infty a^n\,d\mu^{*j}(x) \leq C_d \ell^{d-1} \sup_{0<a<1}\int_{\mathbb{R}^{d-1}} \prod_{i=1}^{d-1}f(\ell x_i)/(1-a\chi^{(\theta)}_\mu(\xi))\,d\xi$, which implies $\sum_{j=1}^\infty\hat{P}^{(\theta)}_{x,y}(d_1(\Vec{X}^{(\theta)}_j,\Vec{Y}^{(\theta)}_j)\leq\ell) \leq C_d \ell^{d-1} \int_{\Bar{B}_{1/\ell}} 1/(1-\chi^{(\theta)}_\mu(\xi))\,d\xi$ {\color{red}by taking $a \nearrow 1$}. {\color{red}Since} $|Z^*_1|$ is bounded by $\tau^{(L)}_1$ {\color{red}uniformly for} $\theta\in\Theta$, the Taylor expansion {\color{red}yields} $\chi_\mu^{(\theta)}(\xi)\leq 1-\frac{1}{2}\sum_{j,k=1}^{d-1} a^{(\theta)}_{L,jk}\xi_j\xi_k +C(L)\abs{\xi}^3$ for some $C(L)>0$ independent of $\theta$, where $(a^{(\theta)}_{L,jk})_{j,k}$ is the covariance matrix of $Z^*_1$. Since this process $(Z^*_n)_{n\in\mathbb{N}}$ has effective dimension $d-1$, the matrix $(a^{(\theta)}_{L,jk})_{j,k}$ is positive-definite for each $\theta\in\Theta$. Thus for any compact $\Theta\subseteq\mathbb{R}^{d-1}$, there exists $\ell_0=\ell_0(L,\Theta)\in\mathbb{N}$ and $c_0=c_0(L,\Theta)>0$ such that $c_0\abs{\xi}^2\leq1-\chi^{(\theta)}_\mu(\xi)$ for {\color{red}all} $\abs{\xi}\leq1/\ell_0$. Therefore, $\forall~\ell\geq\ell_0(L,\Theta)$,     
        \[
            \sum_{j=1}^\infty\hat{P}^{(\theta)}_{x,y}\big(Z_j^*\in \Bar{B}_\ell\big) \leq C_d\frac{\ell^{d-1}}{c_0(L,\Theta)}\int_{\Bar{B}_{1/\ell}}\frac{1}{\abs{\xi}^2}\,d\xi\leq C^\prime(\Theta)\ell^{d-1}\int_0^{1/\ell}r^{d-4}\,dr\leq C^\prime(\Theta)\ell^2.
        \]
        For any $1\leq\ell\leq\ell_0-1$, this inequality holds {\color{red}by monotonicity:} $\sum_{j=1}^\infty\hat{P}^{(\theta)}_{x,y}(Z_j^*\in \Bar{B}_\ell) \leq \sum_{j=1}^\infty\hat{P}^{(\theta)}_{x,y}(Z_j^*\in \Bar{B}_{\ell_0}) \leq C^\prime(\Theta)(\ell_0^2/\ell^2)\ell^2$ which is less than or equal to $ C^{\prime\prime}(\Theta)\ell^2$, where $C^{\prime\prime}\coloneqq C^\prime(\Theta)\ell_0^2$. {\color{red}Finally,}
        \[
            \sum_{j=1}^\infty E_{\hat{P}^{(\theta)}_x\otimes\hat{P}^{(\theta)}_y}\bigg[ N^{\ell}_j(\Vec{X}^{(\theta)},\Vec{Y}^{(\theta)})\mathbbm{1}_{\{d_1(\Vec{X}^{(\theta)}_j,\Vec{Y}^{(\theta)}_j)\leq \ell\}} \bigg] \leq \sum_{j=1}^\infty E_Q\big[\tau^{(L)}_1\big] \hat{P}^{(\theta)}_x\otimes\hat{P}^{(\theta)}_y\big(Z_j^*\in \Bar{B}_\ell\big)\leq C(L,\Theta)\ell^2,\quad{\color{red}\forall~\ell\in\mathbb{N},~\theta\in\Theta,}
        \]
        verifying (\ref{eqn: left to prove in step III}), which is the only assertion left above. 
    \end{proof}
    \begin{remark}\label{rmk: remark subsequent to lemma 4.2}
        The crucial estimate (\ref{eqn: left to prove in step III}) {\color{red}bounds} the {\color{red}expected hitting} time {\color{red}spent by the} space-time random walk in a ball of radius $\ell$ {\color{red}by} $C\ell^2$. This {\color{red}scale reflects the transience of the} $(d-1)$-dimensional random walk {\color{red}for} $d\geq4$. The same {\color{red}Fourier analytical} approach applies to lower-dimensional {\color{red}projections onto subspaces of} dimension $d-k$ {\color{red}whenever} $k>d-2$ where the {\color{red}transience is preserved}. Hence {\color{red}the equality of} the {\color{red}rate functions asserted in} Theorem \ref{thm: equality of LDP at boundary} {\color{red}extends} to all compact {\color{violet}subsets along the boundary} $\partial\mathbb{D}\backslash\partial\mathbb{D}_{d-k}$ {\color{red}for} $k=d-2$.
    \end{remark}

    \begin{lemma}\label{lem: H is positive}
        \normalfont
        When $\mathbb{P}$ {\color{red}satisfies} $\textbf{(SMF)}_{C,g}$, {\color{red}for any compact} $\Theta\subseteq\mathbb{R}^{d-1}$, {\color{red}there exists} $\epsilon_0 = \epsilon_0(\kappa,\Theta) > 0$ {\color{red}such that} $\mathcal{H}_{\infty,\theta}>0$, $\mathbb{P}\otimes Q$-a.s., whenever $\theta\in\Theta$ and $\text{dis}(\mathbb{P})<\epsilon_0$.
    \end{lemma}
    \begin{proof}
        {\color{red}Notice} that $\mathcal{H}_{n,\theta} =\sum_{k\in\mathbb{N}}\sum_{x\in\mathbb{Z}^d} \overline{P}^0_{ \omega,\epsilon}(\mathbb{B}_k,\,\tau_1^{(L)}=k,\,S_k=x)e^{\langle\theta,x\rangle} \mathcal{H}_{n-1,\theta}(\vartheta_x\omega,\vartheta_k\epsilon)/\mathbb{E}E_Q[\mathcal{L}^0_{\tau^{(L)}_1,\theta}]$. {\color{red}By} Lemma \ref{lem: exponential mixing inequality}, the event $\{\mathcal{H}_{\infty,\theta}(\omega,\epsilon)=0\}$ is translation invariant, and whence $\mathbb{P}\otimes Q(\mathcal{H}_{\infty,\theta}=0)\in\{0,1\}$. {\color{red}On the other hand,} Lemma \ref{lem: square integrable of the martingale H} {\color{red}asserts that} $\forall~\Theta\subseteq\mathbb{R}^{d-1}$ compact, $\exists~\epsilon_0$ {\color{red}such that} $\sup_{n\in\mathbb{N},\,\theta\in\Theta}\normx{\mathcal{H}_{n,\theta}}_{L^2(\mathbb{P}\otimes Q)}<\infty$ when $\text{dis}(\mathbb{P})<\epsilon_0$. Hence this uniformly integrable martingale $(\mathcal{H}_{n,\theta})_{n\in\mathbb{N}}$ satisfies $\mathbb{E}E_Q[\mathcal{H}_{\infty,\theta}] = \lim_{n\to\infty} \mathbb{E}E_Q[\mathcal{H}_{n,\theta}] > 0$, yielding  $\mathbb{P}\otimes Q(\mathcal{H}_{\infty,\theta}>0)=1$. 
    \end{proof}
    And whence, taking the limit $n\to\infty$, we have
    \begin{equation}\label{eqn: first inequality of log L}
        -C e^{-gtL} \leq \varliminf_{n\to\infty}\frac{1}{n}\log\mathcal{L}^0_{\nu_n^{(L)},\theta}(\omega,\epsilon) \leq \varlimsup_{n\to\infty}\frac{1}{n}\log\mathcal{L}^0_{\nu_n^{(L)},\theta}(\omega,\epsilon) \leq C^\prime e^{-gt L},\quad\mathbb{P}\otimes Q\text{-a.s.}
    \end{equation}

    \begin{lemma}\label{lem: log L goes to zero}
        \normalfont
         When $\mathbb{P}$ {\color{red}satisfies} $\textbf{(SMF)}_{C,g}$, $\forall~\Theta\subseteq\mathbb{R}^{d-1}$ compact, $\exists~\epsilon_0 = \epsilon_0(\kappa,\Theta) > 0$ s.t. $\lim_{n\to\infty}\frac{1}{n} \log E_Q[\mathcal{L}^0_{n,\theta}(\omega,\vdot)]=0$, $\mathbb{P}$-a.s., whenever $\theta\in\Theta$ and $\text{dis}(\mathbb{P})<\epsilon_0$.
    \end{lemma}
    \begin{proof}
        {\color{red}By} the {\color{red}strong} law of large numbers, $\tau_n^{(L)}/n\to E_Q[\tau_1^{(L)}]<\infty$ as $n\to\infty$, $Q$-a.s. {\color{red}Define the random index} sequence $(k_n)_{n\geq1}$ such that $\tau^{(L)}_{k_n}\leq n < \tau^{(L)}_{k_n+1}$ for each $n\in\mathbb{N}$. It {\color{red}follows immediately that} $k_n/n\to1/E_Q[\tau_1^{(L)}]$ as $n\to\infty$, $Q$-a.s. By definition of $\mathcal{L}^0_{n,\theta}$, we have $\log\overline{E}^x_{\omega,\epsilon}[e^{\langle\theta,S_n\rangle},\,\mathbb{B}_n] \leq \abs{\theta}_2(\tau^{(L)}_{k_n+1} - \tau^{(L)}_{k_n}) + \log \overline{E}^0_{\omega,\epsilon}[e^{\langle\theta,S(\tau^{(L)}_{k_n})\rangle},\,\mathbb{B}_{\tau^{(L)}_{k_n}}]$ and
        \[
            \overline{\mathbb{E}}^0 \big[e^{\langle\theta,S_n\rangle},\,\mathbb{B}_n\big] \geq \mathbb{E}E_Q\big[e^{-\abs{\theta}_2(\tau^{(L)}_{k_n+1}-\tau^{(L)}_{k_n})}\overline{E}^0_{\omega,\epsilon}\big(e^{\langle\theta,S(\tau^{(L)}_{k_n})\rangle},\,\mathbb{B}_{\tau^{(L)}_{k_n}}\big)\big] = E_Q\big[e^{-\abs{\theta}_2\tau_1^{(L)}}\big] \overline{\mathbb{E}}^0\big[e^{\langle\theta,S(\tau^{(L)}_{k_n})\rangle},\,\mathbb{B}_{\tau^{(L)}_{k_n}}\big],\qquad\forall~n\geq1.
        \]
        {\color{red}Combining these estimates}, we therefore get
        \[
            \frac{1}{n}\log\mathcal{L}^0_{n,\theta} \leq \frac{\abs{\theta}_2}{n}(\tau^{(L)}_{k_n+1} - \tau^{(L)}_{k_n}) + \frac{1}{n}\log\mathcal{H}_{k_n}  - \frac{1}{n}\log E_Q\big[e^{-\abs{\theta}_2\tau^{(L)}_1}\big] - \frac{1}{n}\log\overline{\mathbb{E}}^0\big[e^{\langle\theta,S(\tau^{(L)}_1)\rangle},\,\mathbb{B}_{\tau^{(L)}_1}\big],\qquad\forall~n\geq1.
        \]
        {\color{red}Observe} that $(\tau^{(L)}_{k_n+1} - \tau^{(L)}_{k_n})/n = (\tau^{(L)}_{k_n+1}/k_n+1)(k_n+1/n) - (\tau^{(L)}_{k_n}/k_n)(k_n/n) \to0$ as $n\to\infty$, $\mathbb{P}\otimes Q$-a.s. {\color{red}Applying} Lemma~\ref{lem: H is positive}, we have $\varlimsup_{n\to\infty}\frac{1}{n}\log\mathcal{L}^0_{n,\theta} \leq 0$, $\mathbb{P}\otimes Q$-a.s., {\color{red}$\forall~\theta\in\Theta$.} Via {\color{red}analogous lower-bound} arguments, we {\color{red}have}
        \[
            \log\overline{E}^0_{\omega,\epsilon}\big[e^{\langle\theta,S_n\rangle},\,\mathbb{B}_n\big] \geq -\abs{\theta}_2(\tau^{(L)}_{k_n+1} - \tau^{(L)}_{k_n}) + \log\overline{E}^0_{\omega,\epsilon}\big[e^{\langle\theta,S(\tau^{(L)}_{k_n+1})},\,\mathbb{B}_{\tau^{(L)}_{k_n+1}}\big],\qquad\mathbb{P}\otimes Q\text{-a.s.,}\qquad\forall~n\geq1.
        \]
        And a similar upper-bound for $\overline{\mathbb{E}}^0[e^{\langle\theta,S_n\rangle},\,\mathbb{B}_n]$ gives $\varliminf_{n\to\infty}\frac{1}{n}\log\mathcal{L}^0_{n,\theta} \geq 0$, which yields $\lim_{n\to\infty}\frac{1}{n}\log\mathcal{L}^0_{n,\theta} = 0$, $\mathbb{P}\otimes Q$-a.s. Therefore, $\forall~\delta>0$, $\exists~N_1(\omega,\epsilon)\in\mathbb{N}$ such that for all $n\geq N_1(\omega,\epsilon)$, we have $-\delta<\frac{1}{n} \log\mathcal{L}^0_{n,\theta}(\omega,\epsilon) < \delta$, $\mathbb{P}\otimes Q$-a.s. Because $N_1(\omega,\epsilon)$ is $\mathbb{P}\otimes Q$-a.s. finite, $\exists~N_2(\omega)\in\mathbb{N}$ such that $\forall~n\geq N_2(\omega)$, $Q(-\delta n<\log\mathcal{L}^0_{n,\theta}(\omega,\vdot)<\delta n) > 1-\delta$, $\mathbb{P}$-a.s. Henceforth, for any $n\geq N_2(\omega)$, we have $-\delta(1+\abs{\theta}_2) < \frac{1}{n}E_Q[\log \mathcal{L}^0_{n,\theta}(\omega,\vdot)] < \delta(1+\abs{\theta}_2)$, $\mathbb{P}$-a.s. First let $n\to\infty$, and then let $\delta\to0$, 
        \[
            \varliminf_{n\to\infty} \frac{1}{n}\log E_Q\big[\mathcal{L}^0_{n,\theta}(\omega,\vdot)\big] \geq \lim_{n\to\infty}\frac{1}{n}E_Q\big[\log \mathcal{L}^0_{n,\theta}(\omega,\vdot)\big] = 0,\quad\mathbb{P}\text{-a.s.}
        \]
        {\color{red}Finally,} Jensen's inequality yields $\varlimsup_{n\to\infty}n^{-1}\log E_Q[\mathcal{L}^0_{n,\theta}(\omega,\vdot)]\leq0$ for all $\theta$, verifying the assertion.
    \end{proof}

\subsection{Differentiable mappings on the boundary of velocity surface.}
    {\color{red}In this section,} we prove the {\color{red}equality between the quenched and annealed} LDP {\color{red}rate functions} at the boundary $\partial\mathbb{D}$, {\color{red}provided sufficiently small disorder}.
    \begin{lemma}\label{lem: boundary LDP general existence}
        \normalfont
        When $\mathbb{P}$ {\color{red}satisfie} $\textbf{(SMF)}_{C,g}$, {\color{red}there exists an} event $\Omega_0\subseteq\Omega$ {\color{red}with} $\mathbb{P}(\Omega_0)=1$ {\color{red}such that} $\forall~\omega\in\Omega_0$ the {\color{red}free energy} limit $ -\infty < \Lambda_{q,\omega}(\theta)\coloneqq\lim_{n\to\infty}\frac{1}{n}\log E_{0,\omega}[e^{\langle\theta,S_n\rangle}I_{\mathbb{B}_n}] < \infty$ exists and is finite {\color{red}for all $\theta$}, with $\Lambda_{q,\omega}$ convex continuous on $\mathbb{R}^{d-1}$. If $y=\nabla\Lambda_{q,\omega}(\eta)$ for some $\eta\in\mathbb{R}^{d-1}$, then $\langle\eta,y\rangle - \Lambda_{q,\omega}(\eta) = \sup_{\theta\in\mathbb{R}^{d-1}} \big\{ \langle\theta,y\rangle - \Lambda_{q,\omega}(\theta) \big\} \eqqcolon \Lambda_{q,\omega}^*(y)$. Furthermore, $y$ is an exposed point of $\Lambda^*_{q,\omega}$ with $\eta$ its exposing hyperplane, i.e.~$\langle\eta,y\rangle - \Lambda^*_{q,\omega}(y) > \langle\eta,x\rangle - \Lambda^*_{q,\omega}(x)$, $\forall~x\neq y$. Moreover, {\color{red}this Fenchel--Legendre transform} $\Lambda^*_{q,\omega}$ is convex lower-semicontinuous on $\mathbb{R}^{d-1}$. And $\forall~F\subseteq\mathbb{R}^{d-1}$ closed and $\forall~G\subseteq\mathbb{R}^{d-1}$ open,
        \[
            \varlimsup_{n\to\infty}\tfrac{1}{n}\log P_{0,\omega}\big(\tfrac{1}{n}S_n\in F,\,\mathbb{B}_n\big) \leq -\inf_{x\in F}\Lambda^*_{q,\omega}(x),\qquad \varliminf_{n\to\infty}\tfrac{1}{n}\log P_{0,\omega}\big(\tfrac{1}{n}S_n\in G,\,\mathbb{B}_n\big) \geq -\inf_{x\in G\cap\mathcal{F}^\omega}\Lambda^*_{q,\omega}(x),
        \]
        where $\mathcal{F}^\omega$ denotes the set of exposed points of $\Lambda_{q,\omega}$. {\color{red}Analogously}, the annealed {\color{red}limiting log-moment generating function} $-\infty<\Lambda_a(\theta)\coloneqq \lim_{n\to\infty}\frac{1}{n}\log E_0[e^{\langle\theta,S_n\rangle}I_{\mathbb{B}_n}] < \infty$ exists and is finite $\forall~\theta\in\Theta$. The {\color{red}same analytic} properties {\color{red}hold} for {\color{red}$\Lambda_{a}$, with $P_{0,\omega}$, $\Lambda_{q,\omega}$, $\Lambda_{q,\omega}^*$, and $\mathcal{F}^\omega$ replaced everywhere by $P_0$, $\Lambda_a$, $\Lambda_a^*$, and $\mathcal{F}$, respectively.} 
    \end{lemma}
    \begin{proof}
        The existence {\color{red}and finiteness} of the quenched limit ${\color{red}\lim_n}\frac{1}{n}\log E_{0,\omega} [e^{\langle\theta,S_n\rangle}I_{\mathbb{B}_n}] {\color{red}=}\Lambda_{q,\omega}(\theta) $ simultaneously for all $\theta\in\mathbb{R}^{d-1}$, $\mathbb{P}$-a.s. follows directly from \cite[Theorems 2.2, 2.4]{Rassoul-Agha/Seppalainen}. Its {\color{red}uniform boundedness follows from the trivial geometric bound} $n^{-1}\log P_{0,\omega}[e^{\langle\theta,S_n\rangle},\,\mathbb{B}_n] \leq \abs{\theta}_2(d-1)$ for all $n\in\mathbb{N}$. The {\color{red}existence and finiteness of the} annealed limit $\Lambda_a(\theta)$ {\color{red}follow from} Theorem \ref{thm: existence of LDP} {\color{red}together with} Varadhan's lemma \cite[Theorem 4.3.1]{Dembo/Zeitouni}. {\color{red}Indeed,} $\forall~\gamma>1$, $\varlimsup_{n\to\infty}\frac{1}{n} \log E_{0} [e^{\langle\gamma\theta,S_n\rangle}I_{\mathbb{B}_n}] \leq \gamma\abs{\theta}_2$, {\color{red}which satisfies} the condition \cite[Eqn. (4.3.3)]{Dembo/Zeitouni}.  {\color{red}With the existence of} the quenched and annealed limiting LMGs {\color{red}established}, {\color{red}all} remaining assertions follow from a standard {\color{red}application of} the Gärtner--Ellis theorem \cite[Section 2.3]{Dembo/Zeitouni}. 
    \end{proof}

    \begin{lemma}\label{lem: transfer the compact sets}
        \normalfont
        When $\mathbb{P}$ {\color{red}satisfies} $\textbf{(SMF)}_{C,g}$, {\color{red}for any compact} $\mathcal{K} \subseteq \partial\mathbb{D}(s)\backslash\partial\mathbb{D}_{d-2}$ compact, {\color{red}there exists} $R_{\mathcal{K}} = R_{\mathcal{K}}(d, \kappa, \mathcal{K}) > 0 $ {\color{red}such that} $\pi(\mathcal{K}) \subseteq \{\nabla\Lambda_a(\theta) :\, \theta \in D_{R_{\mathcal{K}}}\}$ where $D_{R_{\mathcal{K}}}\coloneqq\{\theta\in\mathbb{R}^{d-1}:\,\abs{\theta}<R_{\mathcal{K}}\}$ .
    \end{lemma}
    \begin{proof}
        {\color{red}Consider the function} $\psi^s:\theta\in\mathbb{R}^{d-1}\mapsto\sum_{e\in\mathbb{V}_s}e^{\langle\theta,\pi(e)\rangle}\mathbb{E}[\omega(0,e)]$. {\color{red}By} Lemma \ref{lem: B0}, $\Lambda_a$ is differentiable on $\mathbb{R}^{d-1}$ and {\color{red}satisfies the gradient bounds} $(1-\text{dis}(\mathbb{P}))\nabla\log\psi^s(\theta)\leq\nabla\Lambda_a(\theta)\leq(1+\text{dis}(\mathbb{P}))\nabla\log\psi^s(\theta)$, $\forall~\theta\in\mathbb{R}^{d-1}$. Invoking \cite[Lemma 3.8]{Bazaes/Mukherjee/Ramirez/Sagliett}, we can find $R_{\mathcal{K}}>0$ {\color{red}such} that $(1-\text{dis}(\mathbb{P}))^{-1}\pi(\mathcal{K})\subseteq \{\nabla\log\psi^s(\theta) :\, \theta \in D_{R_{\mathcal{K}}}\}$. {\color{red}This inclusion directly yields} $\pi(\mathcal{K}) \subseteq \{\nabla\Lambda_a(\theta) :\, \theta \in D_{R_{\mathcal{K}}}\}$, verifying the assertion.
    \end{proof}

    \begin{proof}[Proof of Theorem \ref{thm: equality of LDP at boundary}]
        \textbf{Step I.} By Lemma \ref{lem: log L goes to zero}, $\forall~R>0$, $\exists~\epsilon_R = \epsilon_R(d,\kappa,R)>0$ {\color{red}so} that whenever $\text{dis}(\mathbb{P})<\epsilon_R$, {\color{red}we have $n^{-1}\log E_Q[\mathcal{L}^0_{n,\theta^\prime}(\omega,\vdot)]\to0$ as $n\to\infty$, $\mathbb{P}$-a.s.} for each $\theta^\prime\in D_R\coloneqq\{\theta\in\mathbb{R}^{d-1}:\,\abs{\theta}_1\leq R\}$. {\color{red}Then, fixing a countable dense subset $\Theta_R \subseteq D_R$,} there exists {\color{red}an} event $\Omega_R=\Omega_R(\mathbb{P},R)\subseteq\Omega_0$ {\color{red}with $\mathbb{P}(\Omega_R) = 1$} so that $\lim_{n}\frac{1}{n} \log E_Q[\mathcal{L}^0_{n,\theta}(\omega,\vdot)]=0$, $\forall~\theta\in\Theta_R$ and $\omega\in\Omega_R$. Henceforth, 
        \begin{equation}\label{eqn: L vanishes on countable subsets}
            \Lambda_{q,\omega}(\theta) = \Lambda_a(\theta) + \lim_{n\to\infty}\frac{1}{n} \log E_Q[\mathcal{L}^0_{n,\theta}(\omega,\vdot)] = \Lambda_a(\theta),\qquad\forall~\theta\in\Theta_R.
        \end{equation}
       Since$\Lambda_{q,\omega}$ is continuous on $D_R$ {\color{red}for each} $\omega\in\Omega_0$, the equality $\Lambda_{q,\omega}(\theta) = \Lambda_a(\theta)$ in (\ref{eqn: L vanishes on countable subsets}) {\color{red}extends to all} $\theta\in D_R$ {\color{red}given} $\omega\in\Omega_R\subseteq\Omega_0$ and $\text{dis}(\mathbb{P})<\epsilon_R$. {\color{red}Now,} for any compact $\mathcal{K} \subseteq \partial\mathbb{D}(s)\backslash\partial\mathbb{D}_{d-2}$, Lemma, \ref{lem: transfer the compact sets} {\color{red}yields a} $R_{\mathcal{K}}$ {\color{red}such that} $\pi(\mathcal{K})\subseteq\{\nabla\Lambda_a(\theta):\,\theta\in D_{R_{\mathcal{K}}}\}$. {\color{red}Define} the open set $ \mathcal{O}_{R_{\mathcal{K}}} \coloneqq\pi^{-1} \{ \nabla\Lambda_a(\theta):\,\theta\in D_{R_{\mathcal{K}}} \}$. {\color{red}Setting} $\epsilon = \epsilon_{R_{\mathcal{K}}+1}>0$ {\color{red}and assuming} $\text{dis}(\mathbb{P})<\epsilon$, there exists a full $\mathbb{P}$-probability event $\Omega_{R_{\mathcal{K}}+1}$ {\color{red}on which} $\Lambda_{q,\omega}(\theta) = \Lambda_a(\theta)$ simultaneously for all $\theta\in D_{R_{\mathcal{K}}+1}$. {\color{red}Hence on $\Omega_{R_{\mathcal{K}}+1}$, we have} $\pi(\mathcal{O}_{R_{\mathcal{K}}}) =  \{ \nabla\Lambda_{q,\omega}(\theta):\,\theta\in D_{R_{\mathcal{K}}} \}$. By Lemma \ref{lem: boundary LDP general existence}, {\color{red}for all} $\omega\in \Omega_{R_{\mathcal{K}}+1}$ the average $(S_n/n)_{n\geq0}$ under $P_{0,\omega}$ satisfies an LDP on $\pi(\mathcal{O}_{R_{\mathcal{K}}})$ with rate function  $\Lambda_{q,\omega}^*(x) = \Lambda^*(x) \coloneqq \langle\theta_x,x\rangle - \nabla\Lambda_a(\theta_x)$, where $\theta_x\in D_{R_{\mathcal{K}}}$ is {\color{red}uniquely} determined {\color{red}by} $x = \nabla\Lambda_a(\theta_x)$. {\color{red}Explicitly, this local} LDP {\color{red}restricted to} $\pi(\mathcal{O}_{R_{\mathcal{K}}})$ {\color{red}means that} 
       \[
                \varlimsup_{n\to\infty} \tfrac{1}{n}\log P_{0,\omega}\big(\tfrac{1}{n}S_n\in F,\,\mathbb{B}_n\big) \leq -\inf_{x\in F} \Lambda^*(x),\qquad \varliminf_{n\to\infty} \tfrac{1}{n}\log P_{0,\omega}\big(\tfrac{1}{n}S_n\in G,\,\mathbb{B}_n\big) \geq -\inf_{x\in G} \Lambda^*(x),
        \]
        for any closed $F\subseteq\pi(\mathcal{O}_{R_{\mathcal{K}}})$ and open $G\subseteq\pi(\mathcal{O}_{R_{\mathcal{K}}})$. On the other hand, the affinity of $\pi:\partial\mathbb{D}(s)\to\mathbb{R}^{d-1}$ implies the events $\{n^{-1}S_n\in\pi(H)\}\cap\mathbb{B}_n$ and $\{n^{-1}X_n\in H\}$ are equivalent, for any Borel set $H\subseteq\partial\mathbb{D}(s)$. Hence the above relations induce an LDP for the moving average, interpreted as
        \[
                \varlimsup_{n\to\infty} \tfrac{1}{n}\log P_{0,\omega}\big(\tfrac{1}{n}X_n\in F \big) \leq -\inf_{x\in F} \Lambda^*(x),\qquad \varliminf_{n\to\infty} \tfrac{1}{n}\log P_{0,\omega}\big(\tfrac{1}{n}X_n\in G\big) \geq -\inf_{x\in G} \Lambda^*(x),
        \]
        for any closed $F\subseteq\mathcal{O}_{R_{\mathcal{K}}}$ and open $G\subseteq\mathcal{O}_{R_{\mathcal{K}}}$. {\color{red}To complete the proof, it suffices to show that $\Lambda^*\circ\pi=I_q=I_a$} on $\mathcal{O}_{R_{\mathcal{K}}}$, because $\mathcal{K}\subseteq\mathcal{O}_{R_{\mathcal{K}}}$. {\color{red}Suppose for contradiction that} $\Lambda^*(\pi(x))< I_q(x)$ for some $x\in\mathcal{O}_{R_{\mathcal{K}}}$. {\color{red}By} the lower-semicontinuity of $I_q$, $\exists~B$ {\color{red}an open ball} of $x$ such that $\inf_{y\in\overline{B}}I_q(y) > \Lambda^*(\pi(x))$. {\color{red}Defining} $G_x = \pi(B)\cap\{y\in\mathbb{R}^d:\,y_d=0\}$, Lemma \ref{lem: boundary LDP general existence} {\color{red}implies} 
        \begin{equation*}\begin{aligned}
            -\Lambda^*(\pi(x)) &= -\Lambda^*_{q,\omega}(\pi(x)) \leq -\inf_{y\in G_x\cap\mathcal{F}^\omega} \Lambda^*_{q,\omega}(y) \leq \varliminf_{n\to\infty} \frac{1}{n}\log P_{0,\omega}\big(\tfrac{1}{n}X_n\in \pi^{-1}(G_x) \big)\\
            &\leq \varlimsup_{n\to\infty} \frac{1}{n}\log P_{0,\omega}\big(\tfrac{1}{n}X_n\in\overline{B}\big)\leq -\inf_{y\in\overline{B}}I_q(y) < -\Lambda^*(\pi(x)),
        \end{aligned}\end{equation*}
        {\color{red}a contradiction. Hence} $I_q(x) \leq \Lambda^*(\pi(x))$, $\forall~x\in\mathcal{O}_{R_{\mathcal{K}}}$.  {\color{red}For the reverse inequality, fix $x \in \mathcal{O}_{R_{\mathcal{K}}}$ and let} $0<\delta<\frac{1}{2}d_1(x,\partial\mathbb{D}_{d-2})$. {\color{red}Define} $B_\delta(x)\coloneqq\{y\in\mathbb{R}^d:\,\abs{y-x}_2<\delta\}$ and $B_{\delta\times\delta^\prime}(x)\coloneqq(B_\delta(x)\cap\partial\mathbb{D}_s)\times s(-\delta,\delta)\subseteq \partial\mathbb{D}_s\times\mathbb{R}$ {\color{red}for} $\delta^\prime>0$. {\color{red}By the continuity of} $I_q$ on $\mathbb{D}$, {\color{red}for any} $\epsilon>0$,   $\exists~\delta=\delta(\epsilon)>0$ small enough so that $ -I_q(x)\leq -\inf_{y\in B_{\delta\times\delta}(x)} I_q(y) + \epsilon/2 \leq \varliminf_{n\to\infty}\frac{1}{n}\log P_{0,\epsilon}(\tfrac{1}{n}X_n\in B_{\delta\times\delta}(x) ) +\epsilon/2$. Hence, $-I_q(x)\leq \varlimsup_{n\to\infty}\frac{1}{n}\log P_{0,\epsilon}(\tfrac{1}{n}X_n\in \overline{B_{\delta\times\delta}(x)}\, ) +\epsilon/2$. {\color{red}Monotonicity yields} that as $t\searrow0$, 
        \[
            \varlimsup_{n\to\infty}\tfrac{1}{n}\log P_{0,\epsilon}\big(\tfrac{1}{n}X_n\in \overline{B_{\delta\times t}(x)}\, \big)\searrow \varlimsup_{n\to\infty}\tfrac{1}{n}\log P_{0,\epsilon}\big(\tfrac{1}{n}X_n\in \overline{B_{\delta}(x)}\cap\partial\mathbb{D}_s \big),\quad\forall~0 < \delta < \tfrac{1}{2}d_1(x,\partial\mathbb{D}_{d-2}).
        \]
        Therefore, {\color{red}choosing} sufficiently small $\delta^\prime(\epsilon)<\delta$, {\color{red}we have} $-I_q(x)$ less than or equal to
        \[     
            \varlimsup_{n\to\infty}\tfrac{1}{n}\log P_{0,\epsilon}\big(\tfrac{1}{n}X_n\in \overline{B_{\delta\times\delta^\prime}(x)}\, \big) +\tfrac{\epsilon}{2}\leq \varlimsup_{n\to\infty}\tfrac{1}{n}\log P_{0,\epsilon}\big(\tfrac{1}{n}X_n\in \overline{B_{\delta}(x)}\cap\partial\mathbb{D}_s \big) + \epsilon \leq -\inf_{y\in\overline{B_{\delta}(x)}\cap\partial\mathbb{D}_s} \Lambda^*_{q,\omega}(\pi(y)) + \epsilon.
        \]
        {\color{red}Taking} $\delta\to0$ {\color{red}followed by} $\epsilon\to0$, we {\color{red}conclude that} $I_q(x) = \Lambda^*_{q,\omega}(\pi(x))$ for any $\omega\in\Omega_{R_{\mathcal{K}}+1}$ and $x\in\mathcal{O}_{R_{\mathcal{K}}}$ where $\mathcal{K}\subseteq\mathcal{O}_{R_{\mathcal{K}}}$. {\color{red}Identical reasoning applies to} $I_a$, {\color{red}establishing $I_a = \Lambda^* \circ \pi = I_q$} on $\mathcal{O}_{R_{\mathcal{K}}}$ {\color{red}and concluding the proof.}
    \end{proof}

\subsection{Effect of imbalance on the rate functions at the boundary.}
    To prove Theorem \ref{thm: imbalance of LDP at boundary}, the argument {\color{red}in the proof of} Theorem \ref{thm: equality of LDP at boundary} reveals that it suffices to {\color{red}establish the existence of constants}  $\epsilon^*=\epsilon^*(d,\kappa)$ and $\eta=\eta(d,\kappa)$ such that {\color{red}$\sup_{n\in\mathbb{N},\,\abs{\theta}\leq\eta}\normx{\mathcal{H}_{n,\theta}}_{L^2(\mathbb{P}\otimes Q)}<\infty$} whenever $\text{imb}_s(\mathbb{P})<\epsilon^*$. {\color{red}The remainder} of this section is {\color{red}therefore} dedicated to {\color{red}proving this uniform square-integrability bound. Throughout the rest of this section,} we {\color{red}assume without loss of generality that} $\text{imb}_s(\mathbb{P})<1/2$.

    \begin{proof}[\textbf{Proof of Theorem \ref{thm: imbalance of LDP at boundary}}.]
       Let {\color{red}$U\coloneqq\{\theta\in\mathbb{R}^{d-1}:\,\abs{\theta}\leq1\}$ and set}
       \[
            \tilde{C}(U)\coloneqq \sup_{\ell\in\mathbb{N},\theta\in U} \sum_{j=1}^\infty \ell^{-2}\hat{E}^{(\theta)}_{x,y} [ N^{\ell}_j(\Vec{X}^{(\theta)},\Vec{Y}^{(\theta)})I_{\{d_1(\Vec{X}^{(\theta)}_j,\Vec{Y}^{(\theta)}_j)\leq \ell\}} ].
       \]
       Then we have
        \[
            \normx{\mathcal{H}_{n,\theta}}^2_{L^2(\mathbb{P}\otimes Q)} = \hat{E}^{(\theta)}_0\big[e^{\sum_{j=1}^n  V(\Vec{X}^{(\theta)}_j,\Vec{Y}^{(\theta)}_j)}\big] = \hat{E}^{(\theta)}_0 \big[e^{\sum_{j=1}^n V(\Vec{X}^{(\theta)}_j,\Vec{Y}^{(\theta)}_j) I\{d_1(\Vec{X}^{(\theta)}_j,\Vec{Y}^{(\theta)}_j)<L\} }\big].
        \]
        Define $\mathscr{V}^{(0)}_{\kappa,s,\theta}\coloneqq(d-1)\abs{\theta}+\frac{1}{2}\text{log}(\tfrac{1+\text{imb}_s(\mathbb{P})}{1-\text{imb}_s(\mathbb{P})})$. {\color{red}For each} $k\geq1$, {\color{red}recursively define} $\mathscr{V}^{(k)}_{\kappa,s,\theta}$ by $e^{\frac{1}{2}\mathscr{V}^{(k+1)}_{\kappa,s,\theta}}=e^{ \frac{1}{2} \mathscr{V}^{(k)}_{\kappa,s,\theta} }+(e^{ \frac{1}{2} \mathscr{V}^{(k)}_{\kappa,s,\theta} }-1)K_{\kappa,\theta}$ {\color{red}where} $K_{\kappa,\theta}\coloneqq e^{2(d-1)\abs{\theta}}(1-2\kappa)$. Following {\color{red}an argument} similar {\color{red}to} \cite[Pages 21--22]{Bazaes/Mukherjee/Ramirez/Sagliett}, we {\color{red}obtain} the estimate 
        \[
            \normx{\mathcal{H}_{n,\theta}}^2_{L^2(\mathbb{P}\otimes Q)} \leq \hat{E}^{(\theta)}_0[ \exp \sum_{j<n} \sum_{i=\tau^{(L)}_{j-1}}^{\tau^{(L)}_j-1} \mathscr{V}^{(n-i-1)}_{\kappa,s,\theta} I\{|X^{(\theta)}_i-Y^{(\theta)}_i|_1<L\} ].
        \]
        {\color{red}Observe} that {\color{red}there exists} $0<\eta^\prime(d,\kappa)<1$ such that {\color{red}$K_{\kappa,\theta}<1$ whenever $\vert{}\theta\vert{}_2 < \eta'(d,\kappa)$.} For any such $\theta$ and any $k\geq1$, 
        \[
            e^{\frac{1}{2}\mathscr{V}^{(k)}_{\kappa,s,\theta}} = e^{\frac{1}{2}\mathscr{V}^{(0)}_{\kappa,s,\theta}}+\big( e^{\frac{1}{2}\mathscr{V}^{(0)}_{\kappa,s,\theta}} - 1 \big)(K_{\kappa,\theta}+ K_{\kappa,\theta}^2+\cdots+ K_{\kappa,\theta}^{k-1} ) \leq e^{\frac{1}{2}\mathscr{V}^{(\infty)}_{\kappa,s,\theta}}\coloneqq e^{\frac{1}{2}\mathscr{V}^{(0)}_{\kappa,s,\theta}}+\big( e^{\frac{1}{2}\mathscr{V}^{(0)}_{\kappa,s,\theta}} - 1 \big)\tfrac{K_{\kappa,\theta}}{1-K_{\kappa,\theta}}.
        \]
        Therefore, $\normx{\mathcal{H}_{n,\theta} }^2_{L^2(\mathbb{P}\otimes Q)}\leq  \hat{E}^{(\theta)}_0 [ e^{\sum_{j=1}^\infty 2N^{L}_j(\Vec{X}^{(\theta)},\Vec{Y}^{(\theta)})\mathscr{V}^{(\infty)}_{\kappa,s,\theta} \mathbbm{1}\{ d_1(\Vec{X}^{(\theta)}_j,\Vec{Y}^{(\theta)}_j)<L\} }]$.
        {\color{red}By shrinking} $\eta^\prime(d,\kappa)$ {\color{red}if necessary,} we {\color{red}can ensure that} $K_{\kappa,\theta}\leq1-\kappa$ and $2(d-1)\abs{\theta}\leq1$ {\color{red}whenever} $\abs{\theta}\leq\eta^\prime(d,\kappa)$. {\color{red}Under these constraints,}
        \[
            \mathscr{V}^{(\infty)}_{\kappa,s,\theta}\leq C^\prime_1\big(\abs{\theta}+\text{imb}_s(\mathbb{P})\big)\frac{1+\text{imb}_s(\mathbb{P})}{1-\text{imb}_s(\mathbb{P})} \leq C_1 \big(\abs{\theta}+\text{imb}_s(\mathbb{P})\big)\big(1+\text{imb}_s(\mathbb{P})\big)
        \]
        for some constant $C_1=C_1(d,\kappa)>0$, where the second inequality {\color{red}uses the assumption} $\text{imb}_s(\mathbb{P})\leq1/2$. On the other hand, {\color{red}applying} (\ref{eqn: left to prove in step III}) we {\color{red}have} $\sup_{\abs{\theta}\leq\eta^\prime}\sum_{j=1}^\infty \hat{E}^{(\theta)}_{x,y}[ N^{L}_j(\Vec{X}^{(\theta)},\Vec{Y}^{(\theta)})I_{\{d_1(\Vec{X}^{(\theta)}_j,\Vec{Y}^{(\theta)}_j)\leq L\}} ] \leq \tilde{C}L^2\coloneqq C_2(d,\kappa)$. 
        {\color{red}Arguments analogous to those in} the proof of Lemma \ref{lem: square integrable of the martingale H} {\color{red}then yield constants} $\eta(d,\kappa)<\eta^\prime$ and $\epsilon^*(d,\kappa)$ {\color{red}that guarantee the  $L^2(\mathbb{P}\otimes Q)$} integrability of $(\mathcal{H}_{n,\theta})_{n\geq1}$, provided {\color{red}$\vert{}\theta\vert{}_2 \le \eta(d,\kappa)$ and} $\text{imb}_s(\mathbb{P})\leq \epsilon^*$. {\color{red}This completes the proof.}
    \end{proof}

\section{Threshold of disorder in the interior of the velocity surface}
    When $\mathbb{P}$ {\color{red}satisfies} $\textbf{(SMF)}_{C,g}$, {\color{red}we establish} the existence of a nontrivial {\color{red}disorder} threshold for any compact subset {\color{violet}of the interior} $\text{int}(\mathbb{D})\backslash\{0\}$ of the velocity surface. Below this threshold, the {\color{red}two} rate functions $I_a,I_q$ are forced to match. Although {\color{violet}ballisticity} and the CLT are not {\color{red}required}, we emphasize that our result holds only when $\text{dis}(\mathbb{P})$ is {\color{red}sufficiently} small. The {\color{red}central} technique is an approximate renewal sequence {\color{red}which} separates the entangled information from the  environment.

\subsection{Construction of an auxiliary Markovian $Q^z$-random walk.}\label{sec: 5.1}
    Take an arbitrary $z\in\text{int}(\mathbb{D})\backslash\{0\}$ {\color{red}and} choose $C^{\mathbb{P}}_{z}>0$ {\color{red}ensuring} $ f(C^{\mathbb{P}}_{z})\coloneqq \frac{1}{2}\sum_{e\in\mathbb{V}} \sqrt{\abs{\langle  z,e\rangle}^2+ 4 C^{\mathbb{P}}_{z} \mathbb{E}[\omega(0,e)]\mathbb{E}[\omega(0,-e)]}=1$. {\color{red}The existence of} such $C^{\mathbb{P}}_{z}<\infty$ {\color{red}follow directly from the properties} $f(0)=\abs{z}<1$ and $f(\infty)=\infty$. Define the probability vector $(u_\lambda(e))_{e\in\mathbb{V}}$ by 
    \[
        2u_z(e)\coloneqq \langle z,e\rangle + \sqrt{ \abs{\langle  z,e\rangle}^2+ 4 C^{\mathbb{P}}_{z} \mathbb{E}[\omega(0,e)]\mathbb{E}[\omega(0,-e)]},\qquad\forall~e\in\mathbb{V}.
    \]
    {\color{red}We} define the $Q^z$-walk $(Z_n)_{n\geq0}$ on $\mathbb{Z}^d$ starting {\color{red}at} $x\in\mathbb{Z}^d$ with law $Q^z_x$ prescribed by $ Q^z_x(Z_{n+1}=Z_n+e) = u_z(e)$ for any $e\in\mathbb{V}$ and $n\geq0$. {\color{red}Under} this {\color{red}specification}, there exists $c_z=c_z(\kappa)>0$ such that $u_z(e)\geq c_z$ for all $e\in\mathbb{V}$, {\color{red}and} $E^{Q^z}_x[Z_{n+1}-Z_n]= z$ for all $n\in\mathbb{N}$. Proceeding {\color{red}along arguments} analogous {\color{red}to} \cite[Lemma 3.2]{Bazaes/Mukherjee/Ramirez/Sagliett0}, for each $n\geq1$ and $\theta\in\mathbb{R}^d$ we have
    \begin{equation}\label{eqn: Q walk form1, annealed}    
        E^{Q^{z}}_0\bigg[ e^{\langle\theta,Z_n\rangle}\mathbb{E} \prod_{j=1}^n\xi(Z_{j-1},\Delta_j(Z)) \bigg] = \mathcal{D}_{z,\mathbb{P}}^n E_{0}\bigg[ e^{\langle\theta+\theta^{\mathbb{P}}_{z},X_n\rangle} \bigg],\quad E^{Q^{z}}_0\bigg[ e^{\langle\theta,Z_n\rangle}\prod_{j=1}^n\xi(Z_{j-1},\Delta_j(Z)) \bigg] = \mathcal{D}_{z,\mathbb{P}}^n E_{0,\omega}\bigg[ e^{\langle\theta+\theta^{\mathbb{P}}_{z},X_n\rangle} \bigg]
    \end{equation}
    where $\mathcal{D}_{z,\mathbb{P}}\coloneqq (C^{\mathbb{P}}_{z})^{1/2}$ and the vector $\theta^{\mathbb{P}}_{z}\in\mathbb{R}^d$ is defined {\color{red}by} $ \langle \theta^{\mathbb{P}}_{z},e_j\rangle\coloneqq \text{log}(\mathcal{D}_{z,\mathbb{P}}^{-1}u_z(e_j)/\mathbb{E}[\omega(0,e_j)])$, {\color{red}for all} $j$. {\color{red}Here} $\Delta_j(Z)$ stands for $Z_j-Z_{j-1}$ with $j=1,\ldots,n$. Following \cite[Theorem 4.3.1]{Dembo/Zeitouni} for the annealed scenario and \cite[Theorem 2.6]{Rassoul-Agha/Seppalainen} for the quenched scenario, {\color{red}identity} (\ref{eqn: Q walk form1, annealed}) yields the limiting {\color{red}free energies} 
    \[
        \overline{\Lambda}^a_z(\theta)\coloneqq\lim_{N\to\infty}\frac{1}{N}\log E^{Q^{z}}_0\bigg[ e^{\langle\theta,Z_N\rangle}\mathbb{E} \prod_{j=1}^N\xi(Z_{j-1},\Delta_j(Z)) \bigg],\;\; \overline{\Lambda}^q_{z}(\theta)\coloneqq\lim_{N\to\infty}\frac{1}{N}\log E^{Q^{z}}_0\bigg[ e^{\langle\theta,Z_N\rangle} \prod_{j=1}^N\xi(Z_{j-1},\Delta_j(Z)) \bigg],\;\;\forall~\theta\in\mathbb{R}^d,
    \]
    {\color{red}which} satisfy $\overline{\Lambda}^a_{z}(\theta) = \log(C^{\mathbb{P}}_{z,})^{1/2} + \Lambda_a(\theta+\theta^{\mathbb{P}}_{z})$, $ \overline{\Lambda}^q_{z}(\theta) = \log(C^{\mathbb{P}}_{z})^{1/2} + \Lambda_q(\theta+\theta^{\mathbb{P}}_{z})$. We also {\color{red}write} $\overline{\Lambda}^*_z(\theta)\coloneqq\log \sum_{z\in\mathbb{V}}e^{\langle\theta,e\rangle}u_z(e)$ for the limiting free energy in the zero-disorder {\color{red}case for $\theta\in\mathbb{R}^d$}. Fix $\ell\in\mathbb{V}$ {\color{red}} with $\langle z,\ell\rangle>0$, {\color{red}and set} $\mathcal{W}\coloneqq\mathbb{V}\cup\{0\}$. Consider the probability measure $\overline{Q}^z_0\coloneqq U\otimes\overline{Q}^z_{\epsilon,0}$ on $  (\mathcal{W})^{\mathbb{N}}\times(\mathbb{Z}^d)^{\mathbb{N}}$, where $U$ is product measure on $(\mathcal{W})^{\mathbb{N}}$ s.t. $U(\epsilon_1=e)=\CYRk$, $\forall~e\in\mathbb{V}$, $U(\epsilon_1=0)=1-2d\CYRk$, $\forall~\epsilon=(\epsilon_1,\epsilon_2,\ldots)\in(\mathcal{W})^{\mathbb{N}}$ for some small $\CYRk=\CYRk(\kappa)>0$. {\color{red}The} Markov chain $\overline{Q}^z_{\epsilon,0}$ on $\mathbb{Z}^d$ is defined via 
    \[
        \overline{Q}^z_{\epsilon,0}(Z_{n+1}=x+e|Z_n=x) = \mathbbm{1}_{\{ \epsilon_{n+1}=e \} } +\frac{\mathbbm{1}_{\{ \epsilon_{n+1}=0 \} }}{1-2d\CYRk}\big(u_z(e)-\CYRk\big),\qquad\forall~e\in\mathbb{V},\;\; x\in\mathbb{Z}^d.
    \]
    {\color{red}In particular,} we can {\color{red}choose} $2\CYRk(\kappa)\coloneqq d^{-1}\wedge\inf\{u_z(e):\,e\in\mathbb{V}\}$. {\color{red}Crucially,} the {\color{red}marginal} law {\color{red}of $(Z_n)_{n \ge 0}$ under} $\overline{Q}^z_0 = U\otimes\overline{Q}^z_{\epsilon,0}$ {\color{red}coincides with} $Q^z_0$. {\color{red}Let} $\Bar{\epsilon}\coloneqq\ell$ and the cone $\CYRz(x,\zeta,\ell)\coloneqq\{\nu\in\mathbb{Z}^d:\,\langle\nu-x,\ell\rangle\geq\zeta\abs{\ell}\abs{\nu-x}\}$, $\forall~x\in\mathbb{Z}^d$ and $0<\zeta<1$. {\color{red}We denote by} $\Bar{\epsilon}^{(L)}\coloneqq(\Bar{\epsilon},\ldots,\Bar{\epsilon})$ the $L$-repetition of $\Bar{\epsilon}$, $\forall~L\geq L_0>r$, {\color{red}where} $\mathbb{P}$ {\color{red}is} assumed $r$-Markovian. For $u\in\mathbb{R}$, {\color{red}define}
    \[
        H^{(L)}_u\coloneqq\inf\{ n\geq L:\,\langle Z_{n-L},\ell\rangle > u,\,(\epsilon_{n-L},\ldots,\epsilon_{n-1})=\Bar{\epsilon}^{(L)} \},\quad S_0=0,\quad \beta^{(\zeta)}_0\coloneqq \inf\{n\geq1:\,Z_n\not\in\CYRz(Z_0,\zeta,\ell)\}
    \]
    and $R_0\coloneqq\langle Z_0,\ell\rangle$. {\color{red}Inductively} define the sequences of stopping times $(S_k)_{k\geq0}$, $(\beta^{(\zeta)}_k)_{k\geq0}$ and {\color{red}the sequence} $(R_k)_{k\geq0}$ by
    \[
        S_{k+1}\coloneqq H^{(L)}_{R_k},\;\;\;\beta^{(\zeta)}_{k+1}\coloneqq\inf\{n>S_{k+1}:\,Z_n\not\in\CYRz(Z_{S_{k+1}},\zeta,\ell)\},\;\;\; R_{k+1}\coloneqq\sup\{\langle Z_n,\ell\rangle:\,0\leq n\leq \beta^{(\zeta)}_{k+1}-L\}\;\;\;\text{if}\;\; \beta^{(\zeta,)}_{k+1}<\infty,
    \]
    and $R_{k+1}\coloneqq\langle Z_{S_{k+1}},\ell\rangle$ if $\beta^{(\zeta)}_{k+1}=\infty$. {\color{red}Finally,} the renewal sequence $(\tau^{(L)}_k)_{k\geq1}$ {\color{red}is} defined {\color{red}by} $\tau^{(L)}_k\coloneqq S(W^{(L)}_k)$, {\color{red}where} $W^{(L)}_0\coloneqq0$ and $W^{(L)}_{k+1}\coloneqq\inf\{n>W^{(L)}_{k}:\,\beta^{(\zeta)}_n=\infty\}$ for {\color{red}all} $k\geq1$ Here  $E^{\overline{Q}^z}_0\coloneqq E_{\overline{Q}^z_0}$ {\color{red}denotes} the $\overline{Q}^z_0$-expectation.

    \begin{proposition}\label{prop: mixing exponential of tau < infinity}
        \normalfont
        When $\mathbb{P}$ {\color{red}satisfies} $\textbf{(SMF)}_{C,g}$, {\color{red}there exist two constants} $\gamma_0=\gamma_0(\kappa,z)>0$ and $\eta=\eta(\kappa,z)>1$ {\color{red}such that} $E^{\overline{Q}^z}_0[e^{\gamma\CYRk^L\tau^{(L)}_1}|\beta^{(\zeta)}_0=\infty]<\eta$ and $E^{\overline{Q}^z}_0[ \CYRk^L\tau^{(L)}_1|\beta^{(\zeta)}_0=\infty]>\eta^{-1}$, {\color{red}for all}  $\gamma<\gamma_0$ and $L\geq L_0$.
    \end{proposition}
    \begin{proof}
        {\color{red}By} \cite[Lemma 3.4]{Bazaes/Mukherjee/Ramirez/Sagliett0}, {\color{red}there exists} $\eta(\kappa)>1$ {\color{red}such that} $Q^z_0(\beta^{(\zeta)}_0=\infty)\geq\eta^{-1}$. {\color{red}Partitioning} according to {\color{red}the values} of $W^{(L)}_1$, we observe that
        \begin{equation*}\begin{aligned}
            E^{\overline{Q}^z}_0[e^{\gamma\CYRk^L \langle Z(\tau^{(L)}_1),\ell\rangle}] \leq \sum_{k=1}^\infty E^{\overline{Q}^z}_0  [e^{\gamma\CYRk^L \langle Z_{S_k},\ell\rangle},\,W^{(L)}_1=k ] \leq \sum_{k=1}^\infty  E^{\overline{Q}^z}_0 [e^{\gamma\CYRk^L \langle Z_{S_k},\ell\rangle},\,S_k<\infty,\,R_k=\infty ].
        \end{aligned}\end{equation*}
        Hence, $E^{\overline{Q}^z}_0[e^{\gamma \CYRk^L \langle Z(\tau^{(L)}_1),\ell\rangle}] \leq \sum_{k=1}^\infty E^{\overline{Q}^z}_0 [e^{\gamma \CYRk^L \langle Z_{S_k},\ell\rangle},\,S_k<\infty ] \overline{Q}^z_0(\beta^{(\zeta)}_0=\infty)$. {\color{red}For any} $a\in\mathbb{R}$, define $\overline{T}^{\ell}_a\coloneqq\inf\{n\geq0:\,\langle Z_n,\ell\rangle>a\}$, and {\color{red}for each} $k\in\mathbb{N}$, {\color{red}set} $M^z_k\coloneqq\sup\{\langle Z_n,\ell\rangle:\,n\leq R_k \}$. {\color{red}Define }the time sequence $(t^{(n)}_k)_{n\geq0}$ {\color{red}inductively} by
        \[
            t^{(0)}_k\coloneqq\overline{T}^\ell_{M^z_k}\qquad\text{and}\qquad t^{(n+1)}_k\coloneqq \overline{T}^\ell_{\langle Z(t^{(n)}_k),\ell\rangle},\qquad\forall~n\in\mathbb{N}.
        \]
        We also {\color{red}introduce} the parametrized events $\mathcal{A}^{(L)}_{n,k}$ and $\mathcal{B}^{(L)}_{n,k}$ in $(\mathcal{W})^{\mathbb{N}}$ {\color{red}given by}
        \[
            \mathcal{A}^{(L)}_{n,k}\coloneqq\big\{\epsilon:\,(\epsilon_{t^{(n)}_k},\epsilon_{t^{(n)}_k+1},\ldots,\epsilon_{t^{(n)}_k+L-1})=\Bar{\epsilon}^{(L)}\big\},\qquad \mathcal{B}^{(L)}_{n,k}\coloneqq\big\{\epsilon:\,(\epsilon_{t^{(j)}_k},\epsilon_{t^{(j)}_k+1},\ldots,\epsilon_{t^{(j)}_k+L-1})\neq\Bar{\epsilon}^{(L)},\,\forall~j\leq n-1\big\}.
        \]
        Following arguments {\color{red}similar to those} in \cite[Section 4]{Guerra Aguilar} and \cite[Section 6]{Guerra Aguilar/Ramirez}, we {\color{red}have} $\{S_{k+1}<\infty\}\subseteq \cup_{n\geq0}\big\{t^{(n)}_k<\infty,\,\mathcal{A}^{(L)}_{n,k},\,\mathcal{B}^{(L)}_{n,k}\big\}$, {\color{red}for all} $k\in\mathbb{N}$. {\color{red}Furthermore}, conditioned on $\mathcal{A}^{(L)}_{n,k}\cap\mathcal{B}^{(L)}_{n,k}$, {\color{red}it holds $\overline{Q}^z_0$-a.s. that} $S_{k+1}=t^{(n)}_k+L$. Hence $\forall~k\geq0$, we have
        \[
            E^{\overline{Q}^z}_0[e^{\gamma \CYRk^L \langle Z_{S_{k+1}},\ell\rangle},\,S_{k+1}<\infty] \leq \sum_{n\in\mathbb{N}} E^{\overline{Q}^z}_0 [e^{\gamma \CYRk^L \langle Z_{S_{k+1}},\ell\rangle},\,t^{(n)}_k<\infty,\,\mathcal{A}^{(L)}_{n,k},\,\mathcal{B}^{(L)}_{n,k} ].
        \]
        {\color{red}On} the event $\{ t^{(n)}_k<\infty,\,\mathcal{A}^{(L)}_{n,k},\,\mathcal{B}^{(L)}_{n,k}\}$ we have {\color{red}the estimate} $\langle Z_{S_{k+1}},\ell \rangle \leq M^z_k+\abs{\ell}n +L$. {\color{red}It follows that}
        \begin{equation*}\begin{aligned}
            &E^{\overline{Q}^z}_0  [e^{\gamma \CYRk^L \langle Z_{S_{k+1}},\ell\rangle},\,S_{k+1}<\infty ] \leq \sum_{n\in\mathbb{N}} E^{\overline{Q}^z}_0 [e^{\gamma \CYRk^L(M^z_k+\abs{\ell}n +L)},\,t^{(n)}_k<\infty,\,\mathcal{A}^{(L)}_{n,k},\,\mathcal{B}^{(L)}_{n,k} ]\\
            &\qquad\leq \big(\sum_{n=0}^{L^2-1}+\sum_{n= L^2}^\infty\big) \CYRk^L E^{\overline{Q}^z}_0 [e^{\gamma \CYRk^L(M^z_k+\abs{\ell}n +L)},\,t^{(n)}_k<\infty,\,\mathcal{B}^{(L)}_{n,k} ].
        \end{aligned}\end{equation*}        
        {\color{red}We now apply} a precise and refined argument {\color{red}based on} \cite[Equation (4.11)]{Guerra Aguilar}. {\color{red}By} \cite[Lemma 6.2]{Guerra Aguilar/Ramirez}, for any $n\geq L^2$, we have $\overline{Q}^z_0(\mathcal{B}^{(L)}_{n,k})\leq (1-\tilde{c}L^2\CYRk^L)^{[ n/L^2]}$, {\color{red}where} $\tilde{c}>(1-2\CYRk)/(1-\CYRk)$. {\color{red}Choosing} $\CYRk(\kappa,z)$ sufficiently small so that $\eta^{-1}>\CYRk/(1-\CYRk)$, there exists {\color{red}a constant} $\Bar{\eta}\coloneqq\Bar{\eta}(y,\kappa)>1$ {\color{red}such that} $\frac{1-1/\eta}{1-1/\Bar{\eta}}<\Tilde{c}$. Henceforth, we {\color{red}obtain} the following
        \begin{equation*}\begin{aligned}
            & E^{\overline{Q}^z}_0  [e^{\gamma \CYRk^L\langle Z_{S_{k+1}},\ell \rangle},\,S_{k+1}<\infty ] \leq L^2\CYRk^L e^{\gamma \CYRk^L(L^2\abs{\ell}+L)} E^{\overline{Q}^z}_0  [ e^{\gamma \CYRk^L M^z_k},\,t^{(0)}_k<\infty  ] + e^{\gamma  L\CYRk^L} E^{\overline{Q}^z}_0 [ e^{\gamma \CYRk^L M^z_k},\,t^{(0)}_k<\infty  ] \\ 
            &\quad\vdot\sum_{n\geq L^2} \big(e^{\gamma  L^2\CYRk^L\abs{\ell} }(1-\tilde{c}L^2\CYRk^L)\big)^{[ n/L^2]}  \leq L^2\CYRk^L e^{\gamma \CYRk^L(L^2\abs{\ell} +L)} E^{\overline{Q}^z}_0 [ e^{\gamma \CYRk^L M^z_k},\,t^{(0)}_k<\infty  ]  + \CYRe^{\gamma,z}_{\CYRk,L}  E^{\overline{Q}^z}_0 [ e^{\gamma \CYRk^L M^z_k},\,t^{(0)}_k<\infty  ].
        \end{aligned}\end{equation*}
        where $\CYRe^{\gamma,z}_{\CYRk,L}\coloneqq \frac{ e^{\gamma  L\CYRk^L}  L^2\CYRk^L}{e^{-\gamma  L^2\CYRk^L\abs{\ell} } - (1-\Tilde{c}L^2\CYRk^L) }$. {\color{red}Choosing} $\gamma\leq\gamma_0^\prime$ for {\color{red}a sufficiently small constant} $ \gamma_0^\prime(\kappa,z)>0$, {\color{red}we have} 
        \begin{equation*}\begin{aligned}
            &E^{\overline{Q}^z}_0 [e^{\gamma \CYRk^L\langle Z_{S_{k+1}},\ell \rangle},\,S_{k+1}<\infty] \leq \big( L^2\CYRk^L e^{\gamma  \CYRk^L (L^2\abs{\ell} +L)} +\tfrac{1-1/2\Bar{\eta}}{1-1/\eta} e^{ \gamma  L\CYRk^L } \big) E^{\overline{Q}^z}_0 [ e^{\gamma \CYRk^L M^z_k},\,t^{(0)}_k<\infty  ]\\
            &\qquad\quad \leq \tfrac{1-1/3\Bar{\eta}}{1-1/\eta} E^{\overline{Q}^z}_0  [e^{\gamma \CYRk^L\langle Z_{S_{k}},\ell\rangle},\,S_{k}<\infty ] E^{\overline{Q}^z}_0 [e^{\gamma  (R^{(\zeta )}+1)},\,\beta^{(\zeta )}_0<\infty ]
        \end{aligned}\end{equation*}
        where $R^{(\zeta )}\coloneqq\sup\{\langle Z_n,\ell\rangle:\,0\leq n\leq \beta^{(\zeta )}_0\}$. Notice that for any $N\geq1$,
        \begin{equation}\label{eqn: mixing tau dog1}        
            E^{\overline{Q}^z}_0 [ e^{\gamma (R^{(\zeta)}+1)},\,\beta_0^{(\zeta )}<\infty ] \leq e^{\gamma N} \overline{Q}^z_0(\beta_0^{(\zeta )}<\infty) + \sum_{n\in\mathbb{N}} e^{\gamma N 2^{n+1}}\overline{Q}^z_0(2^n\leq R^{(\zeta)}+1<2^{n+1},\,\beta_0^{(\zeta)}<\infty).
        \end{equation}
        {\color{red}Define} $\beta^\dagger\coloneqq\inf\{n\geq0:\,Z_n\not\in\CYRz(Z_0,\zeta,\ell)\}$ {\color{red}and} $T^{\ell,z}_u\coloneqq\inf\{ n\geq0:\,\langle Z_n,\ell\rangle\geq u \}$ {\color{red}for} $u>0$. {\color{red}Let} $\vartheta$ {\color{red}denote} the canonical time-shift operator on $\mathbb{N}$. Then {\color{red}for each} $n\in\mathbb{N}$,
        \begin{equation*}\begin{aligned}
            &\overline{Q}^z_0(2^n\leq R^{(\zeta)}+1<2^{n+1},\,\beta_0^{(\zeta)}<\infty) \leq \overline{Q}^z_0(T^{\ell,z}_{2^n}\leq\beta^{(\zeta)}_0<\infty,\, \beta^\dagger\circ\vartheta_{T^{\ell,z}_{2^n}} < T^{\ell,z}_{2^{n+1}} \circ\vartheta_{T^{\ell,z}_{2^n}})\\
            &\qquad \leq \overline{Q}^z_0(  Z_{T^{\ell,z}_{2^n}} \not\in \partial^+ B^\ell_{2^n,\mathfrak{r}2^n},\,T^{\ell,z}_{2^n}\leq\beta^{(\zeta)}_0<\infty) + \overline{Q}^z_0( Z_{T^{\ell,z}_{2^n}} \in \partial^+ B^\ell_{2^n,\mathfrak{r}2^n},\,\beta^\dagger\circ\vartheta_{T^{\ell,z}_{2^n}} < T^{\ell,z}_{2^{n+1}} \circ\vartheta_{T^{\ell,z}_{2^n}}),
        \end{aligned}\end{equation*}
        where $\mathfrak{r}>0$ is specified in the Lemma \ref{lem: D0}. Here $B^{\ell,x}_{\mu,\nu}\coloneqq x+ \mathcal{R}([-\mu,\mu)\times[-\nu,\nu)^{d-1})\cap\mathbb{Z}^d$ {\color{red}with} $\mathcal{R}$ {\color{red}being} the rotation of $\mathbb{R}^d$ {\color{red}mapping $e_1\mapsto\ell\in\mathbb{V}$}, and $B^\ell_{\mu,\nu}\coloneqq B^{\ell,0}_{\mu,\nu}$. In light of Lemma \ref{lem: D0}, we have 
        \[
            \overline{Q}^z_0( Z_{T^{\ell,z}_{2^n}} \not\in \partial^+ B^\ell_{2^n,\mathfrak{r}2^n},\,T^{\ell,z}_{2^n}\leq\beta^{(\zeta)}_0<\infty) \leq e^{-\langle z,\ell\rangle\vdot 2^n },\qquad{\color{red}\forall~n\in\mathbb{N}}.
        \]
        {\color{red}Setting} $F_n\coloneqq\partial^+B^\ell_{2^n,\mathfrak{r}2^n}$, {\color{red}the} strong Markov property {\color{red}implies}
        \[
            \overline{Q}^z_0(  Z_{T^{\ell,z}_{2^n}} \in \partial^+ B^\ell_{2^n,\mathfrak{r}2^n},\,\beta^\dagger\circ\vartheta_{T^{\ell,z}_{2^n}} < T^{\ell,z}_{2^{n+1}} \circ\vartheta_{T^{\ell,z}_{2^n}}) \leq \sum_{\nu\in F_n} \overline{Q}^z_{\nu}(\beta^\dagger< T^{\ell,z}_{2^{n+1}}).
        \]
        {\color{red}For} each $\nu\in F_n$, {\color{red}the probability} $\overline{Q}^z_0( \beta^\dagger\geq T^{\ell,z}_{2^{n+1}} )$ is no less than
        \[
             \sum_{w\in \partial^+ B^{\ell,\nu}_{2^{n-1},\mathfrak{r}2^{n-1}} } \overline{Q}^z_0\big( Z_{T_{B^{\ell,\nu}_{2^{n-1},\mathfrak{r}2^{n-1}}}} \in \partial^+ B^{\ell,\nu}_{2^{n-1},\mathfrak{r}2^{n-1}},\, Z_{T_{B^{\ell,\nu}_{2^{n-1},\mathfrak{r}2^{n-1}}}}=w,\, (Z_{T_{B^{\ell,w}_{2^{n-1},\mathfrak{r}2^{n-1}}}}\in \partial^+ B^{\ell,w}_{2^{n-1},\mathfrak{r}2^{n-1}})\circ\vartheta
            _{T_{B^{\ell,w}_{2^{n-1},\mathfrak{r}2^{n-1}}}} \big),
        \]
        where $T_\Delta\coloneqq\inf\{n\geq:\,Z_n\not\in\Delta\}$ {\color{red}denotes the first exit time from} $\Delta\subseteq\mathbb{Z}^d$. Henceforth, $\overline{Q}^z_0( \beta^\dagger\geq T^{\ell,z}_{2^{n+1}} )$ is no less than
        \[
            \sum_{w\in \partial^+ B^{\ell,\nu}_{2^{n-1},\mathfrak{r}2^{n-1}} } (1-e^{-\langle z,\ell\rangle2^{n-1}})\overline{Q}^z_0(Z_{T_{B^{\ell,\nu}_{2^{n-1},\mathfrak{r}2^{n-1}}}} \in \partial^+ B^{\ell,\nu}_{2^{n-1},\mathfrak{r}2^{n-1}},\, Z_{T_{B^{\ell,\nu}_{2^{n-1},\mathfrak{r}2^{n-1}}}}=w)\geq (1-e^{-\langle z,\ell\rangle2^{n-1}})^2.
        \]
        {\color{red}Then}, $\overline{Q}^z_0(2^n\leq R^{(\zeta)}+1<2^{n+1},\,\beta_0^{(\zeta)}<\infty)\leq e^{-\langle z,\ell\rangle2^n}+\mathfrak{r}2^{n+1}e^{-\langle z,\ell\rangle2^{n-1}}$. {\color{red}Taking} $N$ sufficiently large and $\gamma\leq\gamma^\prime_0$ for some {\color{red}small} $\gamma^\prime_0(\kappa,z)>0$, {\color{red}we get} $E^{\overline{Q}^z}_0[e^{\gamma \CYRk^L\langle Z(\tau^{(L)}_1),\ell\rangle}]\leq\eta$. By the union bound, 
        \[
            \overline{Q}^z_0(\CYRk^L\tau^{(L)}_1>n)\leq \overline{Q}^z_0(\langle \CYRk^LZ(\tau^{(L)}_1),\ell\rangle>\tfrac{1}{2}\langle zn,\ell\rangle) + \overline{Q}^z_0(\CYRk^L\tau^{(L)}_1>n,\, \langle \CYRk^LZ(\tau^{(L)}_1),\ell\rangle\leq\tfrac{1}{2}\langle zn,\ell\rangle).
        \]
        The Chebyshev's inequality then gives
        \[
            \overline{Q}^z_0(\langle \CYRk^LZ(\tau^{(L)}_1),\ell\rangle>\tfrac{1}{2}\langle zn,\ell\rangle) \leq e^{-\gamma^\prime_0\langle z,\ell\rangle n/2}E^{\overline{Q}^z}_0[e^{\gamma^\prime_0 \CYRk^L Z(\tau^{(L)}_1),\ell\rangle}]\leq \eta  e^{-\gamma^\prime_0\langle z,\ell\rangle n/2},\qquad\forall~n\geq1.
        \]
        On the other hand, for each $n\geq1$, 
        \[
            \overline{Q}^z_0(\CYRk^L\tau^{(L)}_1>n,\, \langle \CYRk^LZ(\tau^{(L)}_1),\ell\rangle\leq\tfrac{1}{2}\langle zn,\ell\rangle) \leq e^{-|\langle (\CYRk^L)^{-1}zn,\ell\rangle|^2/(n/ \CYRk^L)}\leq e^{-|\langle z,\ell\rangle|^2n}.
        \]
        Therefore, $\overline{Q}^z_0(\ \CYRk^L\tau^{(L)}_1>n)\leq Ce^{-cn}$ for constants $c,C>0$ depending only on $\kappa,z$. {\color{red}This yields} $E^{\overline{Q}^z}_0[e^{\gamma_0 \CYRk^L\tau^{(L)}_1}|\beta^{(\zeta)}_0=\infty]\leq\eta$ {\color{red}for a} sufficiently small $\gamma_0=\gamma_0(\kappa,z)>0$, {\color{red}using} $\overline{Q}^z_0(\beta^{(\zeta)}_0=\infty)\geq\eta^{-1}$, {\color{red}which verifies} the upper-bound. The lower-bound follows {\color{red}directly} from $\CYRk^LZ(\tau^{(L)}_1)\geq \CYRk^L$ and {\color{red}taking} $\eta\CYRk^L\geq1$, verifying the assertion.
    \end{proof}

\subsection{Equality between $Q^z$-logarithmic moment generating functions}
    We {\color{red}introduce} the following {\color{red}sequence of} random environment{\color{red}al} $\sigma$-algebras
    \[
        \mathscr{G}_N\coloneqq\sigma\big(\omega(y,\vdot):\langle y,\ell\rangle<\langle Z(\tau^{(L)}_N),\ell\rangle-L\abs{\ell}/\abs{\ell}_1,\,(\epsilon_j)_{1\leq j\leq\tau^{(L)}_N},\,(Z_j)_{1\leq j\leq\tau^{(L)}_N}\big),\qquad\forall~N\geq1,
    \]
    as well as the localized environmental $\sigma$-algebras $\mathscr{F}^L_{x,m}\coloneqq\sigma(\omega(y,\vdot):\,\langle y-x,\ell\rangle<-L\abs{\ell}/\abs{\ell}_1,\,\epsilon_i:\,1\leq i\leq m)$ for any $x\in\mathbb{Z}^d$ and $m\geq1$. {\color{red}We} remark that between {\color{red}any two consecutive renewal times} $\tau^{(L)}_{j-1}$ and $\tau^{(L)}_{j}$, the $\epsilon$-sequence {\color{red}generates} a spacing where $(Z_n)_{n\geq0}$ under the law $\overline{Q}^z_0$ travels with full probability at $
    \ell$-direction {\color{red}for a} length $L$. Accordingly, {\color{red}or each $k \ge 1$}, we define the auxiliary transition {\color{red}weight}
    \[
        \psi_k(\CYRu)\coloneqq \prod_{j=\tau^{(L)}_{k-1}+1}^{\tau^{(L)}_{k}} \CYRu_j\;\;\text{with}\;\, \CYRu_j\coloneqq \mathbbm{1}_{ \{\epsilon_j=\Delta_j(Z)\} } + \mathbbm{1}_{ \{\epsilon_j=0\} }\bigg(\frac{\omega(Z_{j-1},\Delta_j(Z))}{\mathbb{E}[\omega(Z_{j-1},\Delta_j(Z))]} + \frac{\CYRk}{u_z(\Delta_j(Z))-\CYRk}\bigg( \frac{\omega(Z_{j-1},\Delta_j(Z))}{\mathbb{E}[\omega(Z_{j-1},\Delta_j(Z))]} - 1 \bigg) \bigg)
    \]
    for each $j\in\mathbb{N}$. Here $\tau^{(L)}_{0}=0$ {\color{red}by convention}. This auxiliary transition {\color{red}construction ensures the crucial identity} 
    \[
        E^{\overline{Q}^z}_0[e^{\langle\theta,Z_n\rangle}\prod_{j=1}^n\CYRu_j] = E^{Q^z}_0[e^{\langle\theta,Z_n\rangle}\prod_{j=1}^n\xi(Z_{j-1},\Delta_j(Z))],\qquad\forall~\theta\in\mathbb{R}^d
    \]
    and $n\geq1$. In particular, {\color{red}this establishes} the following {\color{red}fundamental decoupling} property.

    \begin{proposition}\label{prop: mixing separation information}
        \normalfont
        When $\mathbb{P}$ {\color{red}satisfies} $\textbf{(SMF)}_{C,g}$, for any $L\geq L_0$ we have
        \[            \mathds{E}^{\overline{Q}^z}_0[f(Z_{\tau^{(L)}_{k}+n}-Z_{\tau^{(L)}_{k}},\,n\in\mathbb{N})\mathbb{E}\psi_{k+1}(\CYRu)|\mathscr{G}_{k}] = \mathds{E}^{\overline{Q}^z}_0[f(Z_n-Z_0,\,n\in\mathbb{N})\mathbb{E}\psi_1(\CYRu)],\qquad\forall~k\geq1,
        \]
        where $f:(\mathbb{Z}^d)^{\mathbb{N}}\to[0,\infty)$ is any bounded $\sigma((Z_n)_{n\geq0})$-measurable function, and $\mathds{E}^{\overline{Q}^z}_0$ is the conditional law $E^{\overline{Q}^z}_0[\vdot|\beta^{(\zeta)}_0=\infty]$.
    \end{proposition}
    \begin{proof}
        Without loss of generality, we {\color{red}prove the assertion for} $k=1$. The {\color{red}cases for} $k>1$ follow {\color{red}from an identical argument}. {\color{red}Let $h$ be a} nonnegative bounded $\mathscr{G}_1$-measurable {\color{red}function}. Then,
    \begin{equation*}\begin{aligned}
        &\mathds{E}^{\overline{Q}^z}_0[f(Z_{\tau^{(L)}_{1}+n}-Z_{\tau^{(L)}_{1}},\,n\in\mathbb{N})\mathbb{E}\psi_{2}(\CYRu)h] = \sum_{k=1}^\infty \mathbb{E} E^{\overline{Q}^z}_0[ f (Z_{S^{(L)}_k+\vdot}-Z_{S^{(L)}_k}) h \psi_{2}(\CYRu),\, S^{(L)}_k<\infty,\,R^{(L)}_k=\infty|\beta^{(\zeta)}_0=\infty]\\
        &\qquad = \sum_{k\geq1,m\geq1,x\in\mathbb{Z}^d} \mathbb{E} E^{\overline{Q}^z}_0 [ f (Z_{S^{(L)}_k+\vdot}-x) h_{x,k,m} \psi_{2}(\CYRu),\,Z_{S^{(L)}_k}=x,\, S^{(L)}_k=m,\,\beta^{(\zeta)}_0\circ\vartheta_m=\infty|\beta^{(\zeta)}_0=\infty],
    \end{aligned}\end{equation*}
    where on the event $\{Z(S^{(L)}_k)=x,\,S^{(L)}_k=m,\,\beta^{(\zeta)}_0\circ\vartheta_m=\infty\}$, {\color{red}there exists a bounded function} $h_{x,k,m}$ that is $\mathscr{F}^L_{x,m}\otimes \sigma( (X_i)_{0\leq i\leq m} )$-measurable and equals to $h$ on the same event. Hence, $\mathds{E}^{\overline{Q}^z}_0[f(Z_{\tau^{(L)}_{1}+n}-Z_{\tau^{(L)}_{1}},\,n\in\mathbb{N})\mathbb{E}\psi_{2}(\CYRu)h]$ is equal to
    \[
         \sum_{k\geq1,m\geq1,x\in\mathbb{Z}^d} \mathbb{E}UE^{\overline{Q}^z}_{x,\vartheta_m\epsilon} [ f (Z_{S^{(L)}_k+\vdot}-Z_0) h_{x,k,m} \psi_{2}(\CYRu_{\vartheta_{x,m}(\omega,\epsilon)})|\beta^{(\zeta)}_0=\infty] = \mathds{E}^{\overline{Q}^z}_0[f(Z_n-Z_0,\,n\in\mathbb{N})\mathbb{E}\psi_1(\CYRu)],
    \]
    {\color{red}which proves the claims via the independence of the decoupled environment.}
    \end{proof}
    Observe that by {\color{red}our} choice of $\CYRk$, we have {\color{red}the uniform bounds} $1-\text{dis}(\mathbb{P})/\hbar\leq\CYRu_j\leq1+\text{dis}(\mathbb{P})/\hbar$, $\forall~j\geq1$, $\mathbb{P}\otimes U$-a.s., where $\hbar=\hbar(\kappa,z)\coloneqq1-\CYRk/\inf\{u_\lambda(e):\,e\in\mathbb{V}\}\geq1/2$. Accordingly, let $\CYRb(\vdot)$ {\color{red}stand for} the map $t\mapsto\text{log}(\frac{1+t/\hbar}{1-t/\hbar})$ on $\abs{t}<\hbar$. {\color{red}Define the hitting times} $L_n^{(z)}\coloneqq\inf\{j\geq0:\langle Z_j,\ell\rangle=n\}$ and {\color{red}the free energy} 
    \[
        \CYRPhi^z_n(\theta)\coloneqq \mathds{E}^{\overline{Q}^z}_0\bigg[ e^{\langle \theta,  Z(L^{(z)}_n) \rangle - \overline{\Lambda}^a_z(\theta) L^{(z)}_n  }\prod_{j=1}^{L_n^{(z)}} \CYRu_j,\,L_n^{(z)}=\tau^{(L)}_k\;\;\text{for some}~k\geq1 \bigg],\qquad\forall~\theta\in\mathbb{R}^d \quad\text{and}\quad n\geq1.
    \]

    \begin{proposition}\label{prop: limiting expectation of Phi is positive}    
        \normalfont
        When $\mathbb{P}$ {\color{red}satisfies} $\textbf{(SMF)}_{C,g}$, there exists $\gamma_1=\gamma_1(\kappa,z)>0$ such that whenever $|\theta|\vee\text{dis}(\mathbb{P})<\gamma_1\CYRk^L$,
        we have $\varliminf_{n\to\infty}\mathbb{E} E_U [\CYRPhi^z_n(\theta)]>0$ for any $L\geq L_0$.
    \end{proposition}
    \begin{proof}
        Combining Lemmas \ref{lem: A1}, \ref{lem: A2}, {\color{red}it suffices} to {\color{red}establish that} $\lim_{n\to\infty}\mathbb{E}[\Phi^z_n(\theta,\lambda)]>0$. {\color{red}We} restrict $|\theta|\vee\text{dis}(\mathbb{P})<\gamma_1$ with $\gamma_1=\gamma_1(\kappa,z)$ {\color{red}chosen according to} Lemma \ref{lem: A2}. Define the probability measure $\mu^{(\theta)}$ on $\mathbb{Z}^d$ by
        \[
            \mu^{(\theta)}(x)\coloneqq \mathds{E}^{\overline{Q}^z}_0\bigg[ e^{\langle \theta,Z(\tau^{(L)}_1) \rangle - \overline{\Lambda}^a_z(\theta) \tau^{(L)}_1 } \mathbb{E} \prod_{j=1}^{\tau^{(L)}_1}\CYRu_j,\, \CYRk^L Z(\tau^{(L)}_1)\in\Box_{\ell,x} \bigg],\qquad\forall~x\in\mathbb{Z}^d,
        \]
        where $\Box_{\ell,x}\subseteq\mathbb{R}^d$ {\color{red}denotes} the unit $d$-cube {\color{red}shifted to $x$}, closed {\color{red}on the} $d$ faces {\color{red}in} the direction of $\ell$ and open {\color{red}on the opposite} $d$ faces in the direction {\color{red}of $-\ell$}. Now, consider the random walk $(Z^{(\theta)}_n)_{n\in\mathbb{N}}$ with {\color{red}increment} distribution $\mu^{(\theta)}$. {\color{red}Let $\widehat{P}^{(\theta)}_0$ denote the law of $(Z^{(\theta)}_n)_{n \in \mathbb{N}}$ started at zero, with corresponding expectation $\widehat{E}^{(\theta)}_0$. The limit} $\lim_{n\to\infty}\mathbb{E}E_U[\CYRPhi^z_n(\theta)]$ {\color{red}evaluates} to 
        \begin{equation*}\begin{aligned}
            \lim_{n\to\infty} \sum_{k=1}^\infty\mathds{E}^{\overline{Q}^z}_0\bigg[ e^{\langle \theta,Z(\tau^{(L)}_k)  \rangle - \overline{\Lambda}^a_z(\theta) \tau^{(L)}_k  } \mathbb{E} \prod_{j=1}^{\tau^{(L)}_k}\CYRu_j,\,L^{(z)}_n=\tau^{(L)}_k \bigg]  = \lim_{n\to\infty} \sum_{k=1}^\infty\Hat{P}^{(\theta)}_0(\langle  Z^{(\theta)}_k,\ell\rangle =   n \CYRk^L ) = \CYRk^L/\Hat{E}^{(\theta)}_0[\langle  Z^{(\theta)}_1,\ell\rangle] .
        \end{aligned}\end{equation*}
        {\color{red}To complete the proof,} it {\color{red}remains} to {\color{red}verify that} $\Hat{E}^{(\theta)}_0[\langle Z^{(\theta)}_1,\ell\rangle]<\infty$ {\color{red}under the condition} $|\theta|\vee\text{dis}(\mathbb{P})<\gamma_1\CYRk^L$. {\color{red}Using the elementary estimate}  $\langle   Z^{(\theta)}_1,\ell\rangle\leq\CYRk^L\tau^{(L)}_1\leq(\delta\CYRk^L)^{-1}e^{\delta\CYRk^L\tau^{(L)}_1}$, $\forall~\delta>0$, {\color{red}we obtain}
        \[
            \Hat{E}^{(\theta)}_0[  \langle Z^{(\theta)}_1,\ell\rangle]\leq \frac{1}{\delta\CYRk^L}\mathds{E}^{\overline{Q}^z}_0\bigg[ e^{\langle  \theta,Z(\tau^{(L)}_1)  \rangle - \overline{\Lambda}^a_z(\theta) \tau^{(L)}_1  } \mathbb{E} \prod_{j=1}^{\tau^{(L)}_1}\CYRu_j \bigg]<\infty,
        \]
        {\color{red}provided} $\delta>0$ is {\color{red}chosen} sufficiently small.{\color{red}This completes the verification}.
    \end{proof}

    \begin{proposition}\label{prop: supremum square of Phi is finite}
        \normalfont
        When $\mathbb{P}$ {\color{red}satisfies} $\textbf{(SMF)}_{C,g}$, there exists $\gamma_2=\gamma_2(\kappa,z)>0$ such that whenever $|\theta|\vee\text{dis}(\mathbb{P})<\gamma_2\CYRk^L$, we have $\sup_{n\geq1}\mathbb{E}E_U[\CYRPhi^z_n(\theta)^2]<\infty$ for any $L\geq L_0$.
    \end{proposition}
    \begin{proof}
        {\color{red}The statement} is equivalent to {\color{red}establishing the existence of} $\gamma_2=\gamma_2(\kappa,z)$ such that whenever $|\theta|\vee\text{dis}(\mathbb{P})<\gamma_2\CYRk^L$,
        \begin{equation}\label{eqn: what to show1}        
            \sup_{n\geq1,\abs{\theta}<\gamma_2}\mathds{E}^{\overline{Q}^z}_{0,0}\bigg[ e^{\langle\theta,Z(L^{(z)}_n)+\widetilde{Z}(\widetilde{L}^{(z)}_n)\rangle-\overline{\Lambda}^a_z(\theta)(L^{(z)}_n+\widetilde{L}^{(z)}_n)} \mathbb{E}\prod_{j=1}^{L^{(z)}_n}\CYRu_j\prod_{j=1}^{\widetilde{L}^{(z)}_n}\tilde{\CYRu}_j,\,[ n\CYRk^L]\in\mathcal{L}^{(z)} \bigg]<\infty
        \end{equation}
        where $(Z_n)_{n\geq0}$ and $(\widetilde{Z}_n)_{n\geq0}$ are independent copies of the random walk with law $\mathds{E}^{\overline{Q}^z}_0$, {\color{red}and} $\widetilde{L}^{(z)}_n$, $\widetilde{\tau}^{(L)}_n$ are the analogues of $L^{(z)}_n$ and $\tau^{(L)}_n$ for $(\widetilde{Z}_n)_{n\geq0}$. {\color{red}Here}, $\mathcal{L}^{(z)}\coloneqq\{[ n\CYRk^L]\geq0:\,\langle Z_i,\ell\rangle\geq n,\;\forall~i\geq L^{(z)}_n,\;\langle\widetilde{Z}_j,\ell\rangle\geq n,\;\forall~j\geq\widetilde{L}^{(z)}_n\}$ for $\widetilde{L}^{(z)}_n$ and $\widetilde{L}^{(z)}_n$. {\color{red}And} $\mathds{E}^{\overline{Q}^z}_{x,\tilde{x}}\coloneqq\mathds{E}^{\overline{Q}^z}_x\otimes\mathds{E}^{\overline{Q}^z}_{\tilde{x}}$. {\color{red}For each} $x\in\mathbb{Z}^d$, let $N_x(n)\coloneqq\sum_{j=1}^n\mathbbm{1}_{B_{L\CYRk^L }(\Box_{\ell,x})}(\CYRk^L(Z_{j-1}-Z_0)+Z_0 )$, and {\color{red}similarly} $\widetilde{N}_x(n)$ for $(\widetilde{Z}_n)_{n\geq0}$, for each $n\geq1$. Via an analogous argument, (\ref{eqn: what to show1}) is upper-bounded by 
        \begin{equation}\label{eqn: what to show2}
            \mathscr{A}\coloneqq\sup_{n\geq1,\abs{\theta}<\gamma_2,\nu\in\mathbb{V}_d} A_{\nu,n}(\theta),\qquad\text{where}\quad A_{\nu,n}(\theta)\coloneqq\mathds{E}^{\overline{Q}^z}_{0,\nu}[F_n(\theta),\,[n\CYRk^L] \in\mathcal{L}^{(z)}]
        \end{equation}
        for all $\nu\in\mathbb{V}_d\coloneqq\{\nu\in\mathbb{Z}^d:\,\langle \nu,\ell\rangle=0\}$ and $n\geq1$, {\color{red}where} $F_n(\theta)\coloneqq\phi_n(\theta)\widetilde{\phi}_n(\theta)e^{L\CYRk^L \CYRb(\text{dis}(\mathbb{P}))I_n}$ {\color{red}with}
        \[
            \phi_n(\theta)\coloneqq e^{\langle \theta,(\CYRk^L (Z(L^{(z)}_n)-Z_0)\rangle -\overline{\Lambda}^a_z(\theta) L^{(z)}_n}\mathbb{E}\prod_{j=1}^{L^{(z)}_n}\CYRu_j \qquad\text{and}\qquad I_n\coloneqq\sum_{x\in\mathbb{Z}^d}[N_x(L^{(z)}_n)\wedge\widetilde{N}_x(\widetilde{L}^{(z)}_n)],
        \]
        {\color{red}and} $\widetilde{\phi}_n$ defined analogously {\color{red}for} $(\widetilde{Z}_n)_{n\geq0}$. {\color{red}The assertion then follows directly from} Lemma \ref{lem: what to show2}.
    \end{proof}

    \begin{proposition}\label{prop: equality between log moment generating functions}
        \normalfont
        Define $\Bar{\gamma}=\Bar{\gamma}(\kappa,z)>0$ by $\Bar{\gamma}\coloneqq\gamma_1\wedge\gamma_2$. Then when $\mathbb{P}$ {\color{red}satisfies} $\textbf{(SMF)}_{C,g}$, we have $\overline{\Lambda}^q_z(\theta) = \overline{\Lambda}^a_z(\theta)$ for all $|\theta|\vee\text{dis}(\mathbb{P})<\bar{\gamma}\CYRk^L$.
    \end{proposition}
    \begin{proof}
        By Propositions \ref{prop: limiting expectation of Phi is positive}, \ref{prop: supremum square of Phi is finite}, $\mathbb{P}(\lim_{n\to\infty}\Phi^z_n(\theta)=0)<1$ whenever $|\theta|\vee\text{dis}(\mathbb{P})<\bar{\gamma}\CYRk^L$. Indeed, if $\CYRPhi^z_n(\theta)\to0$ $\mathbb{P}\otimes U$-a.s., then we should have $\mathbb{E}E_U[\CYRPhi^z_n(\theta)]\to0$ as $n\to\infty$, since $(\CYRPhi^z_n(\theta))_{n\geq1}$ is {\color{red}uniformly} integrable by Proposition \ref{prop: supremum square of Phi is finite}, which {\color{red}then} contradicts the conclusion in Proposition \ref{prop: limiting expectation of Phi is positive}. {\color{red}Consequently, we obtain}
        \[
            \mathbb{P}\otimes U\bigg(\varlimsup_{n\to\infty} E^{\overline{Q}^\lambda}_0\bigg[ e^{\langle \theta, Z(L^{(z)}_n)\rangle- \overline{\Lambda}^a_z(\theta,\lambda) L^{(z)}_n} \prod_{j=1}^{L^{(z)}_n}\CYRu_j \bigg]>0\bigg)>0
        \]
        whenever $|\theta|\vee\text{dis}(\mathbb{P})<\bar{\gamma}\CYRk^L$. By Lemma \ref{lem: A3}, we can find one realization $(\omega,\epsilon)$ of the environment and $n\geq1$ such that
        \[
            E^{\overline{Q}^z}_0\bigg[e^{\langle \theta, Z(L^{(z)}_n)\rangle-(\overline{\Lambda}^q_z(\theta)+ \delta)  L^{(z)}_n}\prod_{j=1}^{L^{(z)}_n}\CYRu_j\bigg]< E^{\overline{Q}^z}_0\bigg[e^{\langle \theta, Z(L^{(z)}_n)\rangle- \overline{\Lambda}^a_z(\theta) L^{(z)}_n}\prod_{j=1}^{L^{(z)}_n}\CYRu_j \bigg],\qquad\forall~L\geq L_0.
        \]
        from which it follows that $\overline{\Lambda}^q_z(\theta)+\delta>\overline{\Lambda}^a_z(\theta)$. Letting $\delta\to0$, we get $\overline{\Lambda}^q_z(\theta)\geq\overline{\Lambda}^a_z(\theta)$ when  $|\theta/\lambda^{1/2}|\vee\text{dis}(\mathbb{P})<\bar{\gamma}\CYRk^L$. {\color{red}On the other hand,} $\overline{\Lambda}^q_z(\theta)\leq\overline{\Lambda}^a_z(\theta)$ for all $\theta\in\mathbb{R}^d$  by Jensen's inequality. {\color{red}Combining both inequalities completes the proof of the assertion.}
    \end{proof}

\subsection{Regularity of the differentials and open mappings}
    \begin{proposition}\label{prop: log MGF is differentiable}
        \normalfont
        {\color{red}Let} $\gamma_1=\gamma_1(\kappa,z)$ {\color{red}be as specified in Proposition~\ref{prop: limiting expectation of Phi is positive}}. When $\mathbb{P}$ {\color{red}satisfies} $\textbf{(SMF)}_{C,g}$, if $\text{dis}(\mathbb{P})<\gamma_1\CYRk^L$, then the mapping $\theta\mapsto\overline{\Lambda}^a_z(\theta)$ is analytic on the {\color{red}open} region $\{|\theta|<\gamma_1\CYRk^L\}\subseteq\mathbb{R}^d$ {\color{red}for any} $L\geq L_0$.
    \end{proposition}
    \begin{proof}
        Consider the {\color{red}auxiliary function} $\tilde{\psi}^z:\mathbb{R}^d\times\mathbb{R}\to\mathbb{R}$ defined by $\tilde{\psi}^z(\theta,r)\coloneqq \mathds{E}^{\overline{Q}^z}_0 [ e^{\langle \theta,Z(\tau^{(L)}_1)\rangle-r\tau^{(L)}_1}\mathbb{E} \prod_{j=1}^{\tau^{(L)}_1}\CYRu_j ]$. {\color{red}Setting} $r= \overline{\Lambda}^a_z(\theta) +\delta$ for some $\delta\in\mathbb{R}$, we {\color{red}observe the pointwise estimate} $|\langle\theta,Z(\tau^{(L)}_1 )\rangle-r \tau^{(L)}_1+\log \mathbb{E}\prod_{j=1}^{\tau^{(L)}_1}\CYRu_j | \le  (2|\theta|+\CYRb(\text{dis}(\mathbb{P}))+\abs{\delta})\tau^{(L)}_1$, where $\CYRb(\vdot)$ is {\color{red}the logarithm function defined prior to} Proposition \ref{prop: limiting expectation of Phi is positive}. By the choice of $\gamma_1(\kappa,z)>0$ specified in Proposition \ref{prop: limiting expectation of Phi is positive}, we have $\mathds{E}^{\overline{Q}^z}_0[\tau^{(L)}_1 e^{\langle \theta,Z(\tau^{(L)}_1)\rangle-(\overline{\Lambda}^a_z(\theta) + \delta)\tau^{(L)}_1}\mathbb{E} \prod_{j=1}^{\tau^{(L)}_1}\CYRu_j  ]<\infty$, whenever $|\theta|\vee\text{dis}(\mathbb{P})<\gamma_1\CYRk^L$ and $\abs{\delta}<\delta_z$ for some $\delta_z=\delta_z(\kappa)>0$ independent of $L\geq L_0$. It then follows from the dominated convergence that {\color{red}whenever} $\text{dis}(\mathbb{P})<\gamma_1\CYRk^L$, the {\color{red}mapping} $\tilde{\psi}^z$ is analytic on the region $\mathcal{C}_z\coloneqq\{(\theta,r):\,|\theta|<\gamma_1\CYRk^L,\,|r-\overline{\Lambda}^a_z(\theta)|<\delta_z\}$ with the following {\color{red}local} series expansion
        \[
            \tilde{\psi}^z(\theta,r)=\sum_{n=0}^\infty \frac{1}{n!}\mathds{E}^{\overline{Q}^z}_0\big[ (\langle\theta, Z_{\tau^{(L)}_1}\rangle-r\tau^{(L)}_1)^n\mathbb{E}\prod_{j=1}^{\tau^{(L)}_1}\CYRu_j \big],\qquad\forall~(\theta,r)\in\mathcal{C}_z
        \]
        and its differential $\partial_r\tilde{\psi}^z$ {\color{red} with respect to $r$} given by
        \[
            -\frac{\partial}{\partial r} \tilde{\psi}^z(\theta,r) = \mathds{E}^{\overline{Q}^z}_0\bigg[ \tau^{(L)}_1 e^{\langle \theta,Z(\tau^{(L)}_1)\rangle-r\tau^{(L)}_1} \mathbb{E}\prod_{j=1}^{\tau^{(L)}_1}\CYRu_j \bigg],\qquad\forall~(\theta,r)\in\mathcal{C}_z.
        \]
        {\color{red}Recall from} Proposition \ref{prop: limiting expectation of Phi is positive} that $\tilde{\psi}^z(\theta,\widetilde{\Lambda}^a_z(\theta))=1$ whenever $|\theta|\vee\text{dis}(\mathbb{P})<\gamma_1\CYRk^L$. {\color{red}In particular, we have} $-\partial_r\tilde{\psi}^z(\theta,\widetilde{\Lambda}^a_z(\theta))\geq\tilde{\psi}^z(\theta,\widetilde{\Lambda}^a_z(\theta))>0$. {\color{red}Since $\frac{\partial}{\partial r} \tilde{\psi}^z(\theta, r) \neq 0$ on $\mathcal{C}_z$}, the analytic {\color{red}implicit function theorem applies, guaranteeing that} $\theta\mapsto\overline{\Lambda}^a_z(\theta)$ {\color{red}is analytic} for $|\theta|<\gamma_1\CYRk^L$ whenever $\text{dis}(\mathbb{P})<\gamma_1\CYRk^L$.
    \end{proof}

    \begin{proposition}\label{prop: distance between lambda z and differential}
        \normalfont
        When $\mathbb{P}$ {\color{red}satisfies} $\textbf{(SMF)}_{C,g}$, for any $\delta>0$, there exists $\epsilon_1=\epsilon_1(\kappa,z,\delta)>0$ such that $|\nabla\overline{\Lambda}^a_z(0)- z|<\delta$ whenever $\text{dis}(\mathbb{P})<\epsilon_1\CYRk^L$.
    \end{proposition}
    \begin{proof}
        {\color{red}Differentiating the implicit relation $\tilde{\psi}^z(\theta, \overline{\Lambda}^a_z(\theta)) = 1$ with respect to $\theta$} (analogous to \cite[Eqn. (4.2)]{Bazaes/Mukherjee/Ramirez/Sagliett0}) {\color{red}we obtain the explicit expression for the gradient of the annealed log MGF:}
        \begin{equation}\label{eqn: differential of annealed log MGF}                       \nabla \overline{\Lambda}^a_z(\theta) =  \frac{\mathds{E}^{\overline{Q}^z}_0[\CYRk^LZ(\tau^{(L)}_1)e^{\langle \theta,Z(\tau^{(L)}_1)\rangle- \overline{\Lambda}^a_z(\theta)\tau^{(L)}_1}\mathbb{E}\prod_{j=1}^{\tau^{(L)}_1}\CYRu_j]}{\mathds{E}^{\overline{Q}^z}_0[\CYRk^L\tau^{(L)}_1 e^{\langle \theta, Z(\tau^{(L)}_1)\rangle- \overline{\Lambda}^a_z(\theta) \tau^{(L)}_1}\mathbb{E}\prod_{j=1}^{\tau^{(L)}_1}\CYRu_j]},\qquad\forall~|\theta|<\gamma_1\CYRk^L,
        \end{equation}
        {\color{red}provided} $\text{dis}(\mathbb{P})<\gamma_1\CYRk^L$. The law of large numbers then yields $ z = \mathds{E}^{\overline{Q}^z}_0[Z(\tau^{(L)}_1)]/\mathds{E}^{\overline{Q}^z}_0[\tau^{(L)}_1]$. {\color{red}Since} $\mathds{E}^{\overline{Q}^z}_0[\CYRk^L\tau^{(L)}_1]\geq c_\kappa$ for some $c_\kappa=\CYRk^L>0$, it suffices to {\color{red}show that for any given} $\delta^\prime>0$, {\color{red}there exists} $\epsilon_1^\prime=\epsilon_1^\prime(\kappa,z,\delta^\prime)$ such that {\color{red}whenever} $\text{dis}(\mathbb{P})<\delta^\prime\CYRk^L$, {\color{red}evaluating \eqref{eqn: differential of annealed log MGF} at $\theta = 0$ yields the following two estimates:}
        \[
            \bigg| \mathds{E}^{\overline{Q}^z}_0\bigg[\CYRk^LZ(\tau^{(L)}_1)e^{\langle \theta, Z(\tau^{(L)}_1)\rangle- \overline{\Lambda}^a_z(\theta)\tau^{(L)}_1}\mathbb{E}\prod_{j=1}^{\tau^{(L)}_1}\CYRu_j\bigg] - \mathds{E}^{\overline{Q}^z}_0\big[\CYRk^LZ(\tau^{(L)}_1)\big] \bigg|\leq \delta^\prime,\qquad\forall~L\geq L_0
        \]
        as well as when $L\geq L_0$,
        \begin{equation}\label{eqn: 2nd to prove, intermediate}        
            \bigg| \mathds{E}^{\overline{Q}^z}_0\bigg[\CYRk^L\tau^{(L)}_1e^{\langle \theta,Z(\tau^{(L)}_1)\rangle- \overline{\Lambda}^a_z(\theta)\tau^{(L)}_1}\mathbb{E}\prod_{j=1}^{\tau^{(L)}_1}\CYRu_j\bigg] - \mathds{E}^{\overline{Q}^\lambda}_0\big[\CYRk^L\tau^{(L)}_1\big] \bigg|\leq \delta^\prime
        \end{equation}
        when $\theta=0$. Indeed, {\color{red}applying the pointwise estimate from} \cite[Remark 6]{Bazaes/Mukherjee/Ramirez/Sagliett0} and the mean value theorem, we {\color{red}bound the numerator difference by}
        \[
            \bigg| \mathds{E}^{\overline{Q}^z}_0\bigg[\CYRk^L Z_{\tau^{(L)}_1}e^{- \overline{\Lambda}^a_z(0)\tau^{(L)}_1}\mathbb{E}\prod_{j=1}^{\tau^{(L)}_1}\CYRu_j\bigg] - \mathds{E}^{\overline{Q}^z}_0\big[\CYRk^L Z_{\tau^{(L)}_1}\big] \bigg|\leq \CYRb(\text{dis}(\mathbb{P})) \mathds{E}^{\overline{Q}^z}_0\bigg[ \CYRk^L (\tau^{(L)}_1)^2 e^{\CYRb(\text{dis}(\mathbb{P}))\tau^{(L)}_1}\bigg].
        \]
        Hence the upper bound now follows from Proposition \ref{prop: mixing exponential of tau < infinity}. The estimate of (\ref{eqn: 2nd to prove, intermediate}) {\color{red}proceeds in an identical manner}. Thus, we have shown $\forall~\delta>0$, $\exists~\epsilon_1^\prime(\kappa,z,\delta)>0$ s.t. $|\nabla\overline{\Lambda}^a_z(0)-z|<\delta/2$ if $\text{dis}(\mathbb{P})<\epsilon_1^\prime\CYRk^L$. 
    \end{proof}

    \begin{proposition}\label{prop: Hessian matrix is positive definite}
        \normalfont
        {\color{red}Let} $\gamma_1=\gamma_1(\kappa,z)>0$ {\color{red}be as specified in Proposition~\ref{prop: limiting expectation of Phi is positive}.} If $|\theta|\vee\text{dis}(\mathbb{P})<\gamma_1\CYRk^L$, when $\mathbb{P}$ {\color{red}satisfies} $\textbf{(SMF)}_{C,g}$, the Hessian $H^a_z(\theta)$ of $\overline{\Lambda}^a_z(\theta)$ is positive definite. And {\color{red}the image set} $\mathcal{A}^{\mathbb{P}}_{z}\coloneqq\{\nabla\overline{\Lambda}^a_z(\theta):\,\abs{\theta}<\bar{\gamma}\CYRk^L\}$ is open in $\mathbb{R}^d$.
    \end{proposition}  
    \begin{proof}
        {\color{red}Differentiating the gradient formula} (\ref{eqn: differential of annealed log MGF}) {\color{red}with respect to $\theta$}, we obtain an {\color{red}explicit} expression for the Hessian {\color{red}matrix} $H^a_z(\theta)$. In particular, for any {\color{red}non-zero} vector $v\in\mathbb{R}^d$, {\color{red}the quadratic form evaluates to}
        \[
            \langle v,H^a_z(\theta)\vdot v\rangle =  \frac{\mathds{E}^{\overline{Q}^z}_0[|\langle Z_{\tau^{(L)}_1}-\nabla\overline{\Lambda}^a_z(\theta)\tau^{(L)}_1,\CYRk^L v\rangle|^2 e^{\langle \theta, Z(\tau^{(L)}_1)\rangle- \overline{\Lambda}^a_z(\theta)\tau^{(L)}_1}\mathbb{E}\prod_{j=1}^{\tau^{(L)}_1}\CYRu_j]}{\mathds{E}^{\overline{Q}^z}_0[\CYRk^L\tau^{(L)}_1 e^{\langle \theta,Z(\tau^{(L)}_1)\rangle- \overline{\Lambda}^a_z(\theta)\tau^{(L)}_1}\mathbb{E}\prod_{j=1}^{\tau^{(L)}_1}\CYRu_j]}.
        \]
        whenever $|\theta|\vee\text{dis}(\mathbb{P})<\gamma_1\CYRk^L$. Henceforth $\langle v,H^a_z(\theta)\vdot v\rangle\geq0$ and equality holds iff $\langle Z(\tau^{(L)}_1)-\nabla\overline{\Lambda}^a_z(\theta)\tau^{(L)}_1,v\rangle = 0$, $\mathds{E}^{\overline{Q}^z}_0$-a.s., i.e.,~$\langle Z(\tau^{(L)}_1)/\tau^{(L)}_1,v\rangle$ is $\mathds{E}^{\overline{Q}^\lambda}_0$-a.s. constant. {\color{red}However, such a linear degeneracy is ruled out by the $\kappa$-uniform ellipticity of the random walk environment $\mathbb{P}$, which guarantees genuine $d$-dimensional fluctuations of the trajectory.} Hence $H^a_z(\theta)$ is positive definite and the openness of $\mathcal{A}^{\mathbb{P}}_{z}$ follows from inverse mapping theorem.
    \end{proof}

    \begin{proposition}\label{prop: region A contains an open ball}
        \normalfont
        When $\mathbb{P}$ {\color{red}satisfies} $\textbf{(SMF)}_{C,g}$, there exist $\epsilon=\epsilon(\kappa,z)$ and $r=r(\kappa,z)>0$ such that whenever $\text{dis}(\mathbb{P})<\epsilon\CYRk^L$, the open ball $B_r(z)\subseteq\mathcal{B}^{\mathbb{P}}_{z} \coloneqq\{\nabla\overline{\Lambda}^a_z(\theta):\,\abs{\theta}<\bar{\gamma}\CYRk^L\}$.
    \end{proposition}
    \begin{proof}
        {\color{red}Our strategy is to} first {\color{red}establish that} $B_{r^\prime}(z)\subseteq\mathcal{A}^{\mathbb{P}}_{z} = \{\nabla\overline{\Lambda}^a_z(\theta):\,\abs{\theta}<\bar{\gamma}\CYRk^L\}$ for some radius $r^\prime(\kappa,z)>0$ {\color{red}by applying a quantitative version of the Inverse Mapping Theorem, and then deduce the full proposition} via Lemma \ref{lem: B0}. It suffices to verify {\color{red}that} $\exists~\epsilon^\prime=\epsilon^\prime(\kappa,z)>0$ s.t. the family of $C^1$ {\color{red}gradients} $\mathcal{F}^{(z)}\coloneqq\{\nabla\overline{\Lambda}^z_a:\,\text{dis}(\mathbb{P})<\epsilon^\prime\}$ satisfies the conditions of \cite[Theorem 4.5]{Bazaes/Mukherjee/Ramirez/Sagliett0}. {\color{red}First,} Lemma \ref{lem: B1} {\color{red}asserts} that $\forall~c>0$, $\exists~\epsilon_1=\epsilon_1(\kappa,z,c)>0$ such that $\sup_{|\theta|\vee\text{dis}(\mathbb{P})<\epsilon_1}|\nabla\overline{\Lambda}^z_a(\theta)-\nabla\overline{\Lambda}^z_a(0)|<c$. {\color{red}Next, by} Lemma \ref{lem: B2}, {\color{red}for any} $c>0$, {\color{red}there exists} $\epsilon_2=\epsilon_2(\kappa,z,c)>0$ such that $\text{dis}(\mathbb{P})<\epsilon_2$ {\color{red}implies}
        \begin{equation}\label{eqn: mixing dogdog2*}
            \normx{H^a_z(0)-H^*_a(0)}<c,\qquad  H^*_z(0)\coloneqq\frac{\mathds{E}^{\overline{Q}^z}_0[( Z(\tau^{(L)}_1)-z \tau^{(L)}_1)^T ( Z(\tau^{(L)}_1)-z \tau^{(L)}_1)]}{\mathds{E}^{\overline{Q}^z}_0[  \tau^{(L)}_1 ]}.
        \end{equation}
        Furthermore, Lemma \ref{lem: B3} {\color{red}guarantees} that $\sup_{\mathbb{P}}\normx{H^*_z(0)^{-1}}<\infty$, where the supremum is taken over all $\kappa$-uniformly elliptic $\mathbb{P}$ {\color{red}satisfying} $\textbf{(SMF)}_{C,g}$. With (\ref{eqn: dogdog2}), we {\color{red}obtain a threshold} $\epsilon_2^\prime=\epsilon_2^\prime(\kappa,z)>0$ {\color{red}so that} $\normx{H^a_z(0)-H^*_z(0)}\leq(2\sup_{\mathbb{P}}\normx{H^*_z(0)^{-1}})^{-1}$ whenever $\text{dis}(\mathbb{P})<\epsilon_2^\prime$. {\color{red}Using the resolvent identity for inverse matrices, we have}
        \[
            \normx{H^a_z(0)^{-1}-H^*_z(0)^{-1}}\leq \normx{H^a_z(0)^{-1}}\normx{H^a_z(0)-H^*_z(0)}\normx{H^*_z(0)^{-1}}\leq \normx{H^a_z(0)^{-1}}/2.
        \]
        {\color{red}By the triangle inequality,} $\normx{H^a_z(0)^{-1}}\leq\normx{H^a_z(0)^{-1}}/2+\normx{H^*_z(0)^{-1}}$, which {\color{red}directly yields the uniform bound} $\normx{H^a_z(0)^{-1}}\leq 2\sup_{\mathbb{P}}\normx{H^*_a(0)^{-1}}$, {\color{red}thereby establishing} \cite[Theorem 4.5, Condition (I1)]{Bazaes/Mukherjee/Ramirez/Sagliett0}. Define
        \[
            \Gamma(v)\coloneqq(  Z_{\tau^{(L)}_1}-v  \tau^{(L)}_1)^T( Z_{\tau^{(L)}_1}-v \tau^{(L)}_1),\qquad \varphi(\theta)\coloneqq e^{\langle \theta, Z(\tau^{(L)}_1)\rangle - \overline{\Lambda}^a_z(\theta) \tau^{(L)}_1}\mathbb{E}\prod_{j=1}^{\tau^{(L)}_1}\CYRu_j,\qquad\forall~v,\theta\in\mathbb{R}^d.
        \]
        Now  $\normx{\Gamma(\nabla\overline{\Lambda}^a_z(\theta)) - \Gamma(\nabla\overline{\Lambda}^a_z(0))}\leq 5( | \nabla\overline{\Lambda}^a_z(\theta) - \nabla\overline{\Lambda}^a_z(0) |)(\tau^{(L)}_1)^2$ and $|\varphi(\theta)-\varphi(0)|\leq 2(|\theta|+\CYRb(\text{dis}(\mathbb{P}) )) \tau^{(L)}_1 e^{2(|\theta|+\CYRb(\text{dis}(\mathbb{P}) ))\tau^{(L)}_1}$. Invoking (\ref{eqn: dogdog1}) and Proposition \ref{prop: mixing exponential of tau < infinity}, we derive $\forall~c>0$ that
        \[
            \exists~\epsilon^\prime_1=\epsilon^\prime_1(\kappa,z,c)>0\qquad\text{s.t.}\quad \sup_{ |\theta|\vee\text{dis}(\mathbb{P})<\epsilon^\prime_1\CYRk^L} \normy{ \mathds{E}^{\overline{Q}^z}_0[ \Gamma(\nabla\overline{\Lambda}^a_z(\theta))\varphi(\theta)] - \mathds{E}^{\overline{Q}^z}_0[ \Gamma(\nabla\overline{\Lambda}^a_z(0))\varphi(0)] }< c,
        \]
        together with $\mathds{E}^{\overline{Q}^z}_0[\CYRk^L\tau^{(L)}_1]\geq c_\kappa$, we {\color{red}verify} \cite[Theorem 4.5, Condition (I2)]{Bazaes/Mukherjee/Ramirez/Sagliett0}, and an application of \cite[Theorem 9.24, Eqn. (46-49)]{Rudin} yields $B_{ r^{\prime\prime}}(\nabla\overline{\Lambda}^a_z(0))\subseteq\mathcal{A}^{\mathbb{P}}_{z}$ for some $r^{\prime\prime}(\kappa,z)>0$. {\color{red}Finally, from} Proposition \ref{prop: distance between lambda z and differential}, we can find some $r^\prime(\kappa,z)<r^{\prime\prime}$ so that $B_{ r^{\prime }}( z)\subseteq\mathcal{A}^{\mathbb{P}}_{z }$. And the assertion is verified.
    \end{proof}

    \begin{proof}[Proof of Theorem \ref{thm: equality of LDP in interior}]
        {\color{red}Fix} $x\in B_{ r}(z)$, {\color{red}where $B_r(z)$ is the open ball guaranteed by Proposition~\ref{prop: region A contains an open ball}.} Define {\color{red}the Fenchel--Legendre transforms} $\tilde{I}_q(x)\coloneqq\sup_{\theta\in\mathbb{R}^d}\{\langle\theta,x\rangle-\overline{\Lambda}^q_z(\theta)\}$ and $\tilde{I}_a(x)\coloneqq\sup_{\theta\in\mathbb{R}^d}\{\langle\theta,x\rangle-\overline{\Lambda}^a_z(\theta)\}$. A standard derivation of \cite[Lemma 2.3.9]{Dembo/Zeitouni} tells $\tilde{I}_q(x)=\langle\theta^{q}_{\mathbb{P}},x\rangle-\overline{\Lambda}^q_z(\theta^{q}_{\mathbb{P}})$ and $\tilde{I}_a(x)=\langle\theta^{a}_{\mathbb{P}},x\rangle-\overline{\Lambda}^a_z(\theta^{a}_{\mathbb{P}})$ for any $\theta^{q}_{\mathbb{P}}$ and $\theta^{a}_{\mathbb{P}}$ satisfying $\nabla\overline{\Lambda}^q_z(\theta^{q}_{\mathbb{P}}) = \nabla\overline{\Lambda}^a_z(\theta^{a}_{\mathbb{P}},)=x$. Notice that such $\theta^{a}_{\mathbb{P}}$ exists and satisfies $|\theta^{a}_{\mathbb{P}}|<\bar{\gamma}\CYRk^L$ since $x\in\mathcal{A}^{\mathbb{P}}_{z}$. Also, $\theta^{q}_{\mathbb{P}}$ exists and equals to $\theta^{a}_{\mathbb{P}}$ since $\overline{\Lambda}^q_z(\theta)=\overline{\Lambda}^a_z(\theta)$, $\forall~|\theta|<\bar{\gamma}\CYRk^L$. Then, $\tilde{I}_q(x)=\tilde{I}_a(x)$, $\forall~x\in B_{ r}( z)$. {\color{red}Recalling the shift relation between the original log moment generating functions and their localized counterparts, we obtain}
        \[
            \Tilde{I}_q(x)+\log \mathcal{D}_{z,\mathbb{P}} +\langle\theta^{\mathbb{P}}_{z },x\rangle = \sup_{\theta\in\mathbb{R}^d}\{\langle \theta+ \theta^{\mathbb{P}}_{z },x\rangle - \Lambda_q(\theta+ \theta^{\mathbb{P}}_{z })  \} = I_q(x),\; \Tilde{I}_a(x)+\log \mathcal{D}_{z,\mathbb{P}} +\langle\theta^{\mathbb{P}}_{z },x\rangle =I_a(x),\quad \forall~x\in B_{ r}(z).
        \]
        Henceforth whenever $\text{dis}(\mathbb{P})<\bar{\gamma}\CYRk^L$ we conclude $I_q(x)=I_a(x)$, $\forall~x\in B_{ r}( z)$, for any $z\in\text{int}(\mathbb{D})$ with $\abs{z}>0$. Now we {\color{red}select} $L>L_0$ depending on $\kappa,z$. {\color{red}For} any compact $\mathcal{K}\subseteq\text{int}(\mathbb{D})\backslash\{0\}$, {\color{red}choose} finitely many nonzero $z_1,\ldots,z_m\in\mathcal{K}$ such that $\mathcal{K}\subseteq\cup_{i\leq m}B_{\tilde{r}(z_i)}(z_i)$. {\color{red}Defining the uniform threshold} $\gamma_1(d,\kappa,\mathcal{K})\coloneqq \bar{\gamma}(z_1)\wedge\cdots\wedge\bar{\gamma}(z_m)>0$, we then have $I_a =I_q $ on $\mathcal{K}$ provided $\text{dis}(\mathbb{P})<\gamma_1\CYRk^L$ with fixed $L$ Hence the assertion is proved.
    \end{proof}

\subsection{Non-conformality of rate functions in the domain interior.}\label{sec: 5.4} This section is dedicated to {\color{red}the proof of} Theorem \ref{thm: non-equality of rate functions}. We introduce the $Q$-stopping sequence $(\tilde{\tau}^{(L)}_n)_{n\geq0}$ {\color{red}defined recursively} by $\tilde{\tau}^{(L)}_0=0$ and
    \[
        \tilde{\tau}^{(L)}_{n}\coloneqq\inf\{j\geq\tilde{\tau}^{(L)}_{n-1}+L:\,(\epsilon_{j-L},\ldots,\epsilon_{j-1}) = +s,\ldots,s\},\qquad s=\ell, \qquad\forall~n\geq1.
    \]
    Between each {\color{red}successive stopping times} $\tilde{\tau}^{(L)}_{j-1}$ and $\tilde{\tau}^{(L)}_{j}$, the $\epsilon$-sequence creates a spacing {\color{red}during which the process} $(Z_n)_{n\geq0}$ under the law $\overline{Q}^z_0$ travels with full probability at $
    \ell$-direction of length $L$. Accordingly, we define the auxiliary transition
    \[
        \psi_k(\CYRu)\coloneqq \prod_{j=\tilde{\tau}^{(L)}_{k-1}+1}^{\tilde{\tau}^{(L)}_{k}} \CYRu_j\;\;\text{with}\;\, \CYRu_j\coloneqq \mathbbm{1}_{ \{\epsilon_j=\Delta_j(Z)\} } + \mathbbm{1}_{ \{\epsilon_j=0\} }\bigg(\frac{\omega(Z_{j-1},\Delta_j(Z))}{\mathbb{E}[\omega(Z_{j-1},\Delta_j(Z))]} + \frac{\CYRk}{u_z(\Delta_j(Z))-\CYRk}\bigg( \frac{\omega(Z_{j-1},\Delta_j(Z))}{\mathbb{E}[\omega(Z_{j-1},\Delta_j(Z))]} - 1 \bigg) \bigg)
    \]
    for each $k,j\in\mathbb{N}$. This {\color{red}construction yields the identity} 
    \[
        E^{\overline{Q}^z}_0[e^{\langle\theta,Z_N\rangle}\prod_{j=1}^N\CYRu_j] = E^{Q^z}_0[e^{\langle\theta,Z_N\rangle}\prod_{j=1}^N\xi(Z_{j-1},\Delta_j(Z))],\qquad\forall~\theta\in\mathbb{R}^d,
    \]
    and $N\geq1$. Denote by $\mathcal{S}_N$ the set of all nearest-neighbor paths of length $N$ on $\mathbb{Z}^d$. Then $\forall~z\in\text{int}(\mathbb{D})\backslash\{0\}$ and $\ell\in\mathbb{V}$,
    \[
        E^{\overline{Q}^z}_0\bigg[ e^{-\langle\theta^{\mathbb{P}}_z,Z_N\rangle}\prod_{j=1}^N\CYRu_j I_{ \{ Z_j=\ell j \} } \bigg] = \mathcal{D}^N_{z,\mathbb{P}}P_{0,\omega}(X_N=\ell N),\quad E^{\overline{Q}^z}_0\bigg[ e^{-\langle\theta^{\mathbb{P}}_z,Z_N\rangle}\mathbb{E}\prod_{j=1}^N\CYRu_j I_{ \{ Z_j=\ell j \} } \bigg] = \mathcal{D}^N_{z,\mathbb{P}}P_{0}(X_N=\ell N),
    \]
    for any $N\geq1$. {\color{red}Consequently},
    \begin{equation}\begin{aligned}\label{eqn: express Ia}  
        -I_a(\ell) &\geq \varlimsup_{N\to\infty} \frac{1}{N}\log P_0(X_N=\ell N) \geq -\text{log}( e^{\langle \theta_{z}^{\mathbb{P}},\ell\rangle} \mathcal{D}_{z,\mathbb{P}} ) + \varlimsup_{N\to\infty} \frac{1}{N}\log E^{\overline{Q}^z}_0\bigg[\mathbb{E}\prod_{j=1}^N \CYRu_j I_{ \{ Z_j=\ell j \} } \bigg]\\
        &\geq -\text{log}( e^{\langle \theta_{z}^{\mathbb{P}},\ell\rangle} \mathcal{D}_{z,\mathbb{P}} ) + \varlimsup_{k\to\infty} \frac{1}{E^{\overline{Q}^z}_0[\tilde{\tau}^{(L)}_k]}\log E^{\overline{Q}^z}_0 \bigg[\mathbb{E}\prod_{j=1}^{\tilde{\tau}^{(L)}_k} \CYRu_j I_{ \{ Z_j=\ell j \} } \bigg].
   \end{aligned}\end{equation}
   where the {\color{red}final inequality} follows from Lemma \cite[Lemma 4.4]{Chen}. In the quenched scenario, $\forall~\epsilon>0$ {\color{red}the upper-bound can be estimated by}
   \begin{equation*}\begin{aligned}
        -I_q(\ell) &\leq \varliminf_{N\to\infty} \frac{1}{N}\log P_{0,\omega}(\langle X_N,\ell\rangle>(1-\epsilon/E^{\overline{Q}^z}_0[\tilde{\tau}^{(L)}_1])N)\leq -\text{log}( e^{\langle \theta_{z}^{\mathbb{P}},\ell\rangle} \mathcal{D}_{z,\mathbb{P}} )\\
        &\quad+ \varliminf_{N\to\infty} \frac{1}{N}\log E^{\overline{Q}^z}_0 \bigg[\prod_{j=1}^N \CYRu_j,\,\langle Z_N,\ell\rangle>(1-\epsilon/E^{\overline{Q}^z}_0[\tilde{\tau}^{(L)}_1])N \bigg] +\epsilon\log e^{\langle \theta_{z}^{\mathbb{P}},\ell\rangle}.
    \end{aligned}\end{equation*}
    {\color{red}Restricting to the renewal epochs $\tilde{\tau}^{(L)}_k$,} we {\color{red}obtain}
    \begin{equation*}\begin{aligned}
        -I_q(\ell) &\leq -\text{log}( e^{\langle \theta_{z}^{\mathbb{P}},\ell\rangle} \mathcal{D}_{z,\mathbb{P}} ) + \varliminf_{k\to\infty} \frac{1}{E^{\overline{Q}^z}_0[\tilde{\tau}^{(L)}_k]}\log E^{\overline{Q}^z}_0\bigg[\prod_{j=1}^{\tilde{\tau}^{(L)}_k} \CYRu_j,\,\langle Z(\tilde{\tau}^{(L)}_k),\ell\rangle> (1-\epsilon/E^{\overline{Q}^z}_0[\tilde{\tau}^{(L)}_1])\tilde{\tau}^{(L)}_k  \bigg]  + \epsilon\log e^{\langle \theta_{z}^{\mathbb{P}},\ell\rangle}\\
        &\leq -\text{log}( e^{\langle \theta_{z}^{\mathbb{P}},\ell\rangle} \mathcal{D}_{z,\mathbb{P}} ) + (1-\epsilon/E^{\overline{Q}^z}_0[\tilde{\tau}^{(L)}_1]) \varliminf_{k\to\infty} \frac{1}{E^{\overline{Q}^z}_0[\tilde{\tau}^{(L)}_k]}\log E^{\overline{Q}^z}_0\bigg[\prod_{j=1}^{\tilde{\tau}^{(L)}_k} \CYRu_j I_{ \{ Z_j=\ell j \} } \bigg] + \epsilon\log e^{\langle \theta_{z}^{\mathbb{P}},\ell\rangle} + O(\epsilon),
    \end{aligned}\end{equation*}
    since for each $N\geq1$, the number of paths constituting the event $\{\langle Z_N,\ell\rangle>(1-\epsilon/E^{\overline{Q}^z}_0[\tilde{\tau}^{(L)}_1])N\}$ is $e^{ O(\epsilon)N}$, where there are at least $(1-\epsilon/E^{\overline{Q}^z}_0[\tilde{\tau}^{(L)}_1])N$ such jumps advancing in the $\ell$-direction. {\color{red}Taking} $\epsilon\to0$, {\color{red}we have}
    \begin{equation}\label{eqn: express Iq}
        -I_q(\ell) \leq -\text{log}( e^{\langle \theta_{z}^{\mathbb{P}},\ell\rangle} \mathcal{D}_{z,\mathbb{P}} ) + \frac{1}{E^{\overline{Q}^z}_0[\tilde{\tau}^{(L)}_1]}\mathbb{E}\log E^{\overline{Q}^z}_0\bigg[\prod_{j=1}^{\tilde{\tau}^{(L)}_1} \CYRu_j I_{ \{ Z_j=\ell j \} } \bigg].
    \end{equation}
    \begin{proof}[\textbf{Proof of Theorem \ref{thm: non-equality of rate functions}}.]
    Invoking a separation lemma for the stopping sequence $(\tilde{\tau}^{(L)}_k)_{k\geq1}$ via Lemma \ref{lem: exponential mixing inequality}, we have
        \begin{equation}\label{eqn: express Ia, sharp}
            -I_a(\ell) \geq -\text{log}( e^{\langle \theta_{z}^{\mathbb{P}},\ell\rangle} \mathcal{D}_{z,\mathbb{P}} ) +   \frac{1}{E^{\overline{Q}^z}_0[\tilde{\tau}^{(L)}_1]}\log E^{\overline{Q}^z}_0\bigg[\mathbb{E}\prod_{j=1}^{\tilde{\tau}^{(L)}_1} \CYRu_jI_{ \{ Z_j=\ell j\} } \bigg].
        \end{equation}
        Writing $W^z_{\mathbb{P}}\coloneqq\text{log}( e^{\langle \theta_{z}^{\mathbb{P}},\ell\rangle} \mathcal{D}_{z,\mathbb{P}} )$ and combining (\ref{eqn: express Iq}) with (\ref{eqn: express Ia, sharp}), we have       
    \[
        \frac{W^z_{\mathbb{P}}-I_q(\ell)}{W^z_{\mathbb{P}}-I_a(\ell)} \leq \frac{\mathbb{E}\log E^{\overline{Q}^z}_0 [\prod_{j=1}^{\tilde{\tau}^{(L)}_1}  \CYRu_jI_{ \{ Z_j=\ell j\} } ]}{\log E^{\overline{Q}^z}_0[\mathbb{E}\prod_{j=1}^{\tilde{\tau}^{(L)}_1}  \CYRu_jI_{ \{ Z_j=\ell j\} } ] }.
    \]
    {\color{red}Recall that the law of the environment $\mathbb{P}$ is non-degenerate, i.e., its support is not a singleton at any site $x \in \mathbb{Z}^d$ and direction $\ell \in \mathbb{V}$}, i.e.~we have $\mathbb{P}(\omega(x,\ell)=\mathbb{E}[\omega(x,\ell)])<1$. Henceforth, by Jensen's inequality, 
    \begin{equation}\begin{aligned}\label{eqn: what we guess}
        \frac{1}{E^{\overline{Q}^z}_0[\tilde{\tau}^{(L)}_1]} \mathbb{E} \log E^{\overline{Q}^z}_0\bigg[\prod_{j=1}^{\tilde{\tau}^{(L)}_1} \CYRu_jI_{ \{ Z_j=\ell j\} } \bigg] < \frac{1}{E^{\overline{Q}^z}_0[\tilde{\tau}^{(L)}_1]}\log E^{\overline{Q}^z}_0\bigg[\mathbb{E}\prod_{j=1}^{\tilde{\tau}^{(L)}_1} \CYRu_jI_{ \{ Z_j=\ell j\} } \bigg],
    \end{aligned}\end{equation}
    yielding that $I_q(\ell)>I_a(\ell)$ with $\ell\in\mathbb{V}$. {\color{red}Finally,} since both {\color{red}rate functions} $I_a$ and $I_q$ are continuous on the velocity surface $\mathbb{D}$, we can find some $x\in\text{int}(\mathbb{D})$ close to $\ell\in\mathbb{V}$ such that $I_q(x)>I_a(x)$. {\color{red}This concludes the proof.}
    \end{proof}

\section{Technical lemmas: I}
    \begin{lemma}\label{lem: A0}
        \normalfont
        When $\mathbb{P}$ {\color{red}satisfies} either $\textbf{(SM)}_{C,g}$ or $\textbf{(SMG)}_{C,g}$, there exists a unique probability measure $\mathbb{P}_w$ {\color{red}for every transient path} $w\in\mathbb{W}^{\text{tr}}$ that is consistent with the specification $(\mathbb{P}_{V_m,w}^{\omega_{V_m^c}})_{V_m\subseteq\mathbb{Z}^d}$ defined in (\ref{eqn: derivative of Pw}).
    \end{lemma}
    \begin{proof}
        {\color{red}Fix} $m\geq1$ and {\color{red}finite volumes} $\Lambda\subseteq V_m\subseteq\mathbb{Z}^d$. {\color{red}Select} $\mathbbm{x}\in\partial^rV_m$ and $\omega,\tilde{\omega}\in\Omega$ with $\omega_x=\tilde{\omega}_x$ whenever $x\neq \mathbbm{x}$. {\color{red}Let} $\mathbb{P}^{\omega_{V_m^c}}_{w,\Lambda,V_m}$ {\color{red}denote} the projection of $\mathbb{P}^{\omega_{V_m^c}}_{w,V_m}$ onto $\mathscr{F}_{\Lambda}$, {\color{red}and define} $\mathbb{P}^{\tilde{\omega}_{V_m^c}}_{w,\Lambda,V_m}$ {\color{red}analogously by substituting} $\tilde{\omega}_{V_m^c}$ {\color{red}for} $\omega_{V_m^c}$. {\color{red}Observe that} $\mathbb{P}^{\omega_{V_m^c}}_{w,V_m}$ is just $\mathbb{P}( (\omega_x)_{x\in V_m}\in\vdot| \mathscr{F}_{V_m^c} )$ perturbed by a zero-range perturbation {\color{red}induced} by the transient path $w\in\mathbb{W}^{\text{tr}}$, the {\color{red}local mixing} inequalities (\ref{eqn: sm}) and (\ref{eqn: smg}) are preserved under {\color{red}modified decay parameters $(C_w, g_w)$ depending explicitly on the path $w$.} {\color{red}Specifically}, we have $d\mathbb{P}^{\omega_{V_m^c}}_{w,\Lambda,V_m}( (\omega_x)_{x\in \Lambda}\in\vdot| \mathscr{F}_{V_m^c} )/d\mathbb{P}^{\tilde{\omega}_{V_m^c}}_{w,\Lambda,V_m}( (\omega_x)_{x\in \Lambda}\in\vdot| \tilde{\mathscr{F}}_{V_m^c} )\leq\mathscr{A}_{\mathbbm{x}}$ where
        \[
            \mathscr{A}_{\mathbbm{x}}\coloneqq\exp\bigg(C_w\sum_{x\in\partial^r\Lambda,y\sim\mathbbm{x}}e^{-g_w\abs{x-y}_1}\bigg)\quad\text{under}\;\;\textbf{(SM)}_{C,g},\qquad \mathscr{A}_{\mathbbm{x}}\coloneqq\exp\bigg(C_w\sum_{x\in \Lambda} e^{-g_w\abs{x- \mathbbm{x} }_1}\bigg)\quad\text{under}\;\;\textbf{(SMG)}_{C,g}.
        \]
        Here $\tilde{\mathscr{F}}_{V_m^c}$ is the $\sigma$-algebra generated by $\tilde{\omega}_{V_m^c}$. Hence $\normx{\mathbb{P}^{\omega_{V_m^c}}_{w,\Lambda,V_m}-\mathbb{P}^{\tilde{\omega}_{V_m^c}}_{w,\Lambda,V_m}}_{FV}$ satisfies the same inequality \cite[Eqn. (4.1)]{Rassoul-Agha} and therefore by \cite{Dobrushin/Shlosman2} and \cite[Section 5.2]{Georgii}, the specification $(\mathbb{P}_{V_m,w}^{\omega_{V_m^c}})_{V_m\subseteq\mathbb{Z}^d}$ {\color{red}admits a unique thermodynamic limit measure} $\mathbb{P}_w$ on $(\Omega, \mathscr{F})$ {\color{red}consistent with the given boundary specifications. This completes the proof.}
    \end{proof}

    
    \begin{lemma}(Cf. Lemma \ref{lem: continuity of q(w,z)})\label{lem: A00}
        \normalfont
        When $\mathbb{P}$ {\color{red}satisfies} either $\textbf{(SM)}_{C,g}$ or $\textbf{(SMG)}_{C,g}$, the transition kernels $q:\mathbb{W}^{\text{tr}}\times\mathbb{V}\to[\kappa,1]$ are continuous with respect to the subspace topology of $\mathcal{U}$ on $\mathbb{W}$.
    \end{lemma}
    \begin{proof}
        For completeness and to avoid {\color{red}unnecessary} repetition, we {\color{red}refer the} readers to \cite{Rassoul-Agha} for {\color{red}standard verification} steps, {\color{red}focusing here on the proof} of \cite[Eqn. (4.2)]{Rassoul-Agha}. {\color{red}Because} the Radon--Nikodým derivative (\ref{eqn: derivative of Pw}) of $\mathbb{P}_w$ {\color{red}modifies} only the one-dimensional marginals of $\Omega$ {\color{red}without altering} the correlations among different sites, we first {\color{red}demonstrate} that \cite[Hypothesis B]{Rassoul-Agha} is {\color{red}preserved under} our mixing scenarios. {\color{red}Specifically,} after adding weights to one-dimensional marginals of $\Omega$, we {\color{red}represent} the {\color{red}perturbed} environment {\color{red}by} $(\omega^w_x)_{x\in\Delta}$, where $\omega^w_x=\alpha_x\omega_x+\beta_x\xi_x$ {\color{red}for each vertex $x$}. Here {\color{red}$(\xi_x)_{x \in \mathbb{Z}^d}$ is a family of mutually} independent {\color{red}random variables}, and $0<\alpha_x\leq1$, $\beta_x=1-\alpha_x$ {\color{red}govern the intensity} of the zero-range perturbation. Under (\ref{eqn: sm}) of $\textbf{(SM)}_{C,g}$, [{\color{red}The argument for} $\textbf{(SMG)}_{C,g}$ {\color{red}proceeds mutatis mutandis}], the ratio
        \[
            \frac{\mathbb{P}((\omega^w_x)_{x\in\Delta}=\vec{k}|\eta)}{\mathbb{P}((\omega^w_x)_{x\in\Delta}=\vec{k}|\eta^\prime)}=\frac{\sum_{\vec{h}}\mathbb{P}((\alpha_x\omega_x)_{x\in\Delta}=\vec{h}|\eta)\mathbb{P}((\beta_x\xi_x)_{x\in\Delta}=\vec{k}-\vec{h}|\eta^\prime)}{\sum_{\vec{h}}\mathbb{P}((\alpha_x\omega_x)_{x\in\Delta}=\vec{h}|\eta^\prime)\mathbb{P}((\beta_x\xi_x)_{x\in\Delta}=\vec{k}-\vec{h}|\eta^\prime)},\quad\forall~\vec{k}\in[0,1]^{\abs{\Delta}},
        \]
        by the mutual independence between the perturbation $\xi$ and $\omega$. {\color{red}Because the perturbations} $\xi$'s are {\color{red}coordinate-wise} independent and since $\eta$,  $\eta^\prime$ are different only on $A\subseteq V^c$, we have $\mathbb{P}((\beta_x\xi_x)_{x\in\Delta}=\vec{k}-\vec{h}|\eta)=\mathbb{P}((\beta_x\xi_x)_{x\in\Delta}=\vec{k}-\vec{h}|\eta^\prime)$ for any $\vec{h}$ and $\vec{k}$. Invoking (\ref{eqn: sm}) again, the RHS {\color{red}provides a uniform} bound {\color{red}for} the density $\frac{d\mathbb{P}((\omega^w_x)_{x\in\Delta}\in\vdot|\eta)}{d\mathbb{P}((\omega^w_x)_{x\in\Delta}\in\vdot|\eta^\prime)}$  {\color{red}of} the {\color{red}perturbed} $\omega^w$ under zero-range perturbation. {\color{red}Hence} \cite[Hypothesis B]{Rassoul-Agha} {\color{red}holds} in our {\color{red}setting}. Under the mixing assumptions $\textbf{(SM)}_{C,g}$ or $\textbf{(SMG)}_{C,g}$, {\color{red}there exist} $C_w,g_w$ so that $\mathbb{P}_w$ {\color{red}satisfies}  $\textbf{(SM)}_{C_w,g_w}$ or $\textbf{(SMG)}_{C_w,g_w}$, {\color{red}for all} $w\in\mathbb{W}^{\text{tr}}$. {\color{red}Moreover}, there exists an open neighborhood $\mathcal{O}_1(w)$ of $w$ in $\mathcal{U}$ such that when $w^\prime\in\mathcal{O}_1(w)$, we could take $C_{w^\prime}=2C_w$ and $g_{w^\prime}=g_w/2$. {\color{red}Accordingly,}
        \[
            \normy{\mathbb{P}_{w^\prime,\Delta}- \mathbb{P}_{w^\prime}\big((\omega_x)_{x\in\Delta}\in\vdot\big|\mathscr{F}_{V^c}\big)}_{FV}+1 \leq \sup_{\Lambda\in\mathscr{F}_{\Delta}}\mathbb{P}_{w^\prime,\Delta}(\Lambda)  \frac{\mathbb{P}_{w^\prime}((\omega_x)_{x\in\Delta}\in\Lambda|\mathscr{F}_{V^c})}{\mathbb{P}_{w^\prime,\Delta}(\Lambda)} \leq \exp\bigg(2C_w\sum_{x\in\partial^r\Delta,y\in\partial^r V^c}e^{-\frac{1}{2}g_w\abs{x-y}_1}\bigg),
        \]
        when $\mathbb{P}$ {\color{red}satisfies} $\textbf{(SM)}_{C,g}$.
        {\color{red}Similarly, under} $\textbf{(SMG)}_{C,g}$, {\color{red}the corresponding bound reads}
        \[
             \normy{\mathbb{P}_{w^\prime,\Delta}- \mathbb{P}_{w^\prime}\big((\omega_x)_{x\in\Delta}\in\vdot\big|\mathscr{F}_{V^c}\big)}_{FV} \leq \exp\bigg(2C_w\sum_{x\in\Delta,y\in V^c}e^{-\frac{1}{2}g_w\abs{x-y}_1}\bigg) -1.
        \]
        Under either mixing assumptions, $\normx{\mathbb{P}_{w^\prime,\Delta}- \mathbb{P}_{w^\prime}\big((\omega_x)_{x\in\Delta}\in\vdot\big|\mathscr{F}_{V^c}\big)}_{FV}\to0$ uniformly {\color{red}for} $w^\prime\in\mathcal{O}_1(w)$ as $d_1(\Delta,V^c)\to\infty$.
        Similarly, we have $\normx{\mathbb{P}_{w,\Delta}- \mathbb{P}_{w}\big((\omega_x)_{x\in\Delta}\in\vdot\big|\mathscr{F}_{V^c}\big)}_{FV}\to0$ as $d_1(\Delta,V^c)\to\infty$, verifying \cite[Eqn. (4.2)]{Rassoul-Agha} under our mixing conditions. The rest {\color{red}continuity assertion} follows {\color{red}directly} from the argument in \cite[Lemma 4.2]{Rassoul-Agha}.
    \end{proof}

    \begin{lemma}\label{lem: A1}
        \normalfont
       When $\mathbb{P}$ {\color{red}satisfies} $\textbf{(SMF)}_{C,g}$, {\color{red}for every} $\theta\in\mathbb{R}^d$ and $L\geq L_0$, we have $\mathds{E}^{\overline{Q}^z}_0 [  e^{\langle  \theta, Z(\tau^{(L)}_1)  \rangle -  \overline{\Lambda}^a_z(\theta) \tau^{(L)}_1 } \mathbb{E} \prod_{j=1}^{\tau^{(L)}_1}\CYRu_j ]\leq 1$.
    \end{lemma}
    \begin{proof}
        {\color{red}Fix} $\delta>0$ {\color{red}and define} $\eta^z_{\theta,\delta}\coloneqq\overline{\Lambda}^a_z(\theta)+\delta$. {\color{red}For} each $n\geq1$, {\color{red}consider the sum} $\CYRbb_{n}^\delta(\theta) \coloneqq E^{\overline{Q}^z}_0[  e^{\langle \theta,Z(\tau^{(L)}_n)  \rangle -  \eta^z_{\theta,\delta}\tau^{(L)}_n  }\mathbb{E}\prod_{j=1}^{\tau^{(L)}_n}\CYRu_j]$. {\color{red}By} splitting the expectation in $\CYRbb_{n}^\delta(\theta)$ {\color{red}according to the} different {\color{red}possible} realizations of $\tau^{(L)}_n$, we {\color{red}obtain}
        \begin{equation*}\begin{aligned}
            \CYRbb_{n}^\delta(\theta) \leq  \sum_{k=n}^\infty e^{- \eta^z_{\theta,\delta}k} E^{Q^z}_0\bigg[ e^{\langle  \theta,Z_k\rangle} \mathbb{E}\prod_{j=1}^{k}\xi(Z_{j-1},\Delta_j(Z)) \bigg] \leq \sum_{k=n}^\infty e^{-\delta k/2} \leq \frac{ e^{-\delta n/2}}{1- e^{-\delta /2}},\qquad\forall~n\geq N(\theta,\delta).
        \end{aligned}\end{equation*}        
        where {\color{red}the asymptotic rate} $E^{Q^z}_0[ e^{\langle  \theta,Z_k\rangle} \mathbb{E}\prod_{j=1}^{k}\xi(Z_{j-1},\Delta_j(Z)) ] = e^{(\overline{\Lambda}^a_z(\theta,\lambda)+o(1))k}$ as $k\to\infty$ guarantees the existence of {\color{red}a} sufficiently large {\color{red}threshold} $N(\theta,\delta)\in\mathbb{N}$. The renewal structure {\color{red}in} Proposition \ref{prop: mixing separation information} of the mixing environment yields
        \[
            \CYRbb_{n}^\delta(\theta) \geq  \CYRbb_{1}^\delta(\theta) \bigg( \mathds{E}^{\overline{Q}^z}_0\bigg[  e^{\langle \theta,Z(\tau^{(L)}_1) \rangle -  \eta^z_{\theta,\delta}\tau^{(L)}_1 } \mathbb{E}\prod_{j=1}^{\tau^{(L)}_1}\CYRu_j \bigg]\bigg)^{n-1},\qquad\forall~n\geq1.
        \]
        In light of previous arguments, we {\color{red}deduce} that $\mathds{E}^{\overline{Q}^z}_0 [  e^{\langle  \theta,Z(\tau^{(L)}_1)  \rangle -  \eta^z_{\theta,\delta} \tau^{(L)}_1 }\mathbb{E}\prod_{j=1}^{\tau^{(L)}_1}\CYRu_j]\leq e^{-\delta/2}$. {\color{red}Taking} $\delta\to0$ and invoking the monotone convergence theorem, {\color{red}we conclude the desired} assertion.
    \end{proof}

    \begin{lemma}\label{lem: A2}
        \normalfont
        When $\mathbb{P}$ {\color{red}satisfies} $\textbf{(SMF)}_{C,g}$, $\exists~\gamma_1=\gamma_1(\kappa,z)>0$ such that $ \mathds{E}^{\overline{Q}^z}_0[  e^{\langle  \theta, Z(\tau^{(L)}_1)  \rangle -  \overline{\Lambda}^a_z(\theta) \tau^{(L)}_1 } \mathbb{E} \prod_{j=1}^{\tau^{(L)}_1}\CYRu_j ]\geq 1$, $\forall~L\geq L_0$ and whenever $|\theta|\vee\text{dis}(\mathbb{P})<\gamma_1\CYRk^L$.
    \end{lemma}
    \begin{proof}
        {\color{red}For} any $n\geq1$ and $\theta\in\mathbb{R}^d$, {\color{red}define the weighted trajectory variable} $\CYRe^r_n(\theta)\coloneqq e^{\langle  \theta,Z_n \rangle - r  n  } \mathbb{E}\prod_{j=1}^{n}\CYRu_j$. {\color{red}We decompose the expectation} $E^{\overline{Q}^z}_0[\CYRe^r_n(\theta)]$ {\color{red}over} the disjoint events $\{\tau^{(L)}_m<n\leq\tau^{(L)}_{m+1},\,m=\tau^{(L)}_m+i\}$ for $0\leq m<n$ and $0<i\leq n$. {\color{red}Then},
        \begin{equation*}\begin{aligned}
            &\qquad  E^{\overline{Q}^z}_0 [\CYRe^r_n(\theta)]\leq \sum_{m=0}^{n-1}\sum_{i=1}^n E^{\overline{Q}^z}_0[\CYRe_{\tau^{(L)}_m}^r(\theta),\,\tau^{(L)}_m=n-i]\mathds{E}^{\overline{Q}^z}_0[\CYRe^r_i(\theta),\,\tau^{(L)}_1>i]\\
            &\leq  \sum_{m=0}^{n-1} E^{\overline{Q}^z}_0[\CYRe_{\tau^{(L)}_m}^r(\theta)] \mathds{E}^{\overline{Q}^z}_0[\sup_{i\leq\tau^{(L)}_1}\CYRe_{i}^r(\theta)] \leq \mathds{E}^{\overline{Q}^z}_0[\sup_{i\leq\tau^{(L)}_1}\CYRe_{i}^r(\theta)] \bigg(1+ E^{\overline{Q}^z}_0[\sup_{i\leq\tau^{(L)}_1}\CYRe^r_i(\theta)]\sum_{m=1}^\infty(\mathds{E}^{\overline{Q}^z}_0[\CYRe^r_{\tau^{(L)}_1}(\theta)])^{m-1}\bigg)
        \end{aligned}\end{equation*}
        where, in order to obtain the last inequality, we use that  $E^{\overline{Q}^z}_0[\CYRe^r_{\tau^{(L)}_m}(\theta)] =  E^{\overline{Q}^z}_0[\CYRe^r_{\tau^{(L)}_1}(\theta)]( \mathds{E}^{\overline{Q}^z}_0[\CYRe^r_{\tau^{(L)}_1}(\theta)] )^{m-1}$, $\forall~m\geq1$, which follows from the renewal structure. {\color{red}Next, we establish the integrability of the supremum.} Notice also that
        \begin{equation}
            \label{eqn: temp}E^{\overline{Q}^z}_0[\sup_{i\leq\tau^{(L)}_1}\CYRe^r_i(\theta)] = E_UQ^z_{\epsilon,0}(\beta^{(\zeta)}_0=\infty) \mathds{E}^{\overline{Q}^z}_0[\sup_{i\leq\tau^{(L)}_1}\CYRe^r_i(\theta)] + E_UQ^z_{\epsilon,0}(\beta^{(\zeta)}_0<\infty) \mathds{E}^{\overline{Q}^z}_0[\sup_{i\leq\tau^{(L)}_1}\CYRe^r_i(\theta)|\beta^{(\zeta)}_0<\infty].
        \end{equation}
        Now we take $r=\overline{\Lambda}^a_z(\theta)-\delta$ for some $\delta>0$ and we have $\log\CYRe^r_i(\theta)\leq (2 |\theta|  + \CYRb(\text{dis}(\mathbb{P})) + \delta)i$, $\forall~i\geq1$. If we choose $\gamma_1$ and $\delta(\kappa,z)$ small enough so that $2|\theta|+\CYRb(\text{dis}(\mathbb{P}))<\gamma_0\CYRk^L$ when $|\theta|\vee\text{dis}(\mathbb{P})<\gamma_1\CYRk^L$, then by Proposition \ref{prop: mixing exponential of tau < infinity} to estimate  $E^{\overline{Q}^z}_0[e^{\gamma_0\CYRk^L\tau^{(L)}_1  } ]$, we have $E_UE^{\overline{Q}^z}_{\epsilon,0}[\sup_{i\leq\tau^{(L)}_1}\CYRe^r_i(\theta)]\leq E_UE^{\overline{Q}^z}_{\epsilon,0}[e^{\gamma_0\CYRk^L\tau^{(L)}_1/2}]<\infty$, and  $\mathds{E}^{\overline{Q}^z}_{0}[\sup_{i\leq\tau^{(L)}_1}\CYRe^r_i(\theta)]<\infty$ as well. {\color{red}Now, assume for the sake of contradiction} that $\mathds{E}^{\overline{Q}^z}_0[\CYRe^r_{\tau^{(L)}_1}(\theta)]<1$.{\color{red}Then the geometric series converges, implying that} $\sup_{n\geq1}E^{\overline{Q}^\lambda}_0[\CYRe^r_{n}(\theta)]<\infty$, yielding $ \varlimsup_{n}\frac{1}{n}\log E^{\overline{Q}^z}_n[\CYRe^r_n(\theta)] \leq 0$. However, by the choice of {\color{red}$r = \overline{\Lambda}^a_z(\theta) - \delta$}, we should have
        \[
            \lim_{n\to\infty}\frac{1}{n}\log E^{\overline{Q}^z}_n[\CYRe^r_n(\theta)] = \delta + \lim_{n\to\infty}\frac{1}{n}\log E^{\overline{Q}^z}_n[\CYRe^{\overline{\Lambda}^a_z(\theta)}_n(\theta)] = \delta>0,
        \]
        leading to {\color{red}a} contradiction. Hence we must have $\mathds{E}^{\overline{Q}^z}_0[\CYRe^r_{\tau^{(L)}_1}(\theta)] = \mathds{E}^{\overline{Q}^z}_0[  e^{\langle \theta,Z(\tau^{(L)}_1) \rangle - (\overline{\Lambda}^a_z(\theta)+\delta) \tau^{(L)}_1 } \mathbb{E}\prod_{j=1}^{\tau^{(L)}_1}\CYRu_j ]\geq1$ whenever $|\theta|\vee\text{dis}(\mathbb{P})<\gamma_1\CYRk^L$. {\color{red}Taking} $\delta\to0$ {\color{red}and} invoking the dominated convergence theorem, the assertion is {\color{red}then} verified.
    \end{proof}

    \begin{lemma}\label{lem: A3}
        \normalfont
        When $\mathbb{P}$ {\color{red}satisfies} $\textbf{(SMF)}_{C,g}$,  $\lim_{n\to\infty} 
            E^{\overline{Q}^z}_0 [  e^{\langle  \theta,  Z(L^{(z)}_n)  \rangle -  (\overline{\Lambda}^q_z(\theta) + \delta )  L^{(z)}_n   }  \prod_{j=1}^{L^{(z)}_n}\CYRu_j ] = 0$, $\mathbb{P}$-a.s., $\forall~\theta\in\mathbb{R}^d$ and $\delta>0$. 
    \end{lemma}
    \begin{proof}
        {\color{red}Set} $\rho^z_{\theta,\delta}\coloneqq\overline{\Lambda}^q_z(\theta)+\delta$. {\color{red}Splitting over the} different {\color{red}possible} values of $L^{(z)}_n$, we {\color{red}obtain}
        \begin{equation*}\begin{aligned}
            E^{\overline{Q}^z}_0\bigg[ e^{\langle \theta, Z(L^{(z)}_n)\rangle- (\overline{\Lambda}^a_z(\theta)+\delta) L^{(z)}_n}\prod_{j=1}^{L^{(z)}_n}\CYRu_j \bigg] \leq  \sum_{k=n}^\infty e^{- k \rho^z_{\theta,\delta}}E^{Q^z}_0\bigg[e^{\langle\theta,Z_k\rangle}\prod_{j=1}^k\xi(Z_{j-1} , \Delta_j(Z) )\bigg]\leq \sum_{k=n}^\infty e^{- \delta k/2} \leq \frac{e^{- \delta n/2}}{1-e^{- \delta /2}},
        \end{aligned}\end{equation*}
        for any $n\geq N_\omega(\theta,\delta)$, {\color{red}where} $N_\omega(\theta,\delta)\in\mathbb{N}$ {\color{red}is a finite random index. Here} we {\color{red}have used } $E^{Q^z}_0[ e^{\langle\theta,Z_k\rangle} \prod_{j=1}^k\xi(Z_{j-1},\Delta_j(Z)) ] = e^{(\overline{\Lambda}^q_z(\theta)+o(1))k}$, $\mathbb{P}$-a.s. {\color{red}Taking} $n\to\infty$ and {\color{red}invoking} the monotone convergence theorem, the assertion {\color{red}follows}.
    \end{proof}

    \begin{lemma}\label{lem: what to show2}
        \normalfont
        When $\mathbb{P}$ {\color{red}satisfies} $\textbf{(SMF)}_{C,g}$, {\color{red}there exists} $\gamma_2=\gamma_2(\kappa,z)>0$ such that (\ref{eqn: what to show2}) {\color{red}remains} finite when $\text{dis}(\mathbb{P})<\gamma_2\CYRk^L$. 
    \end{lemma}
    \begin{proof}
        {\color{red}To establish the} assertion, {\color{red}define the collision stopping time} $\zeta^{(z)}\coloneqq\inf\{[k\CYRk^L]\geq0:\,\exists~i,j\geq1\;\text{s.t.}\; d_1(\CYRk^L Z_i-z, \CYRk^L(\widetilde{Z}_j-z))\leq L\CYRk^L\;\text{and}\;\langle Z_i,\ell\rangle=k\}$. {\color{red}Whenever} $|\theta|\vee\text{dis}(\mathbb{P})<\gamma_1\CYRk^L\wedge1/2$, {\color{red}Lemma~\ref{lem: A1} and the decoupling of independent traces yield}
        \[
            \mathds{E}^{\overline{Q}^z}_{0,\nu} [F_n(\theta),\,[n\CYRk^L]\in\mathcal{L}^{(z)},\,[n\CYRk^L] \leq \zeta^{(z)}] \leq \mathds{E}^{\overline{Q}^z}_{0,\nu}[\phi_n(\theta)\widetilde{\phi}_n(\theta)e^{ L \CYRb(1/2)},\,[n\CYRk^L]\in\mathcal{L}^{(z)}] \leq \mathbb{E}E_U [\CYRPhi^z_n(\theta)]^2 e^{ L\CYRb(1/2)}\leq \exp( L \CYRb(1/2)).
        \]        
        {\color{red}Next, let} $\sigma^{(z)}$ {\color{red}be the renewal time defined} as {\color{red}in} \cite[Eqn. (5.20)]{Bazaes/Mukherjee/Ramirez/Sagliett0}. {\color{red}Decomposing $A_{\nu,n}(\theta)$ according to $\sigma^{(z)} = k$, we get the renewal inequality}  $A_{\nu,n}(\theta)\leq \exp( L \CYRb(1/2)) + \sum_{1\leq k\leq [n\CYRk^L]}  B_{\nu,k}(\theta)\sup_{\nu^\prime\in\mathbb{V}_d} A_{\nu^\prime,n-k}(\theta)$, where {\color{red}the transition kernel is defined by} $B_{\nu,k}(\theta)\coloneqq\mathds{E}^{\overline{Q}^z}_{0,\nu}[F_n(\theta),\,\sigma^{(z)}=  n]$. By Lemma \ref{lem: what to show3}, {\color{red}there exists} $\gamma_3=\gamma_3(\kappa,z)>0$ {\color{red}such that, provided} $|\theta|\vee\text{dis}(\mathbb{P})<\gamma_3\CYRk^L$, we have $\mathscr{B}\coloneqq\sup_{\abs{\theta}<\gamma_3\CYRk^L, \nu\in\mathbb{V}_d}\sum_{n=1}^\infty B_{\nu,n}(\theta)<1$. {\color{red}We now fix} $|\theta|\vee\text{dis}(\mathbb{P})<\gamma_1\CYRk^L\wedge\gamma_3\CYRk^L\wedge1/2\eqqcolon\gamma_2\CYRk^L$, {\color{red}then}
        \[ 
            A_{\nu,n}(\theta)\leq  \exp( L \CYRb(1/2)) + \bigg( \sup_{j\leq N,\nu^\prime\in\mathbb{V}_d} A_{\nu^\prime,j}(\theta) \bigg)\sum_{1\leq k\leq [N\CYRk^L ]}  B_{\nu,k}(\theta),\qquad\forall~n\leq N,\;
            \;N\geq1.
        \]
        {\color{red}Taking the} supremum {\color{red}over} $n\leq N$ and then {\color{red}letting} $N\to\infty$, we have $\mathscr{A}\leq(1-\mathscr{B})\text{exp} L \CYRb(1/2))<\infty$.
    \end{proof}

    \begin{lemma}\label{lem: what to show3}
        \normalfont
        When $\mathbb{P}$ {\color{red}satisfies} $\textbf{(SMF)}_{C,g}$, {\color{red}there exists} $\gamma_3=\gamma_3(\kappa,z)>0$ {\color{red}such that} $\mathscr{B}<1$ whenever $\text{dis}(\mathbb{P})<\gamma_3\CYRk^L$.
    \end{lemma}
    \begin{proof}
        {\color{red}Let} $B^*_{\nu,n}(\theta)$ {\color{red}be} the zero-disorder {\color{red}version} of $B_{\nu,n}(\theta)$, {\color{red}constructed} analogously to \cite[Section 5.3]{Bazaes/Mukherjee/Ramirez/Sagliett0}. {\color{red}By} \cite[Remark 6]{Bazaes/Mukherjee/Ramirez/Sagliett0} and mean value theorem, $|B_{\nu,n}(\theta)-B^*_{\nu,n}(0)|\leq (|\theta|+\CYRb(\text{dis}(\mathbb{P})))\mathds{E}^{\overline{Q}^z}_{0,\nu}[(\tau^{(L)}_n+\widetilde{\tau}^{(L)}_n)e^{ (|\theta|+\CYRb(\text{dis}(\mathbb{P})))(\tau^{(L)}_n+\widetilde{\tau}^{(L)}_n)}]$, {\color{red}and} is no less than
        \[              
             2(|\theta|+ \CYRb(\text{dis}(\mathbb{P}))) \mathds{E}^{\overline{Q}^z}_0 \big[\tau^{(L)}_n e^{ (|\theta|+\CYRb(\text{dis}(\mathbb{P}))) \tau^{(L)}_n}\big] \mathds{E}^{\overline{Q}^z}_0\big[e^{ (|\theta|+\CYRb(\text{dis}(\mathbb{P})))\tau^{(L)}_n}\big]\leq  2(|\theta|+ \CYRb(\text{dis}(\mathbb{P}))) \mathds{E}^{\overline{Q}^z}_0\big[ e^{ 2(|\theta|+\CYRb(\text{dis}(\mathbb{P})))\tau^{(L)}_1}\big]^{2n},
        \]
        {\color{red}uniformly} for {\color{red}all} $\nu\in\mathbb{V}_d$. {\color{red}Consequently, for any fixed} $\eta>0$ and $n\geq1$, we can find $\gamma_4(n)=\gamma_4(\kappa,z,n)>0$ such that
        \begin{equation}\label{eqn: a1}
            \sup_{|\theta|<\gamma_4(n)\CYRk^L,~\nu\in\mathbb{V}_d}|B_{\nu,n}(\theta)-B^*_{\nu,n}(0)| \leq \eta,\qquad\text{whenever}\quad L\geq L_0,\quad \text{dis}(\mathbb{P})<\gamma_4(n)\CYRk^L.
        \end{equation}
        {\color{red}Next}, by Lemma \ref{lem: what to show3-1}, {\color{red}there exist} $\gamma_5(\kappa,z)>0$ and $K_0(\kappa,z)>0$ {\color{red}such that} $\sum_{n=1}^\infty \sup_{\text{dis}(\mathbb{P})<\gamma_5\CYRk^L}\sup_{\substack{\nu\in\mathbb{V}_d\\ |\theta|<\gamma_5\CYRk^L}} B_{\nu,n}(\theta)\leq K_0$. Furthermore, by Lemma \ref{lem: estimate B*}, {\color{red}there exists a gap parameter} $\delta=\delta(\kappa,z)>0$ such that $\sup_{\nu\in\mathbb{V}_d}\sum_{n=1}^\infty B^*_{\nu,n}(0)\leq1-\delta$. {\color{red}Applying} Lemma \ref{lem: what to show3-1} {\color{red}once more, we can choose sufficiently large} $N=N(\kappa,z)\in\mathbb{N}$ {\color{red}to control} 
        \[
            \sum_{n>N}\sup_{\substack{\nu\in\mathbb{V}_d\\ |\theta|<\gamma_5\CYRk^L}}|B_{\nu,n}(\theta)-B^*_{\nu,n}(0)|\leq \sum_{n>N} \sup_{\substack{\nu\in\mathbb{V}_d\\ |\theta|<\gamma_5\CYRk^L}} B_{\nu,n}(\theta)\leq\delta(\kappa,z)/4.
        \]
        {\color{red}For the finite initial segment} $\{1,\ldots,N(\kappa,z)\}$, {\color{red}invoking~\eqref{eqn: a1} with $\eta = \frac{\delta}{4N}$ yields a uniform parameter $\gamma'_4(\kappa,z) > 0$ such that} $\sum_{n\leq N} \sup_{\nu\in\mathbb{V}_d, |\theta|<\gamma_4^\prime\CYRk^L} |B_{\nu,n}(\theta)-B^*_{\nu,n}(0)| \leq 1-\delta(\kappa,z)/4$, whenever $\text{dis}(\mathbb{P})<\gamma_4^\prime\CYRk^L$. {\color{red}Finally, setting} $\gamma_3\CYRk^L\coloneqq\gamma_4^\prime\CYRk^L\wedge\gamma_5\CYRk^L$ {\color{red}and $\text{dis}(\mathbb{P})<\gamma_3\CYRk^L$}, we observe that $\mathscr{B}$ is no larger than
        \[
            \sup_{\nu\in\mathbb{V}_d} \sum_{n=1}^\infty B^*_{\nu,n}(0) + \big(\sum_{n\leq N}+\sum_{n> N}\big)\sup_{\substack{\nu\in\mathbb{V}_d\\ |\theta|<\gamma_3\CYRk^L}} |B_{\nu,n}(\theta)-B^*_{\nu,n}(0)| \leq 1-\delta(\kappa,z)/2.
        \]
        {\color{red}This completes the proof of} the assertion.
    \end{proof}

    \begin{lemma}\label{lem: what to show3-1}
        \normalfont
        When $\mathbb{P}$ {\color{red}satisfies} $\textbf{(SMF)}_{C,g}$, {\color{red}there exist}  $\gamma_5=\gamma_5(\kappa,z)>0$ and $K_0=K_0(\kappa,z)>0$ such that 
        \[
            \sum_{n=1}^\infty \sup_{\text{dis}(\mathbb{P})<\gamma_5\CYRk^L} \sup_{\substack{\nu\in\mathbb{V}_d\\ |\theta|<\gamma_5\CYRk^L}} B_{\nu,n}(\theta)\leq K_0.
        \]
    \end{lemma}
    \begin{proof}
        {\color{red}Define the index} $\psi^{(z)}\coloneqq\sup\{k\in\mathcal{L}^{(z)}:\,k\leq\zeta^{(z)}\}$. {\color{red}Following the decomposition in} \cite[Eqn. (5.28)]{Bazaes/Mukherjee/Ramirez/Sagliett0}, we have
        \[
            B_{\nu,n}(\theta)\leq \sum_{j<n} \sup_{\nu^\prime\in\mathbb{V}_{d}}\mathds{E}^{\overline{Q}^z}_{0,\nu}[F_j(\theta),\,\widetilde{Z}_{\widetilde{L}^{(z)}_j}-Z_{L^{(z)}_j}=\nu^\prime,\,\psi^{(z)}=j]\sum_{\nu^\prime\in\mathbb{V}_d}D_{\nu^\prime,n-k}(\theta),
        \]
        where {\color{red}for every} $n\geq1$ and $\nu \in\mathbb{V}_d$, {\color{red}the exit term is defined by} $D_{\nu,n}(\theta)\coloneqq\mathds{E}^{\overline{Q}^z}_{0,\nu}[F_n(\theta),\,n=\inf\{j\in\mathcal{L}^{(z)}:\,j>0\}>\zeta^{(z)}]$. Notice that {\color{red}the event} $\psi^{(z)}=j$ implies $I_j\leq1$. Hence {\color{red}for each} $j\geq1$, we have
        \begin{equation}\begin{aligned}\label{eqn: dog1}
            &\mathds{E}^{\overline{Q}^z}_{0,\nu}[F_j(\theta),\,\widetilde{Z}_{\widetilde{L}^{(z)}_j}-Z_{L^{(z)}_j}=\nu^\prime,\,\psi^{(z)}=j] \leq  \mathds{E}^{\overline{Q}^z}_{0,\nu}[\phi_j(\theta)\widetilde{\phi}_j(\theta),\, \widetilde{Z}_{\widetilde{L}^{(z)}_j}-Z_{L^{(z)}_j}=\nu^\prime,\,j\in\mathcal{L}^{(z)}]\\
            &\qquad\leq  e^{L \CYRb(\text{dis}(\mathbb{P}))} \sum_{\langle x,\ell\rangle= [j/\CYRk^L]} \sum_{k\in\mathbb{N}} \hat{P}^{(\theta)}_{0}(\langle Z^{(\theta)}_k,\ell\rangle= j) \leq  \sup_{\langle x,\ell\rangle= [j/\CYRk^L] } \sum_{k\in\mathbb{N}}  \mu^{(\theta)}_{k}(x),
        \end{aligned}\end{equation}
        where $\mu^{(\theta)}_{k}(\vdot)$ denotes the $k$-fold convolution of the law $\mu^{(\theta)}(\vdot)$. The above derivation is also seen in \cite[p. 58]{Chen}. And
        \[
            \big| \sum_{x\in\mathbb{Z}^d} e^{c\abs{x}}\mu^{(\theta)}(x) -1 \big|\leq \mathds{E}^{\overline{Q}^z}_0\big[ ((2|\theta|+\CYRb(\text{dis}(\mathbb{P}))+c) \tau^{(L)}_1) e^ {(2|\theta|+\CYRb(\text{dis}(\mathbb{P}))+c) \tau^{(L)}_1}\big].
        \]
        Fixing some $c_1=c_1(\kappa,z)>0$, $|\theta|\vee\text{dis}(\mathbb{P})<\nu_1\CYRk^L$ for some $\nu_1=\nu_1(\kappa,z)>0$, we get $\sum_{x\in\mathbb{Z}^d} e^{c_1\abs{x}}\mu^{(\theta)}(x)\leq2$. Notice $ \Sigma_{\mu^{(\theta)}}=H^a_z(\theta) \mathds{E}^{\overline{Q}^z}_0[\tau^{(L)}_1\phi_1(\theta)]$. Then by the proof of Proposition \ref{prop: region A contains an open ball}, $\exists~\nu_2^\prime(\kappa,z)$, $c(\kappa,z)>0$ s.t. $\sup_{\text{dis}(\mathbb{P})<\nu_2^\prime\CYRk^L}\normx{H^a_z(0)^{-1}}\leq c$ and $\sup_{|\theta|\vee\text{dis}(\mathbb{P})<\nu_2^\prime\CYRk^L}\normx{H^a_z(\theta)-H^a_z(0)}\leq(2c)^{-1}$. Via analogous arguments of \cite[Eqn. (5.34)]{Bazaes/Mukherjee/Ramirez/Sagliett0},  $\sup_{|\theta|\vee\text{dis}(\mathbb{P})<\nu_2\CYRk^L} \normx{\widetilde{H}^a_z(\theta)^{-1}}\leq 2c$. Notice also $\mathds{E}^{\overline{Q}^z}_0[\tau^{(L)}_1\phi_1(\theta)]^{-1}\leq \mathds{E}^{\overline{Q}^z}_0[(\tau^{(L)}_1)^{-1}\phi_1(\theta)^{-1}]\leq c+1$ when $|\theta|\vee\text{dis}(\mathbb{P})<\nu_2\CYRk^L$ for some $0<\nu_2(\kappa,z)<\nu_2^\prime$, which yields the estimate $\inf_{|\theta|\vee\text{dis}(\mathbb{P})<\nu_2\CYRk^L}\sigma_{\text{min}}(\Sigma_{\mu^{(\theta)}})\geq c_2$ for some $c_2=c_2(\kappa,z)>0$, where $\sigma_{\text{min}}(\vdot)$ stands for the minimal singular value of a real matrix. Furthermore, $\langle  Z^{(\theta)}_1,\ell\rangle\geq \CYRk^L$, $\hat{P}^{(\theta)}_0$-a.s. Hence, $|\sum_{x\in\mathbb{Z}^d}\langle x,\ell\rangle\mu^{(\theta)}(x)|\geq \hat{E}^{(\theta)}_0[\langle Z^{(\theta)}_1,\ell\rangle]\geq \CYRk^L\eqqcolon c_3$. In light of \cite[p. 58]{Chen}, we have $\sup_{\langle x,\ell\rangle= [j/\CYRk^L]}\sum_{k\in\mathbb{N}}\mu^{(\theta)}(x)\leq K_1(1+[j/\CYRk^L])^{-(d-1)/2}$ for some $K_1=K_1(\kappa,z)>0$, $\forall~L\geq L_0$ and $j\geq1$. {\color{red}Consequently}, with $\nu(\kappa,z)\coloneqq\nu_1\wedge\nu_2\wedge1/2$,
        \[
            \sup_{\text{dis}(\mathbb{P})<\nu\CYRk^L}\sup_{\substack{\nu\in\mathbb{V}_d\\ |\theta/\CYRk^L|<\nu}} B_{\nu,n}(\theta) \leq e^{L\CYRb(1/2)} K_1\sum_{j<n}(1+[\CYRk^{-L}j])^{-(d-1)/2}\sum_{\nu^\prime\in\mathbb{V}_d}\sup_{\substack{\text{dis}(\mathbb{P})<\nu\CYRk^L\\ |\theta/\CYRk^L|<\nu}} D_{\nu^\prime,n-j}(\theta).
        \]
        {\color{red}Summing over $n \ge 1$ and changing variables $k = n - j$, the above estimate} implies
        \[
            \sum_{n=1}^\infty \sup_{\text{dis}(\mathbb{P})<\nu\CYRk^L}\sup_{\substack{\nu\in\mathbb{V}_d\\ |\theta/\CYRk^L|<\nu}} B_{\nu,n}(\theta) \leq e^{L\CYRb(1/2)} K_1 \sum_{j=1}^\infty (1+j)^{-(d-1)/2} \sum_{k=1}^\infty \sum_{\nu \in\mathbb{V}_d} \sup_{\substack{\text{dis}(\mathbb{P})<\nu\CYRk^L\\ |\theta/\CYRk^L|<\nu}} D_{\nu,k}(\theta).
        \]
        Invoking Lemma \ref{lem: estimate D term} {\color{red}to bound the remaining sum over} $D_{\nu,k}(\theta)$, {\color{red}we conclude that the total sum is bounded by a finite constant $K_0(\kappa,z) < \infty$.}
    \end{proof}

    \begin{lemma}\label{lem: estimate B*}
        \normalfont
        When $\mathbb{P}$ {\color{red}satisfies} $\textbf{(SMF)}_{C,g}$, {\color{red}there exists} $\delta=\delta(\kappa,z)>0$ such that $\sup_{\nu\in\mathbb{V}_d}\sum_{n=1}^\infty B^*_{\nu,n}(0)\leq1-\delta$.
    \end{lemma}
    \begin{proof}
        {\color{red}Notice first that} that $\sum_{n=1}^\infty B^*_{\nu,n}(0) =   \mathds{E}^{\overline{Q}^z}_{0,\nu}(\sigma^{(z)}<\infty)$. {\color{red}Recall}  $Q^z_0(\beta^{(\zeta)}_0=\infty)\geq\eta^{-1}$. And we have $E^{Q^z}_0[(\CYRk^L\tau^{(L)}_1)^\alpha]\leq\eta^\prime(\alpha)$ for some $\eta^\prime(\alpha)>0$, $\forall~\alpha\geq1$. Furthermore, by {\color{red}an} argument {\color{red}similar to the one following} following (\ref{eqn: dog1}), we know {\color{red}that} $\sup_{\nu\in\mathbb{Z}^d}\mathds{E}^{\overline{Q}^z}_0(\CYRk^L Z(\tau^{(L)}_n)\in\Box_{\ell,\nu})=\sup_{\nu\in\mathbb{Z}^d}\mu^*_{n}(\nu)\leq \eta^{\prime\prime}n^{-d/2}$, $\exists~\eta^{\prime\prime}=\eta^{\prime\prime}(\kappa,z)>0$, where $\mu_{\lambda,L}^*$ corresponds to the zero-disorder version of $\mu^{(0)}$. Invoking a slightly revised version of \cite[Propositions 3.1, 3.4]{Berger/Zeitouni}, where we have used the estimate $|Z(\tau^{(L)}_1)|\leq\tau^{(L)}_1$, we have $\sup_{\abs{z}\geq2N}\mathds{E}^{\overline{Q}^z}_{0,z}(\sigma^{(z)}<\infty)\leq 1/2$, for some $N=N(\kappa,z)\geq1$. The {\color{red}remainder of the assertion follows by arguments analogous to} \cite[Lemma 5.5]{Bazaes/Mukherjee/Ramirez/Sagliett0}, where we count the intersections between $\CYRk^L(Z_i-Z_0)-Z_0$ and $\CYRk^L(\widetilde{Z}_j-\widetilde{Z}_0)+\widetilde{Z}_0$ instead of that between $Z_j$ and $Z_j$. And the assertion is then verified.
    \end{proof}

    \begin{lemma}\label{lem: estimate D term}
        \normalfont
        When $\mathbb{P}$ {\color{red}satisfies} $\textbf{(SMF)}_{C,g}$, there exist {\color{red}constants} $\gamma_6=\gamma_6(\kappa,z)>0$ and $K^\prime=K^\prime(\kappa,z)>0$ such that
        \[
            \sum_{n=1}^\infty \sum_{\nu\in\mathbb{V}_d} \sup_{|\theta|\vee\text{dis}(\mathbb{P})<\gamma_6\CYRk^L} D_{\nu,n}(\theta)\leq K^\prime,\qquad\forall~L\geq L_0.
        \]
    \end{lemma}
    \begin{proof}
        {\color{red}Applying the} Cauchy--Schwarz inequality {\color{red}followed by Hölder's inequality, we obtain}
        \begin{equation}\label{eqn: dog2}
            D_{\nu,n}(\theta)\leq \mathds{E}^{\overline{Q}^z}_{0,\nu}[F_n(\theta)^2]^{1/2} \mathds{E}^{\overline{Q}^z}_{0,\nu}(n>\zeta^{(z)})^{1/4} \mathds{E}^{\overline{Q}^z}_{0,\nu}(n=\inf\{k\in\mathcal{L}^{(z)}:\,k>0\})^{1/4}.
        \end{equation}
        For $\forall~n\geq1$, the first term on the RHS of (\ref{eqn: dog2}) is bounded {\color{red}via an estimate on the regeneration times} by
        \begin{equation}\label{eqn: dog2 a}  \mathds{E}^{\overline{Q}^z}_0 [e^{4(|\theta|+\CYRb(\text{dis}(\mathbb{P})))\tau^{(L)}_1}]^{n/2} \leq \mathds{E}^{\overline{Q}^z}_0[e^{\gamma_0\CYRk^L\tau^{(L)}_1}]^{4(|\theta|+\CYRb(\text{dis}(\mathbb{P})))n/\gamma_0\CYRk^L}\leq \eta^{4(|\theta|+\CYRb(\text{dis}(\mathbb{P})))n/\gamma_0\CYRk^L},
        \end{equation}
        {\color{red}provided that} $|\theta|\vee\text{dis}(\mathbb{P})<\gamma_0\CYRk^L$ {\color{red}and} $L\geq L_0$. On other hands, when $n>\zeta^{(z)}$, {\color{red}the trajectories must have coincided at some step, implying} $\CYRk^L Z_i-z=\CYRk^L(\widetilde{Z}_j-z)$ for some $1\leq i\leq\tau^{(L)}_n$, $1\leq j\leq\widetilde{\tau}^{(L)}_n$ if $z\neq0$. In particular, {\color{red}this forces} $\CYRk^L\max\{\tau^{(L)}_n,\widetilde{\tau}^{(L)}_n \}\geq\abs{z}/2$. {\color{red}Consequently, by Chebyshev's inequality for higher moments, we obtain}
        \begin{equation}\label{eqn: dog2 b}
            \mathds{E}^{\overline{Q}^z}_{0,\nu}[n>\zeta^{(z)}]^{1/4}\leq 2\mathds{E}^{\overline{Q}^z}_0[\CYRk^L\tau^{(L)}_n\geq(\abs{z}\vee1)/2]\leq 2(2/(\abs{z}\vee1))^{d+1/4}\mathds{E}^{\overline{Q}^z}_0[(\CYRk^L\tau^{(L)}_n)^{4d+1}]^{1/4}\leq K_1^\prime n^d (\abs{z}\vee1)^{-(d+1/4)}
        \end{equation}
        for some $K_1^\prime(\kappa,z)>0$. To control the {\color{red}third} term {\color{red}in} (\ref{eqn: dog2}), we define $s_*\coloneqq\inf\{[\CYRk^{-L}k]:\,k\in\mathcal{L}^{(z)},\,k>0\}$, $\beta^{(\zeta)}(m)\coloneqq\inf\{n\geq L^{(z)}_m:\,\langle Z_n,\ell\rangle<m\}$ and $R^{(z)}(m)\coloneqq\sup\{\langle Z_n,\ell\rangle:\,L^{(z)}_m\leq n<\beta^{(z)}_0(m)\}$, together with the quantities $\widetilde{\beta}^{(\zeta)}(m)$, $\widetilde{R}^{(z)}(m)$ for $(\widetilde{Z}_n)_{n\geq0}$, $\forall~m\geq0$. We also {\color{red}construct} $(s_j)_{j\geq1}$ recursively by $s_1=1$, $s_{j+1}=R^{(z)}(j)\wedge\widetilde{R}^{(z)}(j)+1$ if $s_j<\infty$, and $s_{j+1}=\infty$ if $s_j=\infty$, $\forall~j\geq1$. {\color{red}Let} $R^{(z)}_s\coloneqq R^{(z)}(0)\wedge\widetilde{R}^{(z)}(0)+1$. {\color{red}By Hölder's inequality,}
        \[
            E^{Q^z}_{0,\nu}[ e^{\gamma\CYRk^L R^{(z)}_s},\,R^{(z)}_s<\infty ]\leq E^{Q^z}_{0,\nu}[ e^{\gamma_0\CYRk^L R^{(z)}_s},\,R^{(z)}_s<\infty ]^{\gamma/\gamma_0} Q^z_{0,\nu}(R^{(z)}_s<\infty)^{1-\gamma/\gamma_0},\qquad\forall~\nu\in\mathbb{V}_d.
        \]
        Here we have $s_*=\sup\{R^{(z)}(j):\,R^{(\lambda)}(j)<\infty\}$. Since $R^{(z)}(0)+1\leq\tau^{(L)}_1$ on $\{\beta^{(\zeta)}_0<\infty\}$, we have
        \[
            E^{Q^z}_{0,\nu}[e^{\gamma_0\CYRk^L R^{(z)}_s},\,R^{(z)}_s<\infty]\leq 2E^{Q^z}_{0}[e^{\gamma_0\CYRk^L (R^{(z)}(0)+1)},\,\beta^{(\zeta)}_0 <\infty]\leq 2E^{Q^z}_0[e^{\gamma_0\CYRk^L\tau^{(L)}_1}] \leq 2\eta,
        \]
        for any small $\gamma$. Since $Q_{0,\nu}(R^{(z)}_s<\infty)\leq Q_{0,\nu}(\beta^{(\zeta)}_0<\infty~\text{or}~\widetilde{\beta}^{(\zeta)}_0<\infty)\leq 1-\eta^{-2}$, {\color{red}we can choose sufficiently} small  $\gamma^\prime(\kappa,z)>0$ {\color{red}such that} $\sup_{\nu\in\mathbb{V}_d}E^{Q^z}_{0,\nu}[e^{\gamma_0\CYRk^L R^{(z)}_s},\,R^{(z)}_s<\infty]\leq 1-\eta^{-2}/2$. Hence,
        \[
            E^{Q^z}_{0,\nu}[e^{\gamma^\prime\CYRk^L s_*}]\leq\sum_{j=1}^\infty E^{Q^z}_{0,\nu}[e^{\gamma^\prime\CYRk^L R^{(z)}(j)},\,R^{(z)}(j)<\infty] \leq  \sum_{j=1}^\infty E^{Q^z}_{0,\nu}[e^{\gamma^\prime\CYRk^L R^{(z)}(j-1)},\,R^{(z)}(j-1)<\infty] (1-\eta^{-2}/2)\leq 2\eta^2 e^{\gamma^\prime},
        \]
        for any $z\in\mathbb{V}_d$. Therefore, we derive $\mathds{E}^{\overline{Q}^z}_{0,\nu}(n=\infty\{k\in\mathcal{L}^{(z)}:\,k>0\})^{1/4}\leq \mathds{E}^{\overline{Q}^z}_{0,\nu}(s_*\geq n)^{1/4}\leq 2\eta^{1/2}e^{\gamma^\prime/4} e^{-\gamma^\prime n/4}$, {\color{red}the desired exponential tail decay}. Combining with (\ref{eqn: dog2 a}) and (\ref{eqn: dog2 b}) to an estimate of $D_{\nu,n}(\theta)$, we {\color{red}have} verified the assertion.
    \end{proof}

\section{Technical lemmas: II}

\begin{lemma}\label{lem: B0}
    \normalfont
    When $\mathbb{P}$ {\color{red}satisfies} $\textbf{(SMF)}_{C,g}$, {\color{red}define} $\lambda_N(\theta)\coloneqq\frac{1}{N}\log E_0 [e^{\langle\theta,S_N\rangle}I_{\mathbb{B}_N}]$,  $\forall~\theta\in\mathbb{R}^d$, $N\geq1$. Then {\color{red}the limiting free energy} $\Lambda_a:\theta\mapsto\lim_N\lambda_N(\theta)$ is {\color{red}continuously} differentiable on $\mathbb{R}^{d-1}$ and $\partial_\theta\lambda_N(\theta)\rightrightarrows\nabla\Lambda_a(\theta)$ uniformly on compacts.
\end{lemma}
\begin{proof}
    Clearly each $\lambda_N(\vdot)\in C^\infty(\mathbb{R}^d)$ and {\color{red}by Lemma \ref{lem: boundary LDP general existence}} they converge pointwise to $\Lambda_a(\theta)$ as $N\to\infty$. It suffices to check $\partial_\theta\lambda_N(\theta)\rightrightarrows\nabla\Lambda_a(\theta,\lambda)$ uniformly on {\color{red}every compact} $K\subseteq\mathbb{R}^{d-1}$, and then by \cite[Theorem 7.17]{Rudin} we can interchange limit and differentials. To this end, we show $\partial_\theta\lambda_N(\theta)$ are equi-continuous on $K$. For any $\theta\in \mathbb{R}^{d-1}$,
    \[
        \partial_\theta\lambda_N(\theta) = \frac{1}{N}\frac{E_0[S_NI_N]}{E_0[e^{\langle\theta,S_N\rangle}I_{\mathbb{B}_N}]},\qquad \partial^2_\theta\lambda_N(\theta) = \frac{1}{N}\bigg( \frac{E_0[S_NS_N^TI_N]}{E_0[e^{\langle\theta,S_N\rangle}I_{\mathbb{B}_N}]} - \frac{E_0[S_NI_N]E_0[S_N^TI_N]}{E_0[e^{\langle\theta,S_N\rangle}I_{\mathbb{B}_N}]^2} \bigg),\qquad\forall~N\geq1.
    \]
    Here we adopt the following notation, $I_N\coloneqq e^{\langle\theta,S_N\rangle}I_{\mathbb{B}_N}$, $I_{i,j}\coloneqq e^{\langle\theta,\Delta_i(S)+\Delta_j(S)\rangle}I_{\mathbb{B}_N}$ and $I^{i,j}_N\coloneqq e^{\sum_{k\neq i,j}\langle\theta,\Delta_k(S)\rangle}I_{\mathbb{B}_N}$, $\forall~i,j\leq N$. Therefore we can express $\partial^2_\theta\lambda_N$ {\color{red}using the product measures} as
    \begin{equation*}\begin{aligned}
        \partial^2_\theta\lambda_N(\theta) &= \sum_{i,j=1}^N \frac{1}{N} \bigg( \frac{E_0[\Delta_i(S)\Delta_j(S)^TI_N]E_0[I_N]}{E_0[e^{\langle\theta,S_N\rangle}I_{\mathbb{B}_N}]^2} - \frac{E_0[\Delta_i(S)I_N]E_0[\Delta_j(S)^TI_N]}{E_0[e^{\langle\theta,S_N\rangle}I_{\mathbb{B}_N}]^2} \bigg)\\
        &= \sum_{i,j=1}^N \mu_{i,j}^2 \frac{1}{N} \frac{E_0\otimes\widetilde{E}_0[\Delta_i(S)\Delta_j(S)^TI_N\tilde{I}_N - \Delta_i(S)I_N\Delta_j(\widetilde{S})^T\tilde{I}_N]}{E_0[I_{i,j}]^2 E_0[I^{i,j}_N]^2 },\qquad\forall~\theta\in\mathbb{R}^{d-1}.
    \end{aligned}\end{equation*}
    Here $(\widetilde{S}_n)_{n\geq0}$ is an independent copy of $(S_n)_{n\geq0}$ under law $\widetilde{P}_0$ and $\tilde{I}_N$ is defined with $(\widetilde{S}_n)_{n\geq0}$ similar as $I_N$ with $(S_n)_{n\geq0}$. The factor $\mu_{i,j}$ comes from a correlation in the denominator and has the estimate $1-\text{dis}(\mathbb{P})\leq \mu_{i,j}^{1/2}\leq1+\text{dis}(\mathbb{P})$. Hence,
    \[
        \partial^2_\theta\lambda_N(\theta) = \sum_{i,j=1}^N\mu_{i,j}^2\nu_{i,j}^2\frac{1}{N} \frac{E_0\otimes \widetilde{E}_0[\Delta_i(S)\Delta_j(S)^TI_{i,j}\tilde{I}_{i,j} - \Delta_i(S)I_{i,j}\Delta_j(\widetilde{S})^T\tilde{I}_{i,j} ] E_0[I_N^{i,j}]^2 }{ E_0[I_{i,j}]^2E_0[I^{i,j}_N]^2 },\qquad\forall~\theta\in\mathbb{R}^{d-1}.
    \]
    Here $\nu_{i,j}$ comes from a correlation in the denominator and has the estimate $1-\text{dis}(\mathbb{P})\leq \nu_{i,j}^{1/2}\leq1+\text{dis}(\mathbb{P})$, $\forall~i,j\leq N$. Let $P^{i,j}_{\theta,N}$ be the law absolutely continuous with respect to $P_0$ with Radon--Nikodým derivative $dP^{i,j}_{\theta,N}/dP_0 = I_{i,j}/E_0[I_{i,j}] $. Then,
    \[
        \partial^2_\theta\lambda_N(\theta)\sum_{i,j=1}^N \mu_{i,j}^2\nu_{i,j}^2\frac{1}{N} \big( E^{i,j}_{\theta,N}[ \Delta_i(S)\Delta_j(S)^T ] - E^{i,j}_{\theta,N}[\Delta_i(S)] E^{i,j}_{\theta,N}[\Delta_j(S)^T] \big),\qquad\forall~\theta\in\mathbb{R}^{d-1}.
    \]
    Observe that $E^{i,j}_{\theta,N}[ \Delta_i(S)\Delta_j(S)^T ] = E^{i,j}_{\theta,N}[\Delta_i(S)] E^{i,j}_{\theta,N}[\Delta_j(S)^T]$ whenever $\abs{i-j}\geq L$, by the nature of $\textbf{(SMF)}_{C,g}$. Thus,
    \[
        |\partial^2_\theta\lambda_N(\theta)| \leq (1+\text{dis}(\mathbb{P}))^4\frac{1}{N}\sum_{i,j=1}^{N\wedge L} \big( E^{i,j}_{\theta,N}[ \Delta_i(S)\Delta_j(S)^T ] - E^{i,j}_{\theta,N}[\Delta_i(S)] E^{i,j}_{\theta,N}[\Delta_j(S)^T] \big) \leq (1+\text{dis}(\mathbb{P}))^4\frac{1}{N}\sum_{i,j=1}^{N\wedge L} 2Id < \infty,
    \]
    uniformly for all $\theta\in K$ , verifying the equi-continuity of $\partial_\theta\lambda_N(\theta)$ on $K$. Now, $\forall~\epsilon>0$, $\exists~\delta>0$ s.t. $|\lambda^\prime_N(\theta)-\lambda^\prime_N(\theta^\prime)|$ for any $\theta^\prime\in B_\delta(\theta)$, $\forall~N\geq1$. Now $\forall~\theta\in K$, $\exists~N_\theta\geq1$ s.t. $|\lambda^\prime_m(\theta)-\lambda^\prime_n(\theta)|<\epsilon/3$, $\forall~m,n\geq N_\theta$. Thus,
    \[
        |\lambda^\prime_m(\theta^\prime)-\lambda^\prime_n(\theta^\prime)|\leq |\lambda^\prime_m(\theta^\prime)-\lambda^\prime_m(\theta)| + |\lambda^\prime_m(\theta)-\lambda^\prime_n(\theta)| + |\lambda^\prime_n(\theta^\prime)-\lambda^\prime_n(\theta)| < \epsilon,\qquad\forall~\theta^\prime\in B_\delta(\theta),\quad m,n\geq N_\theta.
    \]
    Choosing finitely many $\{\theta_1,\ldots,\theta_k\}$ to cover $K$ by $(B_\delta(\theta_i))_{i=1,\ldots,k}$, we select $N_U\coloneqq\max\{N_{\theta_1},\ldots,N_{\theta_k}\}$ and then $ |\lambda^\prime_m(\theta)-\lambda^\prime_n(\theta)|<\epsilon$, $\forall~\theta\in K$ and $m,n\geq N_U$. And the assertion is then verified.

\end{proof}

\begin{lemma}\label{lem: B1}
    \normalfont
    When $\mathbb{P}$ {\color{red}satisfies} $\textbf{(SMF)}_{C,g}$, for any $c>0$, there exists $\epsilon_1=\epsilon_1(\kappa,z,c)>0$ such that
    \begin{equation}\label{eqn: dogdog1}
        \sup_{|\theta|\vee\text{dis}(\mathbb{P})<\epsilon_1\CYRk^L}|\nabla\overline{\Lambda}^a_z(\theta)-\nabla\overline{\Lambda}^a_z(0,)|<c.
    \end{equation}
\end{lemma}
\begin{proof}
    {\color{red}Recall} that $\mathds{E}^{\overline{Q}^z}_0[\CYRk^L\tau^{(L)}_1]\geq c_\kappa$ for some $c_\kappa(z)>0$, it suffices to check {\color{red}for any} $\delta>0$, {\color{red}there exists} $\epsilon_1^\prime=\epsilon_1^\prime(\kappa,z,\delta)>0$ such that $ | \mathds{E}^{\overline{Q}^z}_0[ \CYRk^L Z_{\tau^{(L)}_1}\mathbb{E}\prod_{j=1}^{\tau^{(L)}_1}\CYRu_j(e^{\langle  \theta, Z(\tau^{(L)}_1)\rangle- \overline{\Lambda}^a_z(\theta) \tau^{(L)}_1} - e^{- \overline{\Lambda}^a_z(0) \tau^{(L)}_1} ) ] |<\delta$, whenever $|\theta|\vee\text{dis}(\mathbb{P})<\epsilon_1^\prime\CYRk^L$, as well as in the same regime, $| \mathds{E}^{\overline{Q}^z}_0 [  \CYRk^L \tau^{(L)}_1 \mathbb{E}\prod_{j=1}^{\tau^{(L)}_1}\CYRu_j (e^{\langle  \theta,  Z(\tau^{(L)}_1)\rangle- \overline{\Lambda}^a_z(\theta) \tau^{(L)}_1} - e^{- \overline{\Lambda}^a_z(\theta) \tau^{(L)}_1} ) ] |<\delta$. The {\color{red}key step relies on an} estimation {\color{red}of}  $|\langle \theta, Z(\tau^{(L)}_1)\rangle- \overline{\Lambda}^a_z(\theta) \tau^{(L)}_1+ \overline{\Lambda}^a_z(0) \tau^{(L)}_1 |\leq 2(|\theta|+\CYRb(\text{dis}(\mathbb{P}))) \CYRk^L\tau^{(L)}_1$ as well as $ e^{-\overline{\Lambda}^z_a(0) \tau^{(L)}_1} \mathbb{E}\prod_{j=1}^{\tau^{(L)}_1}\CYRu_j \leq e^{\CYRb(\text{dis}(\mathbb{P})) \CYRk^L \tau^{(L)}_1}$, and the rest follows {\color{red}by the exactly same arguments as in} Proposition \ref{prop: distance between lambda z and differential}.
\end{proof}

\begin{lemma}\label{lem: B2}
    \normalfont
    When $\mathbb{P}$ {\color{red}satisfies} $\textbf{(SMF)}_{C,g}$, $\forall~c>0$, $\exists~\epsilon_2=\epsilon_2(\kappa,z,c)>0$ such that if $\text{dis}(\mathbb{P})<\epsilon_2\CYRk^L$, then
        \begin{equation}\label{eqn: dogdog2}
            \normx{H^a_z(0)-H^*_z(0)}<c.
        \end{equation}
\end{lemma}
\begin{proof}
    In view of $ \normx{\mathds{E}^{\overline{Q}^z}_0[\Gamma( z)]}^2\leq \mathds{E}^{\overline{Q}^z}_0[| Z(\tau^{(L)}_1)- z\tau^{(L)}_1|^2]\leq 2\abs{z}^2\mathds{E}^{\overline{Q}^z}_0[(\tau^{(L)}_1)^2] $ and  $\mathds{E}^{\overline{Q}^z}_0[\tau^{(L)}_1]\geq c_\kappa/\CYRk^L$, it suffices to show the numerators of the two matrices are close, i.e.~that $\forall~\delta>0$, $\exists~\epsilon_2^\prime=\epsilon_2^\prime(\kappa,z,\delta)>0$ s.t. $\normx{\mathds{E}^{\overline{Q}^z}_0[\Gamma(\nabla\overline{\Lambda}^a_z(0))\varphi(0)]-\mathds{E}^{\overline{Q}^z}_0[\Gamma(z,)]}<\delta$ whenever $\text{dis}(\mathbb{P})<\epsilon_2^\prime\CYRk^L$. Indeed, the LHS is bounded by $\mathds{E}^{\overline{Q}^z}_0[\normx{\Gamma(\nabla\overline{\Lambda}^a_z(0))-\Gamma(z)} \abs{\varphi(0)}]+\mathds{E}^{\overline{Q}^z}_0[\normx{\Gamma(z)}\abs{\varphi(0,)-1}]$. Notice further that 
    \[
        \abs{\varphi(0)-1}\leq \CYRb(\text{dis}(\mathbb{P})) \tau^{(L)}_1 e^{\CYRb(\text{dis}(\mathbb{P})) \tau^{(L)}_1},\qquad \normx{\Gamma(\nabla\overline{\Lambda}^a_z(0))-\Gamma(z)}\leq 5|\nabla\overline{\Lambda}^a_z(0)-z|(\tau^{(L)}_1)^2
    \]
    and that $\normx{\Gamma(z)}\leq | Z(\tau^{(L)}_1)- z\tau^{(L)}_1|^2\leq (1+ \abs{z})^2(\tau^{(L)}_1)^2$. The main assertion then follows from the above estimates.
\end{proof}

\begin{lemma}\label{lem: B3}
    \normalfont
    There exists $K=K(\kappa,z)>0$ such that $\sup_{\mathbb{P}} \normx{H^*_z(0)^{-1}}\leq K$, where the supremum is taken over all $\kappa$-uniformly elliptic {\color{red}environment} $\mathbb{P}$ {\color{red}satisfying the condition} $\textbf{(SMF)}_{C,g}$.
\end{lemma}
\begin{proof}
    Let $\mathcal{M}(\mathbb{V})$ {\color{red}denote the space of} all probability vectors on $\mathbb{V}$, and let $\mathcal{M}^\kappa(\mathbb{V})\coloneqq\{p\in\mathcal{M}(\mathbb{V}):\,p(e)\geq\kappa,\,\forall~e\in\mathbb{V}\}$ {\color{red}be the subset of} $\kappa$-uniformly elliptic probability vectors. Following {\color{red}similar continuity arguments as in} \cite[Lemma 4.8]{Bazaes/Mukherjee/Ramirez/Sagliett0}, {\color{red}the mappings} $\alpha\in\mathcal{M}(\mathbb{V})\mapsto \mathds{E}^{\overline{Q}^z}_0[\tau^{(L)}_1]$ and $\alpha\in\mathcal{M}(\mathbb{V})\mapsto \mathds{E}^{\overline{Q}^z}_0[(Z(\tau^{(L)}_1)-z\tau^{(L)}_1)^T(Z(\tau^{(L)}_1)-z\tau^{(L)}_1)]$ are {\color{red}both} continuous {\color{red}functions on $\mathcal{M}(\mathbb{V})$.} Furthermore, the mapping $\alpha\mapsto\normx{H^*_z(0)^{-1}}$ as a composite of continuous mappings $A\mapsto A^{-1}$ and $A\mapsto\normx{A}$ is a continuous mapping on $\mathcal{M}(\mathbb{V})$ as well. In particular, this mapping achieves maximum in the compact subspace $\mathcal{M}^\kappa(\mathbb{V})$ where $\sup_{\mathbb{P}}\normx{H^*_z(0)^{-1}}= \sup_{\alpha\in\mathcal{M}^\kappa(\mathbb{V})} \normx{H^*_z(0)^{-1}}<\infty$. Hence, the assertion follows.
\end{proof}

\begin{lemma}\label{lem: D0}
    \normalfont
    When $\mathbb{P}$ {\color{red}satisfies} $\textbf{(SMF)}_{C,g}$, $\exists~\mathfrak{r}>0$ s.t. $\varlimsup_{k\to\infty}k^{-1}\log \overline{Q}^z_0(Z_{T^{\ell,z}_{k}} \not\in \partial^+ B^\ell_{k,\mathfrak{r}k})\leq-\langle z,\ell\rangle$, where $T^{\ell,z}_{k}$ and $B^\ell_{k,\mathfrak{r}k}$ are defined in the course of Proposition \ref{prop: mixing exponential of tau < infinity}.
\end{lemma}
\begin{proof}
    Let $T^k\coloneqq\inf\{n\geq0:\,\abs{\langle Z_n,\ell\rangle}\geq k\}$, $\forall~k\in\mathbb{N}$. Then by the standard gambler's ruin problem  \cite[Section 4.2]{Lalley},
    \[
        \overline{Q}^z_0(\langle Z_{T^k},\ell\rangle=k) = \frac{1-(q/p)^k}{1-(q/p)^{2k}},\qquad\text{where}\quad 2q\coloneqq -\langle z,\ell\rangle + \sqrt{\abs{\langle z,\ell\rangle}^2+1},\quad p\coloneqq1-q.
    \]
    Using Taylor's expansion and letting $k\to\infty$, we get $\lim_{k\to\infty}k^{-1}\log \overline{Q}^z_0(\langle Z_{T^k},\ell\rangle =k)=-2 \langle z,\ell\rangle$. Invoking \cite[Lemma 2.2 (iii)]{Guerra Aguilar}, we can find some sufficiently large $\mathfrak{r}=\mathfrak{r}(\kappa,z)>0$ which verifies our assertion.
\end{proof}


\bibliographystyle{plain}
\begin{spacing}{1}

\end{spacing}

\end{document}